\documentclass{amsart}

\usepackage{fullpage}
\usepackage{amsfonts}
\usepackage{amsthm}
\usepackage{amsmath}
\usepackage{amssymb}
\usepackage{mathtools}
\usepackage{bbm}
\usepackage{bussproofs}
\usepackage{mathrsfs}
\usepackage{tikz, tikz-cd}
\usepackage[all,2cell]{xy}
\usepackage{comment}
\usepackage{epigraph}
\usepackage{url}
\usepackage{hyperref}
\usepackage{ stmaryrd } 
\usepackage[textsize=scriptsize]{todonotes}

\tikzset{%
	symbol/.style={%
		draw=none,
		every to/.append style={%
			edge node={node [sloped, allow upside down, auto=false]{$#1$}}}
	}
}

\usetikzlibrary{matrix,arrows}

\newtheorem{Theorem}{Theorem}

\newtheorem{prop}[Theorem]{Proposition}
\newtheorem{lem}[Theorem]{Lemma}
\newtheorem{cor}[Theorem]{Corollary}

\theoremstyle{definition}
\newtheorem{example}[Theorem]{Example}
\newtheorem{dfn}[Theorem]{Definition}
\newtheorem{rmk}[Theorem]{Remark}
\newtheorem{assumption}[Theorem]{Assumption}

\DeclareMathOperator{\Cscr}{\mathscr{C}}
\DeclareMathOperator{\Dscr}{\mathscr{D}}
\DeclareMathOperator{\Escr}{\mathscr{E}}
\DeclareMathOperator{\Fscr}{\mathscr{F}}

\DeclareMathOperator{\Iscr}{\mathscr{I}}

\DeclareMathOperator{\Kscr}{\mathscr{K}}
\DeclareMathOperator{\Lscr}{\mathscr{L}}

\DeclareMathOperator{\Uscr}{\mathscr{U}}

\DeclareMathOperator{\Dbb}{\mathbb{D}}

\DeclareMathOperator{\Ibb}{\mathbb{I}}

\DeclareMathOperator{\Lbb}{\mathbb{L}}

\DeclareMathOperator{\Nbb}{\mathbb{N}}

\DeclareMathOperator{\Pbb}{\mathbb{P}}

\DeclareMathOperator{\Rbb}{\mathbb{R}}
\DeclareMathOperator{\Sbb}{\mathbb{S}}
\DeclareMathOperator{\Tbb}{\mathbb{T}}
\DeclareMathOperator{\Ubb}{\mathbb{U}}
\DeclareMathOperator{\Vbb}{\mathbb{V}}

\DeclareMathOperator{\Bcal}{\mathcal{B}}
\DeclareMathOperator{\Ccal}{\mathcal{C}}
\DeclareMathOperator{\Dcal}{\mathcal{D}}

\DeclareMathOperator{\Kcal}{\mathcal{K}}

\DeclareMathOperator{\Mcal}{\mathcal{M}}

\DeclareMathOperator{\Ocal}{\mathcal{O}}

\DeclareMathOperator{\Scal}{\mathcal{S}}

\DeclareMathOperator{\Ucal}{\mathcal{U}}

\DeclareMathOperator{\Cfrak}{\mathfrak{C}}

\DeclareMathOperator{\Dfrak}{\mathfrak{D}}

\DeclareMathOperator{\mfrak}{\mathfrak{m}}

\DeclareMathOperator{\Ufrak}{\mathfrak{U}}

\DeclareMathOperator{\VVec}{\mathbf{Vec}}
\DeclareMathOperator{\QCoh}{\mathbf{QCoh}}
\DeclareMathOperator{\Shv}{\mathbf{Shv}}

\DeclareMathOperator{\id}{id}

\DeclareMathOperator{\Set}{\mathbf{Set}}

\DeclareMathOperator{\R}{\mathbb{R}}

\DeclareMathOperator{\Z}{\mathbb{Z}}
\DeclareMathOperator{\N}{\mathbb{N}}

\DeclareMathOperator{\KAlg}{\mathnormal{K}-\mathbf{Alg}}

\DeclareMathOperator{\RMod}{\mathbf{R-Mod}}

\DeclareMathOperator{\Cring}{\mathbf{Cring}}

\DeclareMathOperator{\Ab}{\mathbf{Ab}}

\DeclareMathOperator{\Q}{\mathbb{Q}}

\DeclareMathOperator{\Spec}{Spec}
\DeclareMathOperator{\Sch}{\mathbf{Sch}}

\DeclareMathOperator{\Free}{\mathbf{Free}}

\DeclareMathOperator{\Forget}{Forget}
\DeclareMathOperator{\Cat}{\mathbf{Cat}}
\DeclareMathOperator{\Coh}{\mathbf{Coh}}

\DeclareMathOperator{\Top}{\mathbf{Top}}

\DeclareMathOperator{\Sym}{Sym}

\DeclareMathOperator{\LocRingSpac}{\mathbf{LRS}}

\DeclareMathOperator{\codom}{Codom}

\DeclareMathOperator{\op}{op}

\DeclareMathOperator{\GL}{GL}

\DeclareMathOperator{\Per}{\mathbf{Per}}

\DeclareMathOperator{\Sf}{\mathbf{Sf}}

\DeclareMathOperator{\dom}{dom}
\DeclareMathOperator{\fCat}{\mathfrak{Cat}}
\DeclareMathOperator{\fCAT}{\mathfrak{CAT}}
\DeclareMathOperator{\SfResl}{\mathbf{SfResl}}
\DeclareMathOperator{\quo}{\mathsf{quot}}

\DeclareMathOperator{\Var}{\mathbf{Var}}

\DeclareMathOperator{\Open}{\mathbf{Open}}

\DeclareMathOperator{\fTan}{\mathfrak{Tan}}

\let\epsilon\varepsilon

\newcommand{\Dbeq}[2]{\mathnormal{D}^{\mathnormal{b}}_{{#1}}({#2})}
\newcommand{\Dbeqstack}[1]{\mathnormal{D}^{\mathnormal{b}}_{\text{eq}}(\underline{[\mathnormal{G} \backslash \mathnormal{X}]}_{\bullet};\overline{\Q}_{\ell})}

\newcommand{\DbQl}[1]{\mathnormal{D}_{\mathnormal{c}}^{\mathnormal{b}}({#1};\overline{\Q}_{\ell})}
\newcommand{\DbeqQl}[1]{\mathnormal{D}_{\mathnormal{G}}^{\mathnormal{b}}({#1};\overline{\Q}_{\ell})}

\DeclareMathOperator{\XGamma}{\quot{\mathnormal{X}}{\Gamma}}
\DeclareMathOperator{\YGamma}{\quot{\mathnormal{Y}}{\Gamma}}
\DeclareMathOperator{\XGammap}{\quot{\mathnormal{X}}{\Gamma^{\prime}}}
\DeclareMathOperator{\YGammap}{\quot{\mathnormal{Y}}{\Gamma^{\prime}}}

\DeclareMathOperator{\YGammapp}{\quot{\mathnormal{Y}}{\Gamma^{\prime\prime}}}
\DeclareMathOperator{\AGamma}{\quot{\mathnormal{A}}{\Gamma}}
\DeclareMathOperator{\AGammap}{\quot{\mathnormal{A}}{\Gamma^{\prime}}}
\DeclareMathOperator{\AGammapp}{\quot{\mathnormal{A}}{\Gamma^{\prime\prime}}}
\DeclareMathOperator{\BGamma}{\quot{\mathnormal{B}}{\Gamma}}
\DeclareMathOperator{\BGammap}{\quot{\mathnormal{B}}{\Gamma^{\prime}}}

\DeclareMathOperator{\const}{\mathsf{const}}

\DeclareMathOperator{\of}{\overline{\mathnormal{f}}}

\DeclareMathOperator{\rhoGamma}{\quot{\rho}{\Gamma}}
\DeclareMathOperator{\SMan}{\mathbf{SMan}}
\DeclareMathOperator{\FMan}{\mathbf{FMan}}

\let\epsilon\varepsilon

\newcommand{\pullbackcorner}[1][dr]{\save*!/#1-1.2pc/#1:(-1,1)@^{|-}\restore}
\newcommand{\quot}[2]{\,_{#2}{#1}}

\newcommand{\Zar}[1]{\mathbf{Zar}({#1})}
\newcommand{\PC}{\mathbf{PC}}

\newcommand{\Bicat}{\mathsf{Bicat}}
\newcommand{\Kah}[2]{\Omega_{{#1}/{#2}}^1}

\UseTwocells

\numberwithin{Theorem}{section}

\title{Pseudolimits for Tangent Categories with Applications to Equivariant Algebraic and Differential Geometry}
\author{Dorette Pronk}
\address{Dalhousie University}
\email{dorette.pronk@dal.ca}

\author{Geoff Vooys}
\address{University of Calgary}
\email{gmvooys@ucalgary.ca}
\date{\today}

\begin{document}

\begin{abstract}
In this paper we show that if $\mathscr{C}$ is a category and if $F\colon\mathscr{C}^{\operatorname{op}} \to \mathfrak{Cat}$ is a pseudofunctor such that for each object $X$ of $\mathscr{C}$ the category $F(X)$ is a tangent category and for each morphism $f$ of $\mathscr{C}$ the functor $F(f)$ is part of a strong tangent morphism $(F(f),\!\quot{\alpha}{f})$ and that furthermore the natural transformations $\!\quot{\alpha}{f}$ vary pseudonaturally in $\mathscr{C}^{\operatorname{op}}$, then there is a tangent structure on the pseudolimit $\mathbf{PC}(F)$ which is induced by the tangent structures on the categories $F(X)$ together with how they vary through the functors $F(f)$.
We use this observation to show that the forgetful $2$-functor $\operatorname{Forget}:\mathfrak{Tan} \to \mathfrak{Cat}$ creates and preserves pseudolimits indexed by $1$-categories. As an application, this allows us to describe how equivariant descent interacts with the tangent structures on the category of smooth (real) manifolds and on various categories of (algebraic) varieties over a field.
\end{abstract}


\maketitle
\tableofcontents

\section{Introduction}
Tangent categories describe a categorification of differential geometry by focusing on and abstracting the properties of the tangent bundle functor $T\colon \mathbf{SMan} \to \mathbf{SMan}$, where $\mathbf{SMan}$ is the category of smooth (real) manifolds, and formalizing notions of tangent bundle and differential bundle that can be applied in further abstract categorical settings.
This abstract formalism provides an algebraic description of what it means to be smooth and brings monadic techniques to the study of manifolds and differential-geometric objects. 
The notion of tangent categories was originally developed in \cite{Rosicky} and then rediscovered and generalized in \cite{GeoffRobinDiffStruct}. 
Tangent categories are now an area of study in category theory in their own right and they are used to study problems in differential geometry, algebraic geometry, linear logic, the $\lambda$-calculus, machine learning, synthetic differential geometry, programming theory, and more. 
However, while tangent categories do a good job of creating an algebraic/equational formulation of what it means for a (real) manifold to be smooth, this formalism has not yet been extended to  model interactions with group actions and local symmetries. 
More generally, the ways in which tangent structures interact with equivariant descent still need to be explored; this paper provides a first step in that direction. A central ingredient towards studying these interactions lies in understanding pseudolimits of diagrams in the $2$-category $\fTan$ of tangent categories.

Pseudolimits are a higher-categorical weakening of the notion of a limit in traditional category theory where the triangles that define the cones involving the projections are only required to commute up to coherent isomorphisms. Pseudolimits provide important tools in higher categorical geometry and topology and in particular in descent theory, as they provide a structure with which to track the ways in which you recognize objects to be ``sufficiently equivalent'' to each other. In fact, the categories of equivariant sheaves and the equivariant derived categories which appear in algebraic geometry, algebraic topology, differential geometry, and representation theory (when one wants to know about equivariant cohomology) can all be viewed as pseudolimits. This perspective is explored at great length in \cite{GeneralGeoffThesis} and used to study both equivariant algebraic geometry and equivariant algebraic topology. 
In this paper we will explore how pseudolimits facilitate equivariance for tangent structures in categorical differential geometry.

Because of their importance as a technical tool in studying equivariant descent and tangent theory simultaneously, the first half of the paper focuses on the nature of when the pseudolimits of pseudofunctors $F\colon \Cscr^{\op} \to \fCat$ (for $\Cscr$ a $1$-category) are tangent categories based on the structure of the fibre categories $F(X)$ for objects $X$ of $\Cscr$ and based on the nature of the transition functors $F(f)$ for morphisms $f$ of $\Cscr$. This then leads us to study how the hom-$2$-category $\Bicat(\Cscr^{\op},\fCat)$ can be used to probe the $2$-category $\fTan_{\operatorname{strong}}$ of tangent categories equipped with strong tangent morphisms for certain pseudolimits.
We then show that, under the same mild technical assumptions as just described, the pseudolimit in $\fCat$ of a pseudofunctor $F\colon\Cscr^{\op} \to \fTan_{\operatorname{strong}}$, when equipped with its pseudolimit tangent structure (defined and proved to exist below), is actually a pseudolimit of the corresponding diagram in $\fTan_{\operatorname{strong}}$ (in other words, the forgetful $2$-functor $\Forget\colon\fTan_{\operatorname{strong}} \to \fCat$ preserves and reflects pseudolimits of pseudofunctors $F$ whose domain is a $1$-category). Because these calculations show the structure of how tangent structures must behave with respect to descent, we then use our study of pseudolimits and tangent categories to study equivariant tangent structures in algebraic and differential geometry.

The study of equivariance (as the broad interpretation of how functions and maps may interact, record, and witness group actions) and equivariant objects in mathematics is deep and varied. These objects arise in the study of representation theory, algebraic geometry, algebraic topology, homotopy theory, statistics, number theory, analysis, stack theory, descent theory, and everywhere and anywhere in between. Of particular interest in equivariant mathematics are orbifolds and geometric representation theory. In both areas of research we want to work with generalized quotients because real quotients do not have the correct structure (for instance, they may not be smooth) and lose too much information. The notion of orbifold atlas handles the problem by keeping track of the local (finite) isotropy groups with their actions on charts and the ways these actions are related to each other, while the Langlands Programme uses descent data taken in quotients of actions of algebraic groups on sufficiently nice varieties in order to study representations of groups in geometric and ring-theoretic ways. 
Orbifolds may also be represented by a particular type of smooth groupoids; in this case the group actions are encoded in the groupoid structure. We see that in all these cases there are natural notions of tangents, but describing these tangent structures formally requires studying how the equivariant descent and tangent theories interact.

In this paper we will show how to make sense of having a descent-equivariant tangent structure on varieties over fields and on smooth manifolds with an action of a compact Lie group.
This gives an idea of how to get tangent category theory to interact with the language of descent theory, scheme theory, differential geometry, and group representation theory. Applications to orbifolds, however, will be presented in a sequel paper.

Finally, some words regarding the general nature of this paper. In \cite{LeungWeil}, Leung showed that to give a tangent category structure on a category $\Cscr$ is the same as giving the category $\Cscr$ a $\mathsf{Weil}_1$-actegory\footnote{The category $\mathsf{Weil}_1$ is the category of Weil-algebras which is a subcategory of the category of augmented nilpotent $\N$-algebras; cf. \cite{LeungWeil} or \cite{BenThesis} for details.} structure. A more high-level approach to what we are doing in this paper could be taken by constructing a certain $\mathsf{Weil}_1$-actegorical structure on the pseudolimit; however, we do not take this perspective in this paper. Instead, we take a much more explicit and calculation-based approach. Because of the applications to scheme theory, orbifold theory, and representation theory we anticipate it is important to have tools and examples showing exactly how computations and manipulations involving equivariant categories should be executed. Additionally, when seeking how to have descent theory (at least as far as it is recorded by a pseudofunctor) and tangent-theoretic information interact we anticipate that it will be useful to have these computations worked out and presented for future work and future practitioners.

It is also worth noting that Lanfranchi has introduced the notion of what it means to be a tangent object (as well as a morphism of tangent objects and a $2$-morphism between morphisms of tangent objects) internal to a $2$-category in \cite{Marcello}. We use this perspective frequently and seriously in the first half of the paper in order to show how tangent objects in the $2$-category $\Bicat(\Cscr^{\op},\fCat)$ compare with various concepts in \cite{Marcello}, as well as how they can be used to determine the pseudolimits in the $2$-category $\fTan$ of tangent categories which are taken over diagrams $D$ whose domain is a $1$-category.

\subsection{Structure of the Paper}
The structure of the paper is as follows. In Section \ref{Section: Higher Cat Review} we review the basic theory of pseudocones, as they form a central technical backbone of the paper. We also introduce the category $\PC(F)$ of pseudocones of a pseudofunctor $F:\Cscr^{\op} \to \fCat$ in this section. As shown in \cite{GeneralGeoffThesis}, this gives a pseudolimit for $F$ in $\fCat$. We also discuss why the lax version of this construction is more poorly behaved and why this necessitates working with pseudonatural transformations and not just lax or oplax transformations in this paper; in particular, we discuss how the lax limit construction in $\fCat$ does \emph{not} lift to $\fTan$, but the pseudolimit construction does (cf. Theorem \ref{Thm: Pseudolimits in Tancat}). Additionally, we give a detailed study of the limits which appear in the pseudocone categories $\PC(F)$.

In Section \ref{Section: Eq Tan Cats} we review the basic theory of tangent categories as developed by Cockett and Cruttwell before taking the rest of the section to prove the first main results of the paper as expressed in the following two theorems. Both these theorems rely on our first key definition: that of a tangent indexing functor (cf. Definition \ref{Defn: tangent pre-equivariant indexing functor}).

\begin{Theorem}[{cf. Theorem \ref{Thm: Pre-Equivariant Tangent Category}}]
Let $F\colon \Cscr^{\op} \to \fCat$ be a tangent indexing functor. Then the category $\PC(F)$ is a tangent category.
\end{Theorem}
\begin{Theorem}[{cf. Theorem \ref{Thm: Tangent objects in Bicat Hom category are tangent indexing functors}}]
Let $\Cscr$ be a $1$-category and consider the $2$-category $\Bicat(\Cscr^{\op},\fCat)$. Then a tuple $(F,T,p,\operatorname{add},0,\ell,c)$ in $\Bicat(\Cscr^{\op},\fCat)$ is a tangent object in the sense of \cite[Definition 3.2]{Marcello} if and only if $F\colon \Cscr^{\op} \to \fCat$ is a tangent indexing functor.
\end{Theorem}

Note that we give rather detailed explicit proofs of these theorems, as we want to provide the details one needs in order to be able to perform computations with equivariant categories and the descent-equivariant tangent structures (cf. Sections \ref{Section: Review of EQCats} -- \ref{Section: Descent Equivariant Manfiold stuffs}). 

In Section \ref{Section:  Eq Tan Mor} we prove that the chosen pseudolimit tangent categories $\PC(F)$ are suitably functorial. In particular, after recalling what it means to have morphisms between tangent objects and transformations between tangent morphisms in a general $2$-category, we characterize the morphisms and transformations in the $2$-category $\fTan(\Bicat(\Cscr^{\op},\fCat))$. We then use this to prove the next main theorems of the paper which show that the $2$-category $\fTan_{\operatorname{strong}}$ has all pseudolimits indexed by $1$-categorical diagrams and that the forgetful $2$-functor $\fTan_{\operatorname{strong}} \to \fCat$ reflects and preserves these pseudolimits.
\begin{Theorem}[{cf. Theorem \ref{Thm: PC of tangent objects in hom two category takes values in tangent categories}}]
The strict $2$-functor $\PC(-)$ restricts to a strict $2$-functor 
$$
\PC\colon \mathfrak{Tan}\left(\Bicat\left(\Cscr^{\op},\fCat\right)\right) \to \fTan_{\operatorname{strong}}
$$ 
and makes the diagram of $2$-categories
\[
\begin{tikzcd}
\Bicat\left(\Cscr^{\op},\fCat\right) \ar[rr]{}{\PC(-)} & & \fCat \\
\fTan\left(\Bicat\left(\Cscr^{\op},\fCat\right)\right) \ar[u]{}{\Forget} \ar[rr, swap]{}{\PC(-)}  &  &\fTan_{\operatorname{strong}} \ar[u, swap]{}{\Forget}
\end{tikzcd}
\]
commute strictly.
\end{Theorem}
\begin{Theorem}[{cf. Theorem \ref{Thm: Pseudolimits in Tancat}}]
Let $F\colon \Cscr^{\op} \to \fCat$ be a tangent indexing functor. The tangent category $\PC(F)$ equipped with the tangent structure of Theorem \ref{Thm: Pre-Equivariant Tangent Category} is the pseudolimit in $\fTan_{\operatorname{strong}}$ of shape $F$.
\end{Theorem}

In Section \ref{Section: Review of EQCats} we review and recall the foundations towards doing equivariant algebraic geometry and equivariant differential geometry. Note that there are many technical footnotes in this section for the reader who is interested in some of the technical details in equivariant geometry, but these details can be safely skipped for the reader who is willing to go with the adjective-laden flow. Most of what appears here is located in various locations in the literature, but we collect the results centrally here to introduce the reader to how to go about the equivariant geometry yoga and to also illustrate the remarkable similarities between the algebraic and differential cases. We also present twelve different examples of the sorts of ``equivariant categories'' which can appear over/on varieties $X$ and smooth manifolds $M$ in Example \ref{Example: Pre-equivariant pseudofunctors} and which can all have their own tangent structures. However, we only pay explicit attention to a few of these examples. 

In Section \ref{Section: Example of equivariant tangent structure on a variety} we compute an extended example and construct the descent equivariant tangent category of schemes over a fixed variety $X$. To do this we recall the Zariski tangent structure on $\Sch_{/S}$ for an arbitrary scheme $S$. After doing this we prove that for an arbitrary scheme morphism $f\colon S \to T$, $f^{\ast}\colon \Sch_{/T} \to \Sch_{/S}$ is a strong tangent morphism for the Zariski tangent structure and that these tangent morphisms vary pseudofunctorially in $S$ and $T$. Putting these together we show how to define the tangent category of descent equivariant schemes over a variety $X$; cf. Definition \ref{Defn: Equivariant Zariski category} and Theorem \ref{Thm: Equivariant Zariski Tangent Structure} for details. We also conclude the section by showing first that the Zariski tangent category of schemes is compatible with gluing of affine schemes in the sense that the Zariski tangent category of schemes over a base scheme $X$ is equivalent to the pseudolimit of any open affine subcover $\lbrace f_i:U_i \to X \; \left. \right| \; i \in I \rbrace$. We also extend this to the equivariant setting when we know that the morphisms in the given cover of $X$ are themselves equivariant.
\begin{prop}[cf. Proposition \ref{Prop: Scheme gluing equivalence of tan cats}]
Let $X$ be a scheme and let $\Ucal = \lbrace f_i:U_i \to X \; \left. \right| \; i \in I \rbrace$ be an affine open cover of $X$. Then there is an equivalence of tangent categories
\[
\left(\Sch_{/X},\Tbb_{\Zar{X}}\right) \simeq \operatorname{pseudolim}\left(\Sch_{/U_i},\Tbb_{\Zar{U_i}}\right).
\]
\end{prop}
\begin{prop}[cf. Proposition \ref{Prop: Equivariantifying the gluing equiv for schemes}]
Let $\Ucal := \lbrace f_i:U_i \to X \; \left. \right| \; i \in I \rbrace$ be an affine open cover of a $K$-variety $X$ for which each map $f_i$ is $G$-equivariant for a smooth algebraic group $G$. Then there is an equivalence of tangent categories
\[
\left(\left(\Sch_{/X}\right)_G,\Tbb_{X}\right) \simeq \operatorname{pseudolim} \left(\left(\Sch_{/U_i}\right)_G, \Tbb_{U_i}\right).
\]
\end{prop}

In Section \ref{Section: Descent Equivariant Manfiold stuffs} we do a second extended example and show how to build the descent equivariant tangent category of smooth manifolds over a fixed smooth manifold $M$ equipped with a (smooth) action by a Lie group $L$ by defining a tangent indexing functor. To do this we recall the notion of what it means to be an {\'e}tale map in a tangent category and then prove some small results about how {\'e}tale maps in tangent categories interact with pullback functors between slice tangent structures. Afterwards we define our pseudofunctor and prove that it is a tangent indexing functor; taking the pseudolimit of this diagram then gives us the descent equivariant tangent category of smooth manifolds over $M$. We also conclude this section much as we do the section on the descent equivariant tangent category of schemes: we show that the tangent category of smooth manifolds over $M$ is equivalent to the pseudolimit of the tangent categories for the charts $\lbrace U_i \to M \; \left. \right| \; i \in I \rbrace$ of a smooth atlas for $M$ and then extend this to the equivariant setting as well.
\begin{prop}[cf. Proposition \ref{Prop: Manifold gluing equivalence of tan cats}]
Let $M$ be a smooth manifold and let $\Ucal = \lbrace f_i:U_i \to M \; \left. \right| \; i \in I \rbrace$ be a chart in a smooth atlas for $M$. Then there is an equivalence of tangent categories
\[
\left(\SMan_{/M},\Tbb_{\mathbf{Smooth}/M}\right) \simeq \operatorname{pseudolim}\left(\SMan_{/U_i},\Tbb_{\mathbf{Smooth}/U_i}\right).
\]
\end{prop}
\begin{prop}[cf. Proposition \ref{Prop: Equivariantifying the gluing equiv for manifolds}]
Let $\Ucal := \lbrace f_i:U_i \to M \; \left. \right| \; i \in I \rbrace$ be a chart in a smooth atlas for a manifold $M$ for which each map $f_i$ is $L$-equivariant for a Lie group satisfying Assumption \ref{Assume on Lie groups}. Then there is an equivalence of tangent categories
\[
\left(\left(\SMan_{/M}\right)_L,\Tbb_{M}\right) \simeq \operatorname{pseudolim} \left(\left(\SMan_{/U_i}\right)_L, \Tbb_{U_i}\right).
\]
\end{prop}

Finally, in Section \ref{Section: Towards Generalizations} we make some comments towards generalizing the results in this paper. In particular, we discuss some thoughts towards generalizing our constructions to orbifolds.

Throughout the paper we will keep careful track of the role of the structure 2-cells in the whole framework; e.g., although the pseudolimit of a diagram defined by a pseudofunctor is functorial along pseudonatural transformations, it need not be along lax or oplax transformations. It is tempting to want to generalize our work to lax indexing functors, (op)lax limits and the category of tangent categories with lax or colax morphisms. However, this is not possible and there are several reasons for this. To clarify what can be generalized and what cannot be generalized, we will add comments throughout the paper to make the reader aware of these 2-categorical subtleties.

\subsection{Remarks on Chronology and Perspectives}
The key definition we use in this paper, namely that of a tangent indexing functor (cf. Definition \ref{Defn: tangent pre-equivariant indexing functor}), is similar to what is called an indexed tangent category in \cite[Definition 3.7]{Marcello}. On its face, our notions are simply similar: on the one hand, we define tangent indexing functors to be pseudofunctors $F\colon\Cscr^{\op} \to \fCat$ which, broadly speaking, have categories $F(X)$ given by tangent categories and which have functors $F(f)$ which form the functor component of a morphism $(F(f),\!\quot{\alpha}{f})$ of a strong morphism of tangent categories (in the lax direction) for which the transformations $\!\quot{\alpha}{f}$ vary pseudonaturally in $\Cscr^{\op}$ (cf. Definition \ref{Defn: Tangent Morphism} for the definition of strong tangent morphisms). On the other hand, in \cite{Marcello} an indexed tangent category is a tangent category $(\Cscr,\Tbb)$ equipped with a pseudofunctor $F\colon\Cscr^{\op} \to \fTan_{\operatorname{strong}}$ where $\fTan_{\operatorname{strong}}$ is the $2$-category of tangent categories and strong tangent morphisms. These notions are readily compared and contrasted in the following immediate ways:
\begin{itemize}
    \item The definition of an indexed tangent category $((\Cscr,\Tbb), F\colon\Cscr^{\op} \to \fTan_{\operatorname{strong}})$ in \cite{Marcello} has nothing to do with the tangent structure $\Tbb$ on $\Cscr$. For instance, if $\Ibb$ is the identity tangent structure and if $\Tbb$ is an arbitrary tangent structure on $\Cscr$ then $((\Cscr,\Ibb),F\colon\Cscr^{\op} \to \fTan_{\operatorname{strong}})$ and $((\Cscr,\Tbb),F\colon\Cscr^{\op} \to \fTan_{\operatorname{strong}})$ form indexed tangent categories with the same pseudofunctor but with distinct tangent structures. In particular, the pseudofunctor $F$ need not depend on any way on the tangent structure $\Tbb$. 
    \item The definition of a tangent indexing functor (cf. Definition \ref{Defn: tangent pre-equivariant indexing functor}), while presented in a minimal way, is equivalent to giving a pseudofunctor $F\colon\Cscr^{\op} \to \fTan_{\operatorname{strong}}$ (cf. Proposition \ref{Prop: Tangent indexing functor is a pseudofunctor into Tanstrong}). The difference in perspective largely comes from the fact that we are interested primarily in pseudolimits (and in particular showing that the pseudolimit of a diagram $F\colon\Cscr^{\op} \to \fCat$ in $\fCat$ lifts to give a pseudolimit $F\colon\Cscr^{\op} \to \fTan_{\operatorname{strong}}$ whenever $F$ is a tangent indexing functor), while \cite{Marcello} is interested in a tangent-categorical version of the Grothendieck construction.
\end{itemize}
We show that these two seemingly distinct concepts are equivalent (cf. Theorems \ref{Thm: Indexed tangent category} and \ref{Thm: Tangent objects in Bicat Hom category are tangent indexing functors}) in the sense that both present data which is both necessary and sufficient to giving tangent objects in the $2$-category $\Bicat(\Cscr^{\op},\fCat)$. We present our focus on indexing tangent functors in order to both be in a position to ask for minimal structure and prove what exists by virtue of the pseudofunctor itself, as well as to show that the $2$-functor $\Forget:\fTan \to \fCat$ creates pseudolimits of pseudofunctors $F$ which are indexed by $1$-categories. Finally, we would also like to remark that a preliminary version of our notion of tangent indexing functor appeared prior to \cite{Marcello}: the second author first presented a notion of tangent indexing functor (and in particular the notions which appear in Sections \ref{Section: Review of EQCats} and \ref{Section: Example of equivariant tangent structure on a variety} regarding varieties) at the 2022 Foundational Methods in Computer Science conference in \cite{FMCS2022Talk}.

\subsection{Acknowledgements}
The first author would like to thank the support of an NSERC Discovery Grant and the second author would like to thank the support of an AARMS postdoctoral fellowship. The second author would also like to thank the organizers of the 2022 Foundational Methods of Computer Science conference for the hospitality and intellectually stimulating environment as well as the opportunity to present a preliminary version of this work. Both authors would also like to thank Marcello Lanfranchi for illuminating discussions and sharing his work and ideas with us. Both authors would also like to thank the anonymous referees for their insightful questions and helpful comments that have led them to write a much better structured paper where the pseudolimits of tangent categories with their properties take the center place, followed by a wide variety of examples. 

Competing interests: The authors declare none.

\section{A Review of Pseudocones}\label{Section: Higher Cat Review}

In this section we review the 2-categorical material on pseudolimits that will be used as foundation for everything we do in this paper.
Specifically, we will be considering lax and pseudocones over pseudofunctors into $\fCat$, the 2-category of categories. We also will be defining the hom-bicategory $\Bicat(\mathcal{B}, \mathcal{B}^{\prime})$ of pseudofunctors, pseudonatural transformations, and modifications between two given bicategories $\Bcal$ and $\Bcal^{\prime}$; when the bicategory $\Bcal^{\prime}$ is a $2$-category, so too is the hom-bicategory $\Bicat(\Bcal,\Bcal^{\prime})$ by \cite[Corollary 4.4.12]{TwoDimCat}.

For readers new to the theory of 2-categories and bicategories, we recommend \cite{TwoDimCat} as a reference.
Specifically, the definitions of 2-categories and bicategories can be found in \cite[Definition 2.1.3]{TwoDimCat}. The notion of pseudofunctor is defined in {\cite[Definition 4.1.2]{TwoDimCat}}. However, as we will use slightly different notation for the laxity constraints, also called structure cells, of a lax or pseudofunctor we include our version of the definition of pseudofunctor here:

\begin{dfn}
    Given two bicategories $\mathcal{B}$ and $\mathcal{B}'$, a pseudofunctor $F\colon \mathcal{B}\to\mathcal{B}'$ consists of
    \begin{itemize}
        \item A function $F_0\colon \mathcal{B}_0\to\mathcal{B}'_0$ on objects;
        \item For each pair of objects $X,Y\in\mathcal{B}_0$, a local functor
        $$F_{X,Y}\colon \mathcal{B}(X,Y)\to\mathcal{B}'(F_0X,F_0Y)$$
        \item For each triple of objects $X,Y,Z\in\mathcal{B}$, invertible structure transformations:
        \[
        \begin{tikzcd}
        \Bcal(Y,Z) \times \Bcal(X,Y) \ar[rrr]{}{c} \ar[d, swap, ""{name = U}]{}{F_{Y,Z} \times F_{X,Y}} & & & \Bcal(X,Z) \ar[d, ""{name = D}]{}{F_{X,Z}} \\
        \Bcal^{\prime}(FY, FZ) \times \Bcal^{\prime}(FX,FY) \ar[rrr, swap]{}{c^{\prime}} & & & \Bcal^{\prime}(FX,FZ)
        \ar[from = U, to = D, Rightarrow, shorten <= 50pt, shorten >= 50pt]{}{\!\quot{\phi}{X,Y,Z}}
        \end{tikzcd}\]
        \[
        \begin{tikzcd}
        \mathbbm{1} \ar[r]{}{1_X} \ar[dr, swap, ""{name = D}]{}{1^{\prime}_{FX}} & \Bcal(X,X) \ar[d, ""{name = R}]{}{F_{X,X}} \\
         & \Bcal^{\prime}(FX,FX)
         \ar[from = D, to = R, Rightarrow, shorten <= 16pt, shorten >= 4pt]{}{\!\quot{\phi}{X}}
        \end{tikzcd}
        \]
        with component 2-cells $\phi_{g,f}={\!\quot{\phi}{X,Y,Z}}_{g,f}\colon Fg\circ Ff\stackrel{\sim}{\Rightarrow} F(gf)$ and $\phi_X={\!\quot{\phi}{X}}_{*}\colon 1'_{FX}\stackrel{\sim}{\Rightarrow}F(1_X)$, where $*$ is the unique object of $\mathbbm{1}$. (Note that we will generally not use the more cumbersome notation involving the objects.)
        These 2-cells are required to satisfy coherence conditions with respect to the associativity and unity cells of $\mathcal{B}'$ as described in \cite[Definition 4.1.2]{TwoDimCat}.
    \end{itemize}
\end{dfn}
The notions of lax natural transformation between these functors is described in  \cite[Definition 4.2.1]{TwoDimCat} and the modifications between transformations can be found in \cite[Definition 4.4.1]{TwoDimCat}.
Here we just want to point out that such a transformation $\alpha\colon F\Rightarrow G\colon\mathcal{B}\rightarrow \mathcal{B}'$ is given by the following data:
\begin{itemize}
    \item for each object $X\in\mathcal{B}_0$, an arrow $\alpha_X\colon FX\to GX$ in $\mathcal{B}'$;
    \item for each arrow $X\stackrel{f}{\rightarrow}Y$ in $\mathcal{B}$, a 2-cell
    \[
    \begin{tikzcd}
    FX \ar[r]{}{Ff} \ar[d, swap, ""{name = L}]{}{\alpha_X} & FY \ar[d, ""{name= R}]{}{\alpha_Y} \\
    GX \ar[r, swap]{}{Gf} & GY
    \ar[from = L, to = R, Rightarrow, shorten <= 16pt, shorten >= 16pt]{}{\alpha_f}
    \end{tikzcd}
    \]
\end{itemize}
These data need to satisfy unity and naturality conditions as can be found in \cite[Definition 4.4.1]{TwoDimCat}.
\begin{rmk}
We make the simplifying assumption that all pseudofunctors $F\colon \Cscr^{\op} \to \fCat$ we consider in this paper are normalized, i.e., $F(\id_{ A}) = \id_{FA}$, for all objects $A \in \Cscr_0$, instead of merely having a natural isomorphism between the functors. This may be done without loss of generality; cf. \cite[Expos{\'e} VI.9, pg. 180, 181]{sga1}.
\end{rmk}

With this notation in place we can begin our discussion of pseudocones and pseudolimits.
\subsection{Pseudocones}
Given a category $\Cscr$, a {\em lax cone} $(\Kscr,\kappa)$ with vertex $\Kscr$ for a contravariant normal pseudofunctor $F\colon\Cscr^{\op}\to\fCat$ is given by a lax transformation $\kappa\colon\const(\Kscr)\Rightarrow F$, where $\const(\Kscr)\colon\Cscr^{\op}\to\fCat$ is the constant functor with value $\Kscr$ on the objects and sending all arrows and 2-cells to identities.

In detail, this lax cone is given by a functor $\!\quot{\kappa}{C}\colon\Kscr\to FC$ for each object $C\in\Cscr$ and a natural transformation,
\[
\begin{tikzcd}
 & \Kscr \ar[dr, ""{name = R}]{}{\!\quot{\kappa}{D}} \ar[dl, swap, ""{name= L}]{}{\!\quot{\kappa}{C}} & \\
FC \ar[rr, swap]{}{Ff} & & FD
\ar[from = L, to = R, Rightarrow, swap, shorten <= 10pt, shorten >= 10pt]{}{\!\quot{\kappa}{f}}
\end{tikzcd}
\]
for each arrow $D\xrightarrow{f}C$ in $\Cscr$.
These transformations are functorial in the sense that
\[
\begin{tikzcd}
 & & \Kscr \ar[dll, swap, ""{name = L}]{}{\!\quot{\kappa}{C}} \ar[d, ""{name = M}]{}{}[description]{\!\quot{\kappa}{D}} \ar[drr, ""{name = R}]{}{\!\quot{\kappa}{E}} & & \\
FC \ar[rr, swap]{}{Ff} & & FD \ar[rr, swap]{}{Fg} & & FE
\ar[from = L, to = M, Rightarrow, shorten <= 10pt, shorten >= 10pt, swap]{}{\!\quot{\kappa}{f}}
\ar[from = M, to = R, Rightarrow, shorten <= 10pt, shorten >= 10pt, swap]{}{\!\quot{\kappa}{g}}
\end{tikzcd}
\]
is equal to
\[
\begin{tikzcd}
 & \Kscr \ar[dr, ""{name = UR}]{}{\!\quot{\kappa}{E}} \ar[dl, swap, ""{name= UL}]{}{\!\quot{\kappa}{C}} \\
FC \ar[rr, ""{name = M}]{}[description]{F(f \circ g)} \ar[dr, swap]{}{Ff} & & FE \\
 & FD \ar[ur, swap]{}{Fg}
\ar[from = 3-2, to = M, Rightarrow, shorten <= 8pt, shorten >= 4pt]{}{\phi_{g,f}}
\ar[from = UL, to = UR, Rightarrow, shorten <= 10pt, shorten >= 10pt, swap]{}{\!\quot{\kappa}{g \circ f}}
\end{tikzcd}
\]
(where $\phi_{g,f}$ is a structure cell for the pseudofunctor $F$)
and $\!\quot{\kappa}{\id_C}=\id_{\!\quot{\kappa}{C}}$. The lax cone $(\Kscr, \kappa)$ is a {\em pseudocone} if all $\!\quot{\kappa}{f}$ are invertible 2-cells.

\begin{rmk}
    Note that we normally write the indices for the components of a transformation as right-hand indices. However, we will encounter a number of occasions, where the components of a transformation are themselves transformations or modifications that have components of their own. In that case, we will write the first components as left-indices, so that the right index is free for the next level.  
\end{rmk}

\begin{dfn}
A pseudocone $(\Lscr,\lambda)$ over $F$ is a {\em limiting pseudocone} if it satisfies the following two universal properties:
\begin{itemize}
    \item For any other pseudocone $(\Kscr, \kappa)$ over $F$ there is a unique arrow $r\colon \Kscr\to\Lscr$ such that $\lambda_Cr=\kappa_C$ for all objects $C\in\Cscr$ and $\lambda_f \ast r=\kappa_f$ for each arrow $f$ in $\Cscr$.
    \item For any map between cones with the same vertex, $(\Kscr,\kappa)\Rightarrow(\Kscr,\kappa')$ given by a compatible family of 2-cells $\theta_C\colon\!\quot{\kappa}{C}\Rightarrow\!\quot{\kappa'}{C}$ there is a unique 2-cell $\rho\colon r\Rightarrow r'\colon \Kscr\rightarrow\Lscr$ between the induced unique arrows such that 
    $\lambda_C\ast \rho=\theta_C$ for all $C\in\Cscr$.
\end{itemize}
\end{dfn}

\begin{rmk}\label{Rmk: limiting_pseudocone}
The basics of limiting pseudocones were discussed in \cite[Chapitre I, Section 1.1.5]{Giraud} originally under the language of fibred categories (without proof that their model gives a pseudolimit and without the pseudouniversal property) and also, more explicitly and $2$-categorically, in \cite[Section 6.10]{SteveLack2Cat}. We follow the exposition of \cite{GeneralGeoffThesis}, however, as it is quite explicit and provides a concrete model for working with the pseudolimit of a pseudofunctor in $\fCat$.

   By Theorem 2.3.16 on page 43 of \cite{GeneralGeoffThesis} the apex for the limiting pseudocone for a pseudofunctor $F$ into $\fCat$ is given by 
$$\PC(F)=\Bicat(\Cscr^{\op},\fCat)(\const(\mathbbm{1}),F).$$ This is the category of pseudocones over $F$; it has been described in full detail in \cite{GeneralGeoffThesis}. Here we include a quick description of its objects and arrows (taken from page 35 of \cite{GeneralGeoffThesis}): 
\begin{itemize}
    \item objects are natural transformations $\const(\mathbbm{1})\Rightarrow F$, which correspond to a family of objects,
\[
A = \lbrace \negthinspace\quot{A}{X} \in F(X)_0 \; \left. \right| \; X \in \Cscr_0 \rbrace
\] 
with a family of invertible transition morphisms
\[
\Sigma_A = \lbrace \tau_f^{A}\colon F(f)(\negthinspace\quot{A}{Y}) \xrightarrow{\sim} \negthinspace\quot{A}{X} \; \left. \right| \; f\colon X \to Y, f \in \Cscr_1 \rbrace
\]
satisfying the cocycle condition,
\[\tau_{g \circ f}^A \circ \phi_{f,g} = \tau_f^A \circ F(f)\left(\tau_g^A\right)\]
for any pair of composable morphisms $X \xrightarrow{f} Y \xrightarrow{g} Z$ in $\Cscr$.
We will write $(A,\Sigma_A)$ for the object corresponding to this data.
\item  a morphism $P\colon (A,\Sigma_A) \to (B, \Sigma_B)$ is a modification determined by a collection of morphisms 
		\[
		P = \lbrace \negthinspace\quot{\rho}{X}\colon \negthinspace\quot{A}{X} \to \negthinspace\quot{B}{X} \; \left. \right| \; X \in \Cscr_0 \rbrace
		\]
		such that for any morphisms $f\colon X \to Y$ the diagram
		\[
		\begin{tikzcd}
			F(f)(\negthinspace\quot{A}{X}) \ar[rr]{}{F(f)(\negthinspace\quot{\rho}{X})} \ar[d, swap]{}{\tau_f^A} & & F(f)(\negthinspace\quot{B}{Y}) \ar[d]{}{\tau_f^B} \\
			\negthinspace\quot{A}{X} \ar[rr, swap]{}{\quot{\rho}{X}}	& & \negthinspace\quot{B}{X}
		\end{tikzcd}
		\]
		commutes.
\end{itemize}
As described in \cite{GeneralGeoffThesis}, the pseudonatural transformation $\const(\PC(F))\Rightarrow F$ defining the limiting pseudocone has component functors $p_X\colon \PC(F) \to F(X)$, indexed by $X \in \Cscr_0$, defined  as the composites
	\[
	\begin{tikzcd}
		\PC(F) \ar[r]{}{p} & \prod\limits_{X \in \Cscr_0} F(X) \ar[r]{}{\pi_X} & F(X)
	\end{tikzcd}
	\]
	where $\pi_X$ is the usual projection out of the product and $p\colon \PC(F) \to \prod_{X \in \Cscr_0} F(X)$ is the functor which on objects $(A,T_A)$ forgets the transition isomorphisms and sends $A$ to $(\negthinspace\quot{A}{X})_{X \in \Cscr_0}$ and on morphisms sends $P$ to $(\negthinspace\quot{\rho}{X})_{X \in \Cscr_0}$. Now define the natural isomorphism $p_f\colon F(f) \circ p_Y \Rightarrow p_X$ by setting, for all objects $(A,T_A) \in \PC(F)_0$,
	\[
	p_f^{A} := \tau_f^{A}.
	\]
\end{rmk}

\subsection{Functoriality of $\PC(F)$}
Similar to the case for ordinary cones for diagrams in 1-categories, a pseudonatural transformation ${\alpha}\colon F \Rightarrow E$ gives rise to a morphism $\underline{\alpha}\colon\PC(F) \to \PC(E)$ between their pseudocone categories, through composition,
\[
\begin{tikzcd}
\mathbbm{1} \ar[r]{}{A} & F \ar[r]{}{\alpha} & E.
\end{tikzcd}
\]
Similarly, a modification $\rho\colon \alpha \Rrightarrow \beta\colon F \Rightarrow E$ gives rise to a natural transformation
$\underline{\rho}$ between functors $\underline{\alpha},\underline{\beta}\colon \PC(F) \to \PC(E)$ through post-whiskering,
\[
\begin{tikzcd}
\mathbbm{1} \ar[r]{}{A} & F \ar[rr, bend left = 30, ""{name = U}]{}{\alpha} \ar[rr, bend right = 30, swap, ""{name = D}]{}{\beta} & & E
\ar[from = U, to = D, Rightarrow, shorten <= 4pt, shorten >= 4pt]{}{\rho}
\end{tikzcd}
\]
However, in the next section we will need the exact details of the components of these morphisms and 2-cells; in particular, the explicit formulation of the transition isomorphisms of an object $\underline{\alpha}(A)$, in order to define tangent structures on pseudocones. We will give this explicit formulation in the next theorem, which was taken from \cite{GeneralGeoffThesis}, and include parts of the proof, because it illustrates why we need to work with pseudofunctors in order for the category of pseudocones to have the universal property of the pseudolimit (and hence also the correct functoriality). In particular we show that this formulation of the transition isomorphisms cannot necessarily be done with lax or oplax transformations. For further explicit details regarding these calculations, see \cite[Theorem 4.1.1, Page 80]{GeneralGeoffThesis} for the case of the functors $\underline{\alpha}\colon \PC(F) \to \PC(E)$ and \cite[Lemma 4.1.11, Page 90]{GeneralGeoffThesis} for the calculation of the explicit form of the natural transformations.

\begin{Theorem}[{\cite{GeneralGeoffThesis}, Theorem 4.1.1, Lemma 4.1.11}]\label{Theorem: Defining functors and nat transforms between equivariant categories and functors}
Assume $F,E\colon \Cscr^{\op} \to \fCat$ are pseudofunctors.
\begin{enumerate}
    \item[$(1)$] If $\alpha\colon F \Rightarrow E$ is a pseudonatural transformation with component functors $\quot{\alpha}{X}$ and natural isomorphisms:
\[
\begin{tikzcd}
F(Y) \ar[rr, bend left = 30, ""{name = U}]{}{\quot{\alpha}{X} \circ F(f)} \ar[rr, bend right = 30, swap, ""{name = L}]{}{E(f) \circ \quot{\alpha}{Y}} & & E(X) \ar[from = U, to = L, Rightarrow, shorten <= 4pt, shorten >=4pt]{}{\quot{\alpha}{f}} \ar[swap, from = U, to = L, Rightarrow, shorten <= 4pt, shorten >=4pt]{}{\cong}
\end{tikzcd}
\]
then there is a functor $\underline{\alpha}\colon \PC(F) \to \PC(E)$ given on objects $(A,\Sigma_A)$ of $\PC(F)$ by $\underline{\alpha}((A,\Sigma_A)=(\underline{\alpha}(A),\Sigma_{\underline{\alpha}A})$, where
\[
\underline{\alpha}(A) = \left\lbrace \negthinspace\quot{\alpha}{X}(\negthinspace\quot{A}{X}) \; \left. \right| \; X \in \Cscr_0, \negthinspace\quot{A}{X} \in A\right\rbrace
\]
and
\[
\Sigma_{\underline{\alpha}A} = \left\lbrace \!\quot{\alpha}{X}\left(\tau_f^A\right) \circ \left(\left(\!\quot{\alpha}{f}\right)_{\negthinspace\quot{A}{Y}}\right)^{-1} \; \colon  \; f \in \Cscr_1 \right\rbrace.
\]
\item[$(2)$] If $\alpha,\beta\colon F \Rightarrow E$ are pseudonatural transformations between $F$ and $E$ and if $\rho\colon \alpha \Rrightarrow \beta$ is a modification from $\alpha$ to $\beta$ then there is a natural transformation $\underline{\rho}$ providing a $2$-cell
\[
\begin{tikzcd}
\PC(F) \ar[rr, bend left = 20, ""{name = U}]{}{\underline{\alpha}} \ar[rr, bend right = 20, swap, ""{name = L}]{}{\underline{\beta}} & & \PC(E) \ar[from = U, to = L, Rightarrow, shorten <= 4pt, shorten >=4pt]{}{\underline{\rho}}
\end{tikzcd}
\]
where $\underline{\rho}_{A}\colon \underline{\alpha}A \to \underline{\beta}A$ is defined by
\[
\underline{\rho}_{A} := \left\lbrace\left. \left(\!\quot{\rho}{X}\right)_{\!\quot{A}{X}}\colon \!\quot{\alpha}{X}\left(\!\quot{A}{X}\right) \to \!\quot{\beta}{X}\left(\!\quot{A}{X}\right) \;  \right| \; X \in \Cscr_0 \right\rbrace.
\]
\end{enumerate}
\end{Theorem}
\begin{proof}[Sketch]
Statement $(1)$: Here we import parts of the proof of \cite[Theorem 4.1.1]{GeneralGeoffThesis} in order to indicate the use of inverses in giving this description. We only focus on the fact that the claimed form of the transition isomorphisms $\Sigma_{\underline{\alpha}A}$ ultimately comes down to the two chains of equalities which we present below.

	Because the vertical composition 
	\[
	\underline{\alpha}A = \alpha \circ (A,\Sigma_A) = (\underline{\alpha}(A), \Sigma_{\underline{\alpha}A})
	\] 
	has object components
	\[
	\underline{\alpha}(A) = \left\lbrace\! \quot{\alpha}{X}\left(\!\quot{A}{X}\right) \; \colon  \; X \in \Cscr_0 \right\rbrace
	\]	
	we only need to determine the transition isomorphisms $\Sigma_{\underline{\alpha}A}$. From our convention regarding the direction of transition isomorphisms and the definition of the vertical composite $\alpha \circ (A,\Sigma_A)$ we get that
	\[
	\tau_f^{\underline{\alpha}A} = \left(\!\quot{(\alpha \circ (A,T_A))}{f}\right)^{-1} = \!\quot{\alpha}{X}\left(\!\quot{(A,T_A)}{f}\right)^{-1} \circ \!\quot{\alpha}{f}^{-1} = \quot{\alpha}{X}(\tau_f^A) \circ \left(\!\quot{\alpha}{f}\right)^{-1}.
	\]
	The fact that this indeed gives a pseudonatural transformation (and that we did not make an error in defining the transition isomorphisms) follows from the fact that the pseudonaturality of $\alpha$ implies that if we have any composable pair of morphisms $X \xrightarrow{f} Y \xrightarrow{g} Z$
	\[
	E(f)\left(\!\quot{\alpha}{Y}\left(\tau_{g}^{A}\right)\right) = \left(\!\quot{\alpha}{f}\right)_{\!\quot{A}{Y}} \circ \!\quot{\alpha}{X}\left(F(f)\tau_g^A\right) \circ \left(\left(\!\quot{\alpha}{f}\right)_{F(g)(\!\quot{A}{Z})}\right)^{-1}.
	\]
	This allows us to compute that
\begin{align}
		&\tau_{f}^{\underline{\alpha}A} \circ E(f)\tau_g^{\underline{\alpha}A} \notag \\
  &= \quot{\alpha}{X}(\tau_f^A) \circ \big(\negthinspace\quot{\alpha_{\quot{A}{Y}}}{f}\big)^{-1} \circ E(f)\big(\negthinspace\quot{\alpha}{Y}(\tau_g^A) \circ \left(\negthinspace\quot{\alpha_{\quot{A}{Z}}}{g}\right)^{-1}\big) \notag \\
		&=\quot{\alpha}{X}(\tau_f^A) \circ \left(\negthinspace\quot{\alpha_{\quot{A}{Y}}}{f}\right)^{-1} \circ E(f)\big(\negthinspace\quot{\alpha}{Y}(\tau_g^A)\big) \circ E(f)\big(\negthinspace\quot{\alpha_{\quot{A}{Z}}}{g}\big)^{-1} \notag \\
		&= \quot{\alpha}{X}(\tau_f^A) \circ \left(\negthinspace\quot{\alpha_{\quot{A}{Y}}}{f}\right)^{-1} \circ \left(\negthinspace\quot{\alpha_{\quot{A}{Y}}}{f}\right) \circ \negthinspace\quot{\alpha}{X}\big(F(f)\tau_g^A\big) \circ \left(\negthinspace\quot{\alpha_{F(g)\!\quot{A}{Z}}}{f}\right)^{-1} \circ E(f)\big(\negthinspace\quot{\alpha_{\quot{A}{Z}}}{g}\big)^{-1} \label{Eqn: First line of the inverse cancellation}\\
		&= \quot{\alpha}{X}(\tau_f^A) \circ \negthinspace\quot{\alpha}{X}(F(f)\tau_g^{A}) \circ \left(\negthinspace\quot{\alpha_{F(g)\!\quot{A}{Z}}}{f}\right)^{-1} \circ E(f)\big(\negthinspace\quot{\alpha_{\quot{A}{Z}}}{g}\big)^{-1} \label{Eqn: Second line of inverse cancellation} \\
		&= \quot{\alpha}{X}\big(\tau_f^A \circ F(f)\tau_g^A\big) \circ \left(\negthinspace\quot{\alpha_{F(g)\!\quot{A}{Z}}}{f}\right)^{-1} \circ E(f)\big(\negthinspace\quot{\alpha_{\quot{A}{Z}}}{g}\big)^{-1} \notag \\
		&= \quot{\alpha}{X}(\tau_{g \circ f}^A \circ \negthinspace\quot{\phi_{f,g}}{F})\circ \left(\negthinspace\quot{\alpha_{F(g)\!\quot{A}{Z}}}{f}\right)^{-1} \circ E(f)\big(\negthinspace\quot{\alpha_{\quot{A}{Z}}}{g}\big)^{-1} \notag \\
		&= \quot{\alpha}{X}(\tau_{g \circ f}^{A}) \circ \negthinspace\quot{\alpha}{X}(\negthinspace\quot{\phi_{f,g}}{F})\circ \left(\negthinspace\quot{\alpha_{F(g)\!\quot{A}{Z}}}{f}\right)^{-1} \circ E(f)\big(\negthinspace\quot{\alpha_{\quot{A}{Z}}}{g}\big)^{-1}. \notag
	\end{align}
	We again use the pseudonaturality of $\alpha$ to obtain the commuting diagram
	\[
	\begin{tikzcd}
		E(f)\big(E(g)\big(\!\quot{\alpha}{Z}(\!\quot{A}{Z})\big)\big) \ar[rr]{}{E(f)\big(\!\quot{\alpha_{\quot{A}{Z}}}{g}\big)^{-1}} \ar[dd, swap]{}{\!\quot{\phi_{f,g}}{E}} & & E(f)\big(\!\quot{\alpha}{Y}(F(g)(\!\quot{A}{Z}))\big) \ar[d]{}{\big(\!\quot{\alpha_{F(g)\quot{A}{Z}}}{f}\big)^{-1}} \\
		& & \ar[d]{}{\!\quot{\alpha}{X}(\!\quot{\phi_{f,g}}{X})} \!\quot{\alpha}{X}\big(F(f)(F(g)(\!\quot{A}{Z}))\big) \\
		E(g \circ f)\big(\!\quot{\alpha}{Z}(\!\quot{A}{Z})\big)\ar[rr, swap]{}{\quot{\alpha_{\quot{A}{Z}}}{g \circ f}^{-1}} & & \!\quot{\alpha}{X}\big(F(g \circ X)(\!\quot{A}{Z})\big)
	\end{tikzcd}
	\]
	and substitute the induced equation to get
	\begin{align*}
		&\!\quot{\alpha}{X}(\tau_{g \circ f}^{A}) \circ \!\quot{\alpha}{X}(\!\quot{\phi_{f,g}}{F})\circ \left(\!\quot{\alpha_{F(g)\!\quot{A}{Z}}}{f}\right)^{-1} \circ E(X)\big(\!\quot{\alpha_{\quot{A}{Z}}}{g}\big)^{-1} \\&= \!\quot{\alpha}{X}(\tau_{g \circ f}^A) \circ \big(\!\quot{\alpha_{\quot{A}{Z}}}{g \circ f}\big)^{-1} \circ \!\quot{\phi_{f,g}}{E} = \tau_{g \circ f}^{\underline{\alpha}A} \circ\! \quot{\phi_{f,g}}{E}.
	\end{align*}
	This shows that $\tau_{g \circ f}^{\underline{\alpha}A} \circ \quot{\phi_{f,g}}{E} = \tau_{f}^{\underline{\alpha}A} \circ E(f)(\tau_g^{\underline{\alpha}A})$ and hence that $\underline{\alpha}(A)$ has the form claimed in the statement of the theorem.
	
	Statement $(2)$ follows from untangling the combinatorial description of the whiskering of the $2$-morphism $P$ by the $1$-morphism $\alpha$ in $\Bicat(\Cscr^{\op},\fCat)$. Explicitly we find
	\[
	\underline{\alpha}(P) = \alpha \circ P = \alpha \ast P
	\]
	and so the description
	\[
	\alpha \ast P = \underline{\alpha}P = \left\lbrace \negthinspace\quot{\alpha}{X}\left(\negthinspace\quot{\rho}{X}\right) \; : \; X \in \Cscr_0\right\rbrace
	\]
	follows. That this is indeed a modification (and hence a morphism in $\PC(E)$) is a straightforward check using the identity
	\[
	E(f)\left(\!\quot{\alpha}{Y}\left(\!\quot{\rho}{Y}\right)\right) = \!\quot{\alpha}{f}_{\!\quot{B}{Y}} \circ \!\quot{\alpha}{X}\left(F(f)\!\quot{\rho}{Y}\right) \circ \left(\!\quot{\alpha}{f}_{\!\quot{A}{Y}}\right)^{-1}
	\]
	which allows us to deduce
	\begin{align*}
		\tau_f^{\underline{\alpha}B} \circ E(f)\big(\!\quot{\alpha}{Y}(\!\quot{\rho}{Y})\big)
		&= \tau_f^{\underline{\alpha}B} \circ \!\quot{\alpha_{\quot{B}{Y}}}{f} \circ \!\quot{\alpha}{X}\big(F(f)\!\quot{\rho}{Y}\big)  \circ \left(\!\quot{\alpha_{\quot{A}{Y}}}{f}\right)^{-1} \\
		&= \!\quot{\alpha}{X}(\tau_f^B) \circ \left(\!\quot{\alpha_{\quot{B}{Y}}}{f}\right)^{-1} \circ \!\quot{\alpha_{\quot{B}{Y}}}{f} \circ \!\quot{\alpha}{X}\big(F(f)\!\quot{\rho}{Y}\big)  \circ \left(\!\quot{\alpha_{\quot{A}{Y}}}{f}\right)^{-1} \\
		&= \!\quot{\alpha}{X}(\tau_f^B) \circ\!\quot{\alpha}{X}\big(F(f)\!\quot{\rho}{Y}\big) \circ \left(\!\quot{\alpha_{\quot{A}{Y}}}{f}\right)^{-1} \\
		&= \!\quot{\alpha}{X}\big(\tau_f^B \circ F(f)\!\quot{\rho}{Y}\big) \circ \left(\!\quot{\alpha_{\quot{A}{Y}}}{f}\right)^{-1} \\
		&= \!\quot{\alpha}{X}(\!\quot{\rho}{X} \circ \tau_f^A) \circ \left(\!\quot{\alpha_{\quot{A}{Y}}}{f}\right)^{-1} \\
        &= \!\quot{\alpha}{X}(\!\quot{\rho}{X}) \circ \!\quot{\alpha}{X}(\tau_f^A) \circ \left(\!\quot{\alpha_{\quot{A}{Y}}}{f}\right)^{-1} \\
		&= \!\quot{\alpha}{X}(\!\quot{\rho}{X}) \circ \tau_{f}^{\underline{\alpha}A}
	\end{align*}
	which is what was to be shown.
\end{proof}

\begin{rmk}
It is worth noting that the exact use of the invertibility of the witness transformations of the pseudonatural transformations appears in Lines (\ref{Eqn: First line of the inverse cancellation}) and (\ref{Eqn: Second line of inverse cancellation}) of the proof above. In particular, we actually cancelled a natural isomorphism with its inverse. This shows that if we tried to define the pseudolimit's 2-dimensional universal property with respect to lax or oplax transformations, this would not work. For similar reasons, we cannot try to define a pseudolimit for a lax or oplax functor in this way.  There is a similar result for oplax limits of lax functors (in terms of the category of oplax cones).  However, the explicit formula for the transition isomorphisms $\Sigma_{\underline{\alpha}A}$ as given in Theorem \ref{Theorem: Defining functors and nat transforms between equivariant categories and functors} in this oplax cone extension of the result will not be the same.
\end{rmk}
\begin{cor}\label{Cor: PC strict model}
    For any $1$-category $\Cscr$, the pseudocone construction $\PC(-)$ gives a strict $2$-functor $\PC:\Bicat(\Cscr^{\op},\fCat) \to \fCat$. In particular, $\PC(-)$ gives a strictly functorial model of $1$-category-indexed pseudolimits in $\fCat$.
\end{cor}

\subsection{Limits in $\PC(F)$}
In Section \ref{Section: Eq Tan Cats}, we will need to discuss ($1$-categorical) limits in pseudocone categories $\PC(F)$ as we will need to consider certain pullbacks when investigating the conditions on $\PC(F)$ to be a tangent category. To aid us, we recall a handy theorem from \cite{GeneralGeoffThesis} which gives a description of these ($1$-categorical) limits in terms of data recorded by the pseudofunctor $F\colon \Cscr^{\op} \to \fCat$. Effectively, this gives that if we have a diagram $d\colon I \to \PC(F)$ for which each induced diagram $I \xrightarrow{d} \PC(F) \xrightarrow{\pi_X} F(X)$ has a limit, and if the functors $F(f)$ preserve these limits, then $\PC(F)$ has a limit for the diagram $d$ whose object-local components are given by the limits in $F(X)$. Explicit details regarding this construction may be found in \cite[Theorem 3.1.4, Page 54]{GeneralGeoffThesis}.

\begin{Theorem}[{\cite{GeneralGeoffThesis}, Theorem 3.1.4}]\label{Thm: Limits in pseudcone categories}
Let $F\colon\Cscr^{\op} \to \fCat$ be a pseudofunctor and assume that we have a diagram of shape $d\colon I \to \PC(F)$ and for each $X \in \Cscr_0$ let $d_{X}\colon I  \to F(X)$ denote the functor
\[
\begin{tikzcd}
I \ar[r]{}{d} \ar[drr] & \PC(F) \ar[r]{}{\tilde{\imath}} & \prod\limits_{X \in \Cscr_0} F(X) \ar[d]{}{\pi_{X}} \\
 & &  F(X)
\end{tikzcd}
\]
where $\tilde{\imath}\colon \PC(F) \to \prod_{X \in \Cscr_0} F(X)$ is defined by forgetting the transition isomorphisms of an object $(A,T_A)$. If each $d_{X}$ has a limit in $F(X)$ and if for all $f \in \Cscr_1$ the functors $F(f)$ preserve the limits $\lim d_{X}$ then $\PC(F)$ admits a limit of $d$.
\end{Theorem}

In conclusion, in this section we have seen a very concrete representation of the pseudolimit of a diagram in $\fCat$ indexed by a pseudofunctor together with an analysis of limits in these pseudolimit categories. 

\section{Tangent Structures and Pseudolimits}\label{Section: Eq Tan Cats}

In this section we begin by recalling the notion of tangent category and the notions of morphisms between them. We will then introduce the notion of a tangent indexing functor and use the remainder of the section to discuss some of its properties. First we show that whenever $F\colon\Cscr^{\op} \to\fCat$ is a tangent indexing functor, its pseudolimit $\PC(F)$ has a tangent structure induced by the tangent structures on the categories $FX$ for the objects $X\in\Cscr^{\op}$. Then we show that our tangent indexing functors are tangent objects in the hom 2-category $\Bicat(\Cscr^{\op}, \fCat)$, as introduced in \cite{Marcello}.
We note here that while the core ideas of this section and the next are fairly straightforward, we do spell out the technical details of the constructions and proofs. We do this in order to support generalizing and extending these ideas (it is easy, if one is not careful, to end up with an arrow to point in the wrong direction), and to indicate how to apply these ideas in various geometric situations.

\subsection{Tangent Categories} 
Let us recall what it means for a category to be a tangent category following \cite{GeoffRobinDiffStruct}. These categories are an abstraction of what it means to have a category and tangent bundle functor, assigning to each object its tangent bundle (which is again an object of the same category), together with all the relations and structure we require/expect of such a bundle object. In what follows below, if $\alpha\colon F \Rightarrow \id_{\Cscr}$ is a natural transformation for an endofunctor $F\colon\Cscr \to \Cscr$, we write $F_2X$ for the pullback
\[
\xymatrix{
F_2X \ar[r]^-{\pi_1} \ar[d]_{\pi_2} \pullbackcorner & FX \ar[d]^{\alpha_X} \\
FX \ar[r]_-{\alpha_X} & X
}
\]
if it exists in $\Cscr$.

\begin{dfn}[{\cite{GeoffRobinDiffStruct}, Definition 2.3}]\label{Defn: Tangent Category}
A category $\Cscr$ has a {\em tangent structure} $\Tbb = (T,p,0,+,\ell,c)$ when:
\begin{enumerate}
	\item $T\colon \Cscr \to \Cscr$ is a functor equipped with a natural transformation $p\colon T \Rightarrow \id_{\Cscr}$ such that for every object $X \in \Cscr_0$ all pullback powers of $p_X\colon TX \to X$ exist and for all $n \in \N$ the functors $T^n$ preserve these pullback powers.
	\item There are natural transformations $+\colon T_2 \Rightarrow T$ and $0\colon \id_{\Cscr} \to T$ for which each map $p_X\colon TX \to X$ describes an additive bundle in $\Cscr$; i.e., $p\colon TX \to X$ is an internal commutative monoid in $\Cscr_{/X}$ with addition and unit given by $+$ and $0$, respectively.
	\item There is a natural transformation $\ell\colon T \Rightarrow T^2$ such that for any $X \in \Cscr_0$, the squares
	\[
	\begin{tikzcd}
	TX \ar[r]{}{\ell_X} \ar[d, swap]{}{p_X} & T^2X \ar[d]{}{Tp_X} & TX \times_{X} TX \ar[d, swap]{}{+_X} \ar[rr]{}{\langle\ell_X\circ \pi_1, \ell_X\circ \pi_2\rangle} & & T^2X \times_{TX} T^2X \ar[d]{}{+_{TX}} & X \ar[d, swap]{}{0_X} \ar[r]{}{0_X} & TX \ar[d]{}{0_{TX}} \\
	X \ar[r, swap]{}{0_X} & TX & TX  \ar[rr, swap]{}{\ell_X} & & T^2X & TX \ar[r, swap]{}{\ell_X} & T^2X
	\end{tikzcd}
	\]
	all commute; i.e., $(\ell_X, 0_X)$ is a morphism of bundles in $\Cscr$ (cf. \cite[Definition 2.2]{GeoffRobinDiffStruct}). 
	\item There is a natural transformation $c\colon T^2 \Rightarrow T^2$ such that for all $X \in \Cscr_0$ the squares
	\[
	\begin{tikzcd}
	T^2X \ar[r]{}{c_X} \ar[d, swap]{}{Tp_X} & T^2X \ar[d]{}{p_{TX}} & T^2X \times_{TX} T^2X \ar[d, swap]{}{(T \ast +)_{X}} \ar[rr]{}{\langle c_X \circ \pi_1, c_X\circ \pi_2\rangle} & & T^2X \times_{TX}T^2X \ar[d]{}{(+ \ast T)_X} & TX \ar[d, swap]{}{(T \ast 0)_X} \ar[r]{}{\id_{TX}} & TX \ar[d]{}{(0 \ast T)_X} \\
	TX \ar[r, swap]{}{\id_{TX}} & TX & T^2X \ar[rr, swap]{}{c_X} & & T^2X & T^2X \ar[r, swap]{}{c_X} & T^2X
	\end{tikzcd}
	\]
	commute, i.e., $(c_X,\id_{TX})$ describes a bundle morphism in $\Cscr$.
	\item We have the equations $c^2 = \id_{\Cscr}, c \circ \ell = \ell$, and the diagrams
	\[
	\begin{tikzcd}
	T \ar[r]{}{\ell} \ar[d, swap]{}{\ell} & T^2 \ar[d]{}{T \ast \ell} & T^3 \ar[d, swap]{}{c \ast T} \ar[r]{}{T \ast c} & T^3 \ar[r]{}{c \ast T} & T^3\ar[d]{}{c \ast T} \\
	T^2 \ar[r, swap]{}{\ell \ast T} & T^3 & T^3 \ar[r, swap]{}{T \ast c} & T^3 \ar[r, swap]{}{c \ast T} & T^3
	\end{tikzcd}
 \]
 \[
	\begin{tikzcd}
	T^2 \ar[d, swap]{}{c} \ar[r]{}{\ell \ast T} & T^3 \ar[r]{}{T \ast c} & T^3 \ar[d]{}{c \ast T} \\
	T^2 \ar[rr, swap]{}{T \ast \ell} & & T^3
	\end{tikzcd}
	\]
	commute in the functor category $[\Cscr,\Cscr]$ where $F \ast \alpha$ and $\alpha \ast F$ denote the post-whiskering of a natural transformation $\alpha$ by a functor $F$ and the pre-whiskering of a natural transformation $\alpha$ by a  functor $F$, respectively.
	\item The diagram
	\[
	\begin{tikzcd}
	T_2X \ar[rrrr]{}{(T \ast +)_X \circ \langle \ell \circ \pi_1, 0_{TX} \circ\pi_2 \rangle} & &	& & T^2X \ar[rrr, shift left = 0.5ex]{}{T(p_X)} \ar[rrr, swap, shift left = -0.5ex]{}{0_X \circ p_X \circ p_{TX}} & & & TX
	\end{tikzcd}
	\]
	is an equalizer in $\Cscr$.
\end{enumerate}
\end{dfn}
\begin{rmk}
	In \cite{GeoffRobinDiffStruct} we find the following useful observations about the parts of Definition \ref{Defn: Tangent Category}:
	\begin{itemize}
		\item Part 1 names $T$ the tangent functor of the tangent structure $\Tbb$ on $\Cscr$ and $p$.
		\item Part 2 describes the object $p_X\colon TX \to X$ as a tangent bundle in $\Cscr$.
		\item Part 3 names $\ell$ as the vertical lift.
		\item Part 4 names $c$ as the canonical flip.
		\item Part 5 gives the coherence relations on $\ell$ and $c$.
		\item Part 6 describes the universality of the vertical lift.
	\end{itemize}
Note also that when we say that $\Cscr$ is a tangent category, we really mean that $(\Cscr,\Tbb)$ is a category equipped with a specific tangent structure $\Tbb$ and have omitted the explicit mention of the tangent structure. In fact, this is an abuse of notation; it is possible for a category to have multiple distinct tangent structures\footnote{The category $\mathbf{SMan}$ of smooth real manifolds is the standard example: there's the identity tangent structure $\mathbb{I}$ and the usual tangent structure defined by taking the bundle functor $T$ to send each manifold to its tangent bundle.}, so when we say $\Cscr$ is a tangent category we mean that $(\Cscr,\Tbb)$ is a category equipped with a specific (potentially unspecified) tangent structure $\Tbb$. 
\end{rmk}
Equally as important to tangent categories are the morphisms of tangent categories. These come in two flavours: one, which is weaker and another which is strong. We will usually work with strong tangent morphisms, but provide the complete definition.

\begin{dfn}[{\cite{GeoffRobinDiffStruct}, Definition 2.7}]\label{Defn: Tangent Morphism}
Let $(\Cscr,\Tbb) = (\Cscr,T,p,0,+,\ell,c)$ and $(\Dscr,\Sbb) = (\Dscr,S,q,0^{\prime},\oplus,\ell^{\prime},c^{\prime})$ be tangent categories. A {\em morphism of tangent categories} is a pair $(F,\alpha)\colon(\Cscr,\Tbb) \to (\Dscr,\Sbb)$ where $F\colon\Cscr \to \Dscr$ is a functor and $\alpha$ is a natural transformation
\[
\begin{tikzcd}
\Cscr \ar[bend left = 30, rr, ""{name = U}]{}{F \circ T} \ar[rr, bend right = 30, swap, ""{name = L}]{}{S \circ F} & & \Dscr \ar[from = U, to = L, Rightarrow, shorten <= 4pt, shorten >= 4pt]{}{\alpha}
\end{tikzcd}
\]
called a {\em lax distributive law}. The 
following diagrams of functors and natural transformations are required to commute:
\[
\begin{tikzcd}
F \circ T \ar[r]{}{\alpha} \ar[dr, swap]{}{F \ast p} & S \circ F \ar[d]{}{q \ast F} \\
 & F
\end{tikzcd}
\begin{tikzcd}
F \ar[r]{}{F \ast 0} \ar[dr, swap]{}{0^{\prime} \ast F} & F \circ T \ar[d]{}{\alpha} \\
 & S \circ F
\end{tikzcd}
\begin{tikzcd}
F \circ T_2 \ar[r]{}{F \ast\alpha_2} \ar[d, swap]{}{F \ast +} & S_2 \circ F \ar[d]{}{\oplus \ast F} \\
F \circ T \ar[r, swap]{}{\alpha} & S \circ F
\end{tikzcd}
\]
\[
\begin{tikzcd}
F \circ T \ar[rr]{}{\alpha} \ar[d, swap]{}{F \ast \ell} & & S \circ F \ar[d]{}{\ell^{\prime} \ast F} \\
F \circ T^2 \ar[rr, swap]{}{(S \ast \alpha) \circ (\alpha \ast T)} & & S^2 \circ F
\end{tikzcd}
\begin{tikzcd}
F \circ T^2 \ar[rr]{}{(S \ast \alpha) \circ (\alpha \ast T)} \ar[d, swap]{}{F \ast c} & & S^2 \circ F \ar[d]{}{c^{\prime} \ast F} \\
F \circ T^2 \ar[rr, swap]{}{(S \ast \alpha) \circ (\alpha \ast T)} & & S^2 \circ F
\end{tikzcd}
\]
Additionally, we say that the morphism $(F,\alpha)$ is {\em strong} if the distributive law $\alpha$ is a natural isomorphism and if $F$ preserves the equalizers and pullbacks of the tangent structure $(\Cscr,\Tbb)$. Finally, we say that $(F,\alpha)$ is \emph{strict} if $\alpha$ is the identity transformation.
\end{dfn}
\begin{rmk}\label{Remark: Colax tangent morphisms}
In \cite{Marcello}, what we have defined as a morphism of tangent categories in Definition \ref{Defn: Tangent Morphism} is called instead a {\em lax} tangent morphism. With this language, a colax morphism of tangent categories $(\Cscr,\Tbb) \to (\Dscr,\Sbb)$ (a colax morphism in \cite{Marcello}) is a pair $(F,\alpha)$ where $F\colon \Cscr \to \Dscr$ is a functor and $\alpha$ is a colax distributive law; i.e., a natural transformation typing as
\[
\begin{tikzcd}
\Cscr \ar[rr, bend left = 20, ""{name = U}]{}{S \circ F} \ar[rr, bend right = 20, swap, ""{name = D}]{}{F \circ T} & & \Dscr
\ar[from = U, to = D, Rightarrow, shorten <= 4pt, shorten >= 4pt]{}{\alpha}
\end{tikzcd}
\]
which satisfies the corresponding analogues of the diagrams defining a morphism of tangent categories in Definition \ref{Defn: Tangent Morphism}. We will use this terminology in Lemma \ref{Lemma: strong lax mor iff strong colax} and in the proof of Theorem \ref{Thm: Indexed tangent category}, but sparingly elsewhere in this paper. 
\end{rmk}

We now present the notion of a tangent natural transformation so that we can define the $2$-category of tangent categories, morphisms of tangent categories, and tangent natural transformations.

\begin{dfn}[\cite{GarnerEmbedding}]\label{Defn: Tangent natural transformation}
A {\em  tangent natural transformation} between morphisms $\rho\colon (F,\alpha)\Rightarrow(G,\beta)\colon (\Cscr,\Tbb) \to (\Dscr,\Sbb)$ is a natural transformation 
\[
\begin{tikzcd}
\Cscr \ar[rr, bend left = 30, ""{name = U}]{}{F} \ar[rr, bend right = 30, swap, ""{name = D}]{}{G} & & \Dscr
\ar[from = U, to = D, Rightarrow, shorten <= 4pt, shorten >= 4pt]{}{\rho}
\end{tikzcd}
\]
for which the equation
\[
(\id_{S} \ast \rho) \circ \alpha = \beta \circ (\rho \ast \id_T)
\]
holds.
\end{dfn}
\begin{dfn}\label{Defn: Two cat of tangent cats}
The $2$-category $\fTan$ of tangent categories is defined as follows: the objects are tangent categories, the morphisms are tangent category morphisms, and the $2$-cells are tangent natural transformations with compositions as in $\fCat$. The full sub-$2$-category of $\fTan$ generated by the strong (lax directional) morphisms of tangent categories will be denoted by $\fTan_{\operatorname{strong}}$.
\end{dfn}

\begin{rmk}
    We will use the terms ``strong lax'' and ``strong colax'' morphism for strong morphisms $(F,\alpha)$ when we want to emphasize the direction of the distributive law natural isomorphism $\alpha$ as indicated in Remark \ref{Remark: Colax tangent morphisms}. By default we will give strong morphisms with the distributive law in the lax direction, but we will have reason to consider a strong colax morphism in Lemma \ref{Lemma: strong lax mor iff strong colax}.
\end{rmk}

\begin{rmk}\label{Remark: Composition in Tan}
It is worth noting that the composition of tangent morphisms in $\fTan$ requires a composition of distributive laws. It is explicitly computed as follows. If $(F,\alpha)\colon(\Cscr,\Tbb) \to (\Dscr,\Sbb)$ and $(G,\beta)\colon(\Dscr,\Sbb) \to (\Escr,\Rbb)$ are tangent morphisms then their composite $(G, \beta) \circ (F,\alpha)\colon(\Cscr,\Tbb) \to (\Escr,\Rbb)$ is computed by taking the underlying functor to be $G \circ F$ and defining the distributive law $\gamma_{G \circ F}$ to be the natural transformation
\[
G \circ F \circ T \xrightarrow{\id_G \ast \alpha} G \circ S \circ F \xrightarrow{\beta \ast \id_F} R \circ G \circ F.
\]
\end{rmk}

\subsection{The Tangent Structure on Pseudolimits}
We now provide our first new key definition: that of a tangent indexing pseudofunctor.

\begin{dfn}\label{Defn: tangent pre-equivariant indexing functor}
We say that a pseudofunctor $F\colon \Cscr^{\op} \to \fCat$ is a {\em tangent indexing functor} if: 
\begin{enumerate}
	\item For each $X \in \Cscr_0$, $F(X)$ is a tangent category; i.e., each category $F(X)$ comes equipped with some specific tangent structure $\!\quot{\Tbb}{X} = (\!\quot{T}{X}, \quot{p}{X}, \quot{0}{X}, \quot{+}{X}, \quot{\ell}{X}, \quot{c}{X})$.
	\item For each $f\colon X \to Y$ in $\Cscr_1$, there is a natural isomorphism
	\[
	\begin{tikzcd}
	F(Y) \ar[rr,bend left = 20, ""{name = U}]{}{F(f) \circ\! \quot{T}{Y}} \ar[rr, swap, bend right = 20, ""{name = L}]{}{\!\quot{T}{X} \circ F(f)} & & F(X) \ar[from = U, to = L, Rightarrow, shorten <= 4pt, shorten >= 4pt]{}{\!\quot{T}{f}} \ar[from = U, to = L, swap, Rightarrow, shorten <= 4pt, shorten >= 4pt]{}{\cong}
	\end{tikzcd}
	\]
	such that the pair $(F(f),\!\quot{T}{f})\colon (F(Y),\!\quot{\Tbb}{Y}) \to (F(X),\!\quot{\Tbb}{X})$ is a strong morphism of tangent categories.
	\item The natural isomorphisms $\!\quot{T}{f}^{-1}$ vary pseudonaturally in the sense that the collection of functors $\quot{T}{X}$ indexed by $X \in \Cscr_0$ and of natural transformations $\!\quot{T}{f}^{-1}$ indexed by $f \in \Cscr_1$ form a pseudonatural transformation $T$:
	\[
	\begin{tikzcd}
	\Cscr^{\op} \ar[rr, bend left = 30, ""{name = U}]{}{F} \ar[rr, bend right = 30, swap, ""{name = L}]{}{F} & & \fCat \ar[from = U, to = L, Rightarrow, shorten <= 4pt, shorten >= 4pt]{}{T}
	\end{tikzcd}
	\]
\end{enumerate}
\end{dfn}

We now show that to give a tangent indexing functor is precisely to give a pseudofunctor into $\fTan_{\operatorname{strong}}$, as this gives another perspective on the information encapsulated by Definition \ref{Defn: tangent pre-equivariant indexing functor}.

\begin{prop}\label{Prop: Tangent indexing functor is a pseudofunctor into Tanstrong}
To give a tangent indexing functor $F\colon\Cscr^{\op} \to \fCat$ it is necessary and sufficient to define a pseudofunctor $\underline{F}\colon\Cscr^{\op} \to \fTan_{\operatorname{strong}}$.
\end{prop}
\begin{proof}
On one hand assume that we have a tangent indexing functor $F\colon\Cscr^{\op} \to \fCat$. We define the pseudofunctor $\underline{F}\colon\Cscr^{\op} \to \fTan_{\operatorname{strong}}$ as follows:
\begin{itemize}
    \item For objects $X \in \Cscr_0$, we define $\underline{F}(X) := (F(X),\quot{\Tbb}{X})$.
    \item For morphisms $f\colon X \to Y$ in $\Cscr_1$, we define $\underline{F}(f) := (F(f), \quot{T}{f})$.
    \item For composable pairs of morphisms $f\colon X \to Y$ and $g\colon Y \to Z$ in $\Cscr$, we define the compositor $\quot{\phi_{f,g}}{\underline{F}}\colon \underline{F}(f) \circ \underline{F}(g) \Rightarrow \underline{F}(g \circ f)$ by setting $\quot{\phi_{f,g}}{\underline{F}} := \quot{\phi_{f,g}}{F}$.
\end{itemize}
Because we already know that $F$ is a pseudofunctor and that the transformations $\quot{T}{f}^{-1}$ vary pseudonaturally in $\Cscr^{\op}$, to deduce that $\underline{F}$ is a pseudofunctor it suffices to prove that $\quot{\phi_{f,g}}{\underline{F}}$ is a tangent transformation, i.e., that the diagram
\[
\begin{tikzcd}
(F(f) \circ F(g)) \circ \quot{T}{Z} \ar[rrrr]{}{(\quot{T}{f} \ast \id_{F(f)}) \circ (\id_{F(f)} \ast \quot{T}{g})} \ar[d, swap]{}{\quot{\phi_{f,g}}{\underline{F}} \ast \quot{T}{Z}} & & & & \quot{T}{X} \circ (F(f) \circ F(g)) \ar[d]{}{\quot{T}{X} \ast \quot{\phi_{f,g}}{\underline{F}}} \\
F(g \circ f) \circ \quot{T}{Z} \ar[rrrr, swap]{}{\quot{T}{g \circ f}} & & & & \quot{T}{X} \circ F(g \circ f)
\end{tikzcd}
\]
commutes. However, the commutativity of this diagram exactly follows from the pseudonaturality of $T\colon F \Rightarrow F$.

On the other hand, if we have a pseudofunctor $\underline{F}\colon\Cscr^{\op} \to \fTan_{\operatorname{strong}}$ then a straightforward but tedious check shows that the assignment $F\colon \Cscr^{\op} \to \fCat$ given by, for the forgetful $2$-functor $\Forget\colon\fTan_{\operatorname{strong}} \to \fCat$,
\begin{itemize}
    \item For objects $X \in \Cscr_0$, $F(X) := \Forget(\underline{F}(X))$.
    \item For morphisms $f\colon X \to Y$ in $\Cscr_1$, $F(f) := \Forget(\underline{F}(f))$.
    \item For any composable pair of morphisms $f\colon X \to Y$ and $g\colon Y \to Z$ in $\Cscr$, $\quot{\phi_{f,g}}{F} := \quot{\phi_{f,g}}{\underline{F}}$
\end{itemize}
is a pseudofunctor. Furthermore, if we have a morphism $f \in \Cscr_1$ and write each distrubtive law of the tangent morphism $\underline{F}(f)$ as $\alpha_f$, then the pseudofunctoriality of $\underline{F}$ and the colax direction of the transformations $\alpha_f$ implies that the assignments
\begin{itemize}
    \item For objects $X \in \Cscr_0$, define the functor $\quot{T}{X}\colon F(X) \to F(X)$ as the tangent functor of the tangent category $\underline{F}(X)$;
    \item For morphisms $f \in \Cscr_1$, define the natural isomorphism $\quot{T}{f} := \alpha_f^{-1}$;
\end{itemize}
define a pseudonatural transformation $T\colon F \Rightarrow F$. Thus $F$, together with the pseudonatural transformation $T$, form a tangent indexing functor.
\end{proof}

\begin{example}
As observed in \cite[Page 10]{GeoffRobinDiffStruct}, every category $\Cscr$ has a trivial tangent structure $$\Ibb_{\Cscr} = (\id_{\Cscr}, \iota_{\id}, \iota_{\id}, \iota_{\id}, \iota_{\id}, \iota_{\id})$$ where $\iota_{\id}$ is the identity natural isomorphism on the identity functor $\id_{\Cscr}$. Each functor $F\colon \Cscr \to \Dscr$ is a strong tangent functor $(F,\iota_F)\colon (\Cscr,\Ibb_{\Cscr}) \to (\Dscr,\Ibb_{\Dscr})$.\footnote{This becomes immediate when you notice that the tangent pullbacks are pullbacks of the identity functor over the identity morphism at an object $X$.} Because of this, any pseudofunctor $F\colon \Cscr^{\op} \to \fCat$ is a tangent indexing functor when each category $F(X)$ is equipped with the tangent structure $\Ibb_{F(X)}$ and the natural isomorphisms $\quot{T}{f}$ are defined to be $\!\quot{T}{f} := \iota_{F(f)}$.
\end{example}

Given a tangent indexing functor $F\colon \Cscr^{\op} \to \fCat$ as in Definition \ref{Defn: tangent pre-equivariant indexing functor} we can give its category of cones, $\PC(F)$, the structure of a tangent category; i.e., we will  define the tangent functor $T\colon \PC(F) \to \PC(F)$ as well as the bundle, zero, addition, vertical lift, and canonical flip natural transformations by making use of Theorem \ref{Theorem: Defining functors and nat transforms between equivariant categories and functors}. We will show that the pair $T = (\!\quot{T}{X},\!\quot{T}{f}^{-1})$ forms a pseudonatural transformation $T\colon F \Rightarrow F$ and then show that the various natural constructions give rise to modifications between the various pullback and compositional powers of $T$ (after we prove that pullback powers of $T$ exist, of course).

\begin{lem}\label{Lemma: Constructing the tangent functor}
If $F\colon \Cscr^{\op} \to \fCat$ is a tangent indexing functor then there is a pseudonatural transformation $T\colon F \Rightarrow F$ whose object-local components are given by the tangent functors in the tangent structures $(F(X),\quot{\Tbb}{X})$. In particular, $T\colon F \Rightarrow F$ determines an endofunctor
\[
\underline{T}\colon \PC(F) \to \PC(F).
\]
\end{lem}

\begin{proof}
The first part of the statement of the lemma is axiomatized, and hence holds by Condition (3) of Definition \ref{Defn: tangent pre-equivariant indexing functor}. The second statement follows from Theorem \ref{Theorem: Defining functors and nat transforms between equivariant categories and functors}.
\end{proof}

\begin{rmk}\label{Rmk: Explicit form of transition isomorphisms}
In Definition \ref{Defn: tangent pre-equivariant indexing functor}, when we ask that the data $(\!\quot{T}{X}, \!\quot{T}{f}^{-1})$ constitute a pseudonatural transformation $T\colon F\Rightarrow F$, in light of Lemma \ref{Lemma: Constructing the tangent functor} we are asking that the transition isomorphisms on an object $\underline{T}A$ take the form
\[
\Sigma_{\underline{T}A} = \left\lbrace \!\quot{T}{X}\left(\tau_f^{A}\right) \circ\! \quot{T}{f}_{\!\quot{A}{Y}} \; : \; f \in \Cscr_1, \tau_f^A \in \Sigma_A \right\rbrace.
\]
In particular, this means that the transition isomorphisms are given as a composition of the tangent functor as applied to the transition isomorphisms an object starts with, pre-composed with the tangent morphism's naturality witness.
\end{rmk}

\begin{lem}\label{Lemma: Construction of Bundle Map}
Let $F\colon \Cscr^{\op} \to \fCat$ be a tangent indexing functor. There is a modification $p\colon T \Rrightarrow \iota_F$ as in the diagram
\[
\begin{tikzcd}
\Cscr^{\op} \ar[rrr, bend left = 30, ""{name = U}]{}{F} \ar[rrr, bend right = 30, swap, ""{name = B}]{}{F} & & &\fCat \ar[from = U, to = B, Rightarrow, shorten <= 4pt, shorten >= 4pt, bend right = 30, swap, ""{name = L}]{}{T} \ar[from = U, to = B, Rightarrow, shorten <= 4pt, shorten >= 4pt, bend left = 30, ""{name = R}]{}{\iota_F} \ar[from = L, to = R, symbol = \underset{p}{\Rrightarrow}, swap]
\end{tikzcd}
\] 
where the natural transformations $p_{X}$ of the modification are all given by
\[
p_{X} := \!\quot{p}{X}\colon \!\quot{T}{X} \to \id_{F(X)}
\]
where $\!\quot{p}{X}$ is the bundle transformation in the tangent structure $(F(X),\!\quot{\Tbb}{X})$.
\end{lem}
\begin{proof}
To prove that $p$ is a modification it suffices to verify that the diagram of functors and natural transformations
\[
\begin{tikzcd}
F(f) \circ \!\quot{T}{Y} \ar[rr]{}{\iota_{F(f)}\ast \quot{p}{Y}} \ar[d,swap]{}{\quot{T}{f}} & & F(f) \circ \id_{F(Y)} \ar[r, equals] \ar[d] & F(f) \ar[d, equals] \\
\!\quot{T}{X} \circ F(f) \ar[rr, swap]{}{\quot{p}{X}\ast \iota_{F(f)}} & & \id_{F(X)} \circ F(f) \ar[r, equals] & F(f) 
\end{tikzcd}
\]
commutes for any $f\colon X \to Y$ in $\Cscr_1$. However, this is equivalent to asking that the diagram of functors and natural transformations
\[
\begin{tikzcd}
F(f) \circ \!\quot{T}{Y} \ar[dr, swap]{}{F(f) \ast \!\quot{\!p}{Y}} \ar[r]{}{\quot{T}{f}} & \!\quot{T}{X} \circ F(f) \ar[d]{}{\quot{\!p}{X} \ast F(f)} \\
& F(f)
\end{tikzcd}
\]
commutes, which follows from the fact that $(F(f),\!\quot{T}{f})$ is a (strong) tangent morphism. Thus $p\colon T \Rrightarrow \iota_{F}$ is a modification.
\end{proof}
\begin{cor}\label{Cor: Bundle transformation equivariant tan functor}
If $F$ is a tangent indexing functor then there is a natural transformation $$\underline{p}\colon \underline{T} \to \id_{\PC(F)}\colon \PC(F) \to \PC(F).$$ 
\end{cor}
\begin{lem}\label{Lemma: Zero modification}
Let $F\colon \Cscr^{\op} \to \fCat$ be a tangent indexing functor. Then there is a modification $0\colon \iota_F \Rrightarrow T$ as in the diagram
\[
\begin{tikzcd}
\Cscr^{\op} \ar[rrr, bend left = 30, ""{name = U}]{}{F} \ar[rrr, bend right = 30, swap, ""{name = B}]{}{F} & & &\fCat \ar[from = U, to = B, Rightarrow, shorten <= 4pt, shorten >= 4pt, bend right = 30, swap, ""{name = L}]{}{\iota_F} \ar[from = U, to = B, Rightarrow, shorten <= 4pt, shorten >= 4pt, bend left = 30, ""{name = R}]{}{T} \ar[from = L, to = R, symbol = \underset{0}{\Rrightarrow}, swap]
\end{tikzcd}
\]
whose object-local natural transformations are given by 
\[
\!\quot{0}{X}\colon\id_{F(X)} \Rightarrow \!\quot{T}{X}
\]
where $\!\quot{0}{X}$ is the zero of the tangent structure $(F(X),\!\quot{\Tbb}{X})$.
\end{lem}
\begin{proof}
To prove that $0$ is a modification it suffices to verify that the diagram of functors and natural transformations
\[
\xymatrix{
F(f) \ar@{=}[r] \ar@{=}[d] & F(f) \circ \id_{F(Y)} \ar[r]^-{F(f) \ast \quot{0}{Y}} \ar@{=}[d] & F(f) \circ\! \quot{T}{Y} \ar[d]^{\!\quot{T}{f}} \\
F(f) \ar@{=}[r] & \id_{F(X)} \circ F(f) \ar[r]_-{\!\quot{0}{X} \ast F(f)} & \!\quot{T}{X} \circ F(f)
}
\]
commutes for any $f\colon X \to Y$ in $\Cscr_1$. However, this is equivalent to asking that the diagram of functors and natural transformations
\[
\xymatrix{
F(f) \ar[drr]_{\!\quot{0}{X} \ast F(f)} \ar[rr]^-{F(f) \ast \!\quot{0}{Y}} & & F(f) \circ \!\quot{T}{Y} \ar[d]^{\quot{T}{f}} \\
& & \!\quot{T}{X} \circ F(f)
}
\]
commutes, which follows from $(F(f),\!\quot{T}{f})$ being a strong tangent morphism.
\end{proof}

\begin{cor}\label{Cor: Construction of equivariant zero transformation}
If $F$ is a tangent indexing functor then there is a natural transformation:
\[
\begin{tikzcd}
\PC(F) \ar[rr, bend left = 30, ""{name = Up}]{}{\id_{\PC(F)}} \ar[rr, bend right = 30, swap, ""{name = Down}]{}{\underline{T}} & & \PC(F) \ar[from = Up, to = Down, Rightarrow, shorten <= 4pt, shorten >= 4pt]{}{\underline{0}}
\end{tikzcd}
\]
\end{cor}


Recall that it is a centrally important aspect of tangent categories that we have a notion of additive tangent bundles. As such, we will require a short discussion of finite pullback powers of the tangent bundle $\underline{p}\colon \underline{T} \Rightarrow \id_{\PC(F)}$ in $\PC(F)$. Perhaps unsurprisingly, this will make use of Theorem \ref{Thm: Limits in pseudcone categories} above. 

\begin{lem}\label{Lemma: Existence of tangent pullbacks in FGX}
Let $F\colon \Cscr^{\op} \to \fCat$ be a tangent indexing functor and let $A \in \PC(F)_0$. Then for any $n \in \N$ the pullback $T_nA$ exists in $\PC(F)$.
\end{lem}
\begin{proof}
Since $T_0A = A$ and $T_1A = p_A\colon TA \to A$, it suffices to show that $T_2A$ exists. For this consider the cospan
\[
TA \xrightarrow{\underline{p}_A} A \xleftarrow{\underline{p}_A} TA
\]
in $\PC(F)$. At each $X \in \Cscr_0$ this gives the cospan
\[
\!\quot{T}{X}\left(\!\quot{A}{X}\right) \xrightarrow{\quot{\! p}{X}_{\!\quot{A}{X}}} \!\quot{A}{X} \xleftarrow{\!\quot{p}{X}_{\!\quot{A}{X}}} \!\quot{T}{X}\left(\!\quot{A}{X}\right)
\]
in the tangent category $(F(X),\!\quot{\Tbb}{X})$. Because $F(X)$ is a tangent category with tangent structure $\!\quot{\Tbb}{X}$, we have that $\!\quot{T}{X}_2\left(\!\quot{A}{X}\right)$ exists in $F(X)$ for all $X \in \Cscr_0$. Furthermore, because each functor $F(f)$ is the functor component of a strong tangent morphism $(F(f),\!\quot{T}{f})$, each functor $F(f)$ preserves each tangent pullback. It thus follows from Theorem \ref{Thm: Limits in pseudcone categories} that the pullback
\[
\xymatrix{
T_2A \ar[r] \ar[d] \pullbackcorner & TA \ar[d]^{\underline{p}_A} \\
TA \ar[r]_-{\underline{p}_A} & A
}
\]
exists.
\end{proof}

We now prove that the pullback powers of the tangent functor $\underline{T}_n$ on $\PC(F)$ arise themselves as pseudocone functors coming from pseudonatural transformations. While not strictly speaking necessary for the existence of the pullbacks $T_2$, this shows how the natural transformations $\!\quot{T_2}{f}$ connecting $F(f) \circ \!\quot{T_2}{Y} \Rightarrow \!\quot{T_2}{X} \circ F(f)$ for any $f\colon X \to Y$ arise and are constructed. In particular, it shows that we can use Theorem \ref{Theorem: Defining functors and nat transforms between equivariant categories and functors} to define our addition natural transformation.

\begin{prop}\label{Prop: Existence of pullback pseudonatural transformation}
Let $F\colon \Cscr^{\op} \to \fCat$ be a tangent indexing functor. There is a pseudonatural transformation $T_2$
\[
\begin{tikzcd}
\Cscr^{\op} \ar[rr, bend left = 30, ""{name = U}]{}{F} \ar[rr, bend right = 30, swap, ""{name = L}]{}{F} & & \fCat \ar[from = U, to = L, Rightarrow, shorten <= 4pt, shorten >= 4pt]{}{T_2}
\end{tikzcd}
\]
whose functors $\!\quot{T_2}{X}$ are given by sending an object $A \in F(X)_0$ to the pullback $\!\quot{T_2}{X}(A)$ of the cospan $\!\quot{T}{X}(A) \xrightarrow{\quot{p}{X}_A} A \xleftarrow{\quot{p}{X}_A} \!\quot{T}{X}(A)$.
\end{prop}
\begin{proof}
We need to construct the natural isomorphism $\!\quot{T_2}{f}\colon F(f) \circ \!\quot{T_2}{Y} \xRightarrow{\cong}\! \quot{T_2}{X} \circ F(f)$. For this let $f\colon X \to Y$ be a morphism in $\Cscr_1$ and let $A \in F(Y)_0$. Consider the cube:
\[
\begin{tikzcd}
F(f)\left(\!\quot{T_2}{Y}(A)\right) \ar[rr]{}{F(f)(\!\quot{\pi_2}{Y})} \ar[dd, swap]{}{F(f)(\!\quot{\pi_1}{Y})} \ar[dr, dashed]{}{\exists!} & & F(f)\left(\!\quot{T}{Y}(A)\right) \ar[dr]{}{\quot{T}{f}_A} \ar[dd, swap, near start]{}{(F(f) \ast \!\quot{p}{Y})_A} \\
 & \!\quot{T_2}{X}\left(F(f)(A)\right) & & \!\quot{T}{X}\left(F(f)(A)\right) \ar[dd]{}{(\!\quot{\!p}{X} \ast F(f))_A} \\
F(f)\left(\!\quot{T}{Y}(A)\right) \ar[rr, near end, swap]{}{(F(f) \ast \!\quot{p}{Y})_A} \ar[dr, swap]{}{\quot{T}{f}_A} & & F(f)(A) \ar[dr, equals]\\
 & \!\quot{T}{X}\left(F(f)(A)\right) \ar[rr, swap]{}{(\!\quot{\!p}{X} \ast F(f))_{A}} & & F(f)(A)
 \ar[from = 2-2, to = 2-4, crossing over, near start, swap]{}{\!\quot{\pi_2}{X}} \ar[from = 2-2, to = 4-2, crossing over, near start]{}{\quot{\!\pi_1}{X}}
\end{tikzcd}
\]
We define $(\!\quot{T_2}{f})_A$ to be the unique dotted arrow which renders the cube above commutative and define $\quot{T_2}{f} := \lbrace (\!\quot{T_2}{f})_A \; \left. \right| \; A \in F(Y)_0\rbrace$. That $\!\quot{T_2}{f}$ is a natural transformation is routinely checked by using that if $\rho \in F(X)(A,B)$ then the morphism $\!\quot{T_2}{X}(\rho)$ is the unique map making the cube
\[
\begin{tikzcd}
\!\quot{\!T_2}{X}(A) \ar[rr]{}{\!\quot{\!\pi_2^{A}}{X}} \ar[dd, swap]{}{\!\quot{\!\pi_1^A}{X}} \ar[dr, dashed]{}{\exists!} & & \!\quot{\!T}{X}(A) \ar[dr]{}{\quot{T}{X}(\rho)} \ar[dd, swap, near start]{}{\quot{p}{X}_A} \\
& \!\quot{\!T_2}{X}(B) & & \!\quot{\!T}{X}(B) \ar[dd]{}{\!\quot{\!p}{X}_B} \\
\!\quot{\!T}{X}(A) \ar[rr, near end, swap]{}{\!\quot{\!p}{X}_A} \ar[dr, swap]{}{\!\quot{\!T}{X}(\rho)} & & A \ar[dr]{}{\rho}\\
& \!\quot{\!T}{X}(B) \ar[rr, swap]{}{\!\quot{\!p}{X}_{B}} & & B
\ar[from = 2-2, to = 2-4, crossing over, near start, swap]{}{\!\quot{\!\pi_2^B}{X}} \ar[from = 2-2, to = 4-2, crossing over, near start]{}{\!\quot{\!\pi_1^B}{X}}
\end{tikzcd}
\]
commute. 
We also note that $\!\quot{T_2}{f}$ is a natural isomorphism because the functor $F(f)$, as a strong tangent morphism, preserves tangent pullbacks and the maps which constitute $\!\quot{T_2}{f}$ witness this.
Finally the pseudonaturality of the collection $(\!\quot{\!T_2}{X},\!\quot{T_2}{f})$ follows from the fact that the object functors $\!\quot{\!T_2}{X}$ are defined by taking pullbacks of the pseudonatural transformation $T$ while the natural isomorphisms $\!\quot{T_2}{f}$ are defined by the limit preservation isomorphisms of $F(f)$.
\end{proof}
\begin{cor}\label{Cor: Pullback power of tangent functor modification arbitrary powers}
For any $n \in \N$ there is a pseudonatural transformation
\[
\begin{tikzcd}
\Cscr^{\op} \ar[rr, bend left = 30, ""{name = U}]{}{F} \ar[rr, bend right = 30, swap, ""{name = L}]{}{F} & & \fCat \ar[from = U, to = L, Rightarrow, shorten <= 4pt, shorten >= 4pt]{}{T_n}
\end{tikzcd}
\]
whose underlying fibre functors $\quot{T_n}{X}$ are the $n$-fold tangent pullback powers.
\end{cor}
\begin{proof}
This is a routine induction with $T_0A := A$, $T_1A := TA$, and taking $T_2A$ as the base case for an induction involving finite pullback powers.
\end{proof}
\begin{cor}\label{Cor: Tangent Pullback functors as equivariant functors}
If $F\colon \Cscr^{\op} \to \fCat$ is a tangent indexing functor then the $n$-fold tangent pullback functor $\underline{T_n}\colon \Cscr^{\op} \to \fCat$ arises as the functor associated to the pseudonatural transformation $T_n:F \Rightarrow F$. In particular, $[\underline{T}]_2 = \underline{T_2}$.
\end{cor}
\begin{proof}
The existence of $\underline{T_n}$ for all $n \in \N$ follows from applying Theorem \ref{Theorem: Defining functors and nat transforms between equivariant categories and functors} to Corollary \ref{Cor: Pullback power of tangent functor modification arbitrary powers}. Finally, note that if $(A,T_A)$ is an object in $\PC(F)$ then $\underline{T_2}A = \underline{T}_2A$. Thus, the fact that $[\underline{T}]_2 = \underline{T_2}$ comes down to noticing that for any $f\colon X \to Y$ in $\Cscr_1$, the naturality isomorphisms
\[
\begin{tikzcd}
F(Y) \ar[rr, bend left = 30, ""{name = U}]{}{F(f) \circ \quot{T_2}{Y}} \ar[rr, bend right = 30, swap, ""{name = L}]{}{\quot{T_2}{X} \circ F(f)} & & F(X) \ar[from = U, to = L, Rightarrow, shorten <= 4pt, shorten >= 4pt]{}{\quot{T_2}{f}}
\end{tikzcd}
\] 
of the pseudonatural transformation $T_2$ are exactly the limit preservation isomorphisms induced by $F(f)$ preserving tangent pullbacks. As such the transition isomorphisms of $T_{\underline{T_2}A}$ coincide with the transition isomorphisms of $T_{\underline{T}_2A}$ and so $\underline{T}_2 = \underline{T_2}$.
\end{proof}
\begin{cor}\label{Cor: Tangent powers preserve pullback functors}
Let $n, m \in \N$ and let $F$ be a tangent indexing functor on $X$. Then $\underline{T}^n \circ \underline{T}_m \cong \underline{T}_m \circ \underline{T}^n,$ i.e., powers of $\underline{T}$ preserve all tangent pullbacks $\underline{T}_m$.
\end{cor}
\begin{proof}
This follows from the fact that object-locally we have natural isomorphisms 
\[
\!\quot{\!T^n}{X} \circ \!\quot{\!T_m}{X} \cong \!\quot{\!T_m}{X} \circ \!\quot{\!T^n}{X}
\] 
by virtue of $(F(X),\!\quot{\Tbb}{X})$ being a tangent category and each $\!\quot{\!T}{X}$ being a tangent functor. From here checking the modification axioms for these natural isomorphisms is a straightforward but tedious process we omit here.
\end{proof}

We now use the tangent pullback functor $\underline{T_2} = \underline{T}_2$ we just built in order to prove that we can add pseudolimit bundles by making sure that the object-local pieces of the bundle addition on $\underline{T}_2A$ are given by the bundle addition $\quot{+}{X}\colon \!\quot{T_2}{X}(\!\quot{A}{X}) \to \!\quot{T}{X}(\!\quot{A}{X})$.
\begin{lem}\label{Lemma: Addition modification}
Let $F\colon \Cscr^{\op} \to \fCat$. Then there is an addition modification $+\colon T_2 \Rrightarrow T$ as in the diagram
\[
\begin{tikzcd}
\Cscr^{\op} \ar[rrr, bend left = 30, ""{name = U}]{}{F} \ar[rrr, bend right = 30, swap, ""{name = B}]{}{F} & & &\fCat \ar[from = U, to = B, Rightarrow, shorten <= 4pt, shorten >= 4pt, bend right = 30, swap, ""{name = L}]{}{T_2} \ar[from = U, to = B, Rightarrow, shorten <= 4pt, shorten >= 4pt, bend left = 30, ""{name = R}]{}{T} \ar[from = L, to = R, symbol = \underset{+}{\Rrightarrow}, swap]
\end{tikzcd}
\]
whose natural transformations are given by 
\[
\quot{+}{X}\colon \!\quot{\!T_2}{X} \to \!\quot{\!T}{X}
\]
where $\quot{+}{X}$ is the bundle addition transformation in the tangent structure $(F(X),\!\quot{\Tbb}{X})$.
\end{lem}
\begin{proof}
To prove that $+$ is a modification it suffices to prove that the diagram of functors and natural transformations
\[
\xymatrix{
F(f) \circ \!\quot{\!T_2}{Y} \ar[d]_{\!\quot{T_2}{f}} \ar[rr]^-{F(f) \ast \quot{\!+}{Y}} & & F(f) \circ \!\quot{\!T}{Y} \ar[d]^{\!\quot{T}{f}} \\
\!\quot{\!T_2}{X} \circ F(f) \ar[rr]_-{\quot{+}{X} \ast F(f)} & & \!\quot{\!T}{X} \circ F(f)
}
\]
commutes. However, this follows immediately because $(F(f),\!\quot{T}{f})$ is a strong tangent morphism.
\end{proof}

\begin{cor}\label{Cor: Bundle addition transformation equivariant category}
For any tangent indexing functor $F$ on $X$ there is an addition natural transformation $\underline{+}\colon \underline{T}_2 \Rightarrow \underline{T}$.
\end{cor}
\begin{proof}
Apply Theorem \ref{Theorem: Defining functors and nat transforms between equivariant categories and functors} to the modification $+\colon T_2 \Rrightarrow T$.
\end{proof}
\begin{prop}\label{Prop: Equivariant Tangent over A is Additive Bundle over A}
Let $F$ be a tangent indexing functor 
on $X$. For any $A \in \PC(F)_0$, the tangent bundle $\underline{T}A$ is an additive bundle in $\PC(F)$, i.e., \[
\left(\underline{p}_A\colon \underline{T}A \to A,\underline{+}_A\colon \underline{T}_2A \to \underline{T}A, \underline{0}_A\colon A \to \underline{T}A\right)
\] 
is a commutative monoid in $\PC(F)_{/A}$.
\end{prop}
\begin{proof}
It suffices to show that $\underline{0}_A$ is a left-and-right unit for $\underline{p}_A\colon \underline{T}A \to A$ and that bundle addition $\underline{+}_A$ is commutative and associative over $A$. However, checking that the diagrams
\[
\begin{tikzcd}
\underline{T}A \ar[r]{}{\cong} \ar[d, swap]{}{\cong} \ar[drrr, equals] & \underline{T}A \times_A A \ar[rr]{}{\id_{\underline{T}A} \times \underline{0}_A} & & \underline{T}_2A \ar[d]{}{\underline{+}_A} \\
A \times_A \underline{T}A \ar[rr, swap]{}{\underline{0}_A \times \id_{\underline{T}A}} & & \underline{T}_2A \ar[r, swap]{}{\underline{+}_A} & \underline{T}A
\end{tikzcd}
\begin{tikzcd}
\underline{T}_2A \ar[r]{}{s} \ar[dr, swap]{}{\underline{+}_A} & \underline{T}_2A \ar[d]{}{\underline{+}_A} \\
 & \underline{T}A
\end{tikzcd}
\]
\[
\begin{tikzcd}
(\underline{T}A \times_A \underline{T}A) \times \underline{T}A \ar[rr]{}{\underline{+}_A \times \id_{\underline{T}A}} \ar[d, swap]{}{\cong} & & \underline{T}_2A \ar[d]{}{\underline{+}_A} \\
\underline{T}A \times_A (\underline{T}A \times_A \underline{T}A) \ar[dr, swap]{}{\id_{\underline{T}A} \times \underline{+}A} & & \underline{T}A \\
 & \underline{T}_2A \ar[ur, swap]{}{\underline{+}_A}
\end{tikzcd}
\]
all hold, where $s$ is the product-swapping map, may be done object-locally, i.e., it may be checked over each category $F(X)$. However, because each category $(F(X), \quot{\Tbb}{X})$ is a tangent category we have that $\quot{T}{X}(\quot{A}{X})$ is an additive bundle over $\quot{A}{X}$ and so each desired object-local diagram commutes. Thus $\underline{T}A$ is an additive $A$-bundle.
\end{proof}

We now recall a basic construction in $2$-category and bicategory theory in order to better understand how to define and work with  the composite $\underline{T}^2$ directly from its descent incarnation from the vertical composition $T^2 = T \circ T$ of pseudonatural transformations. If $F,G,H\colon \Cfrak \to \Dfrak$ are two pseudofunctors between $2$-categories and if we have pseudonatural transformations as in the diagram
\[
\begin{tikzcd}
\Cfrak \ar[rr, bend left = 50, ""{name = U}]{}{F} \ar[rr, ""{name = M}]{}[description]{G} \ar[rr, bend right = 50, swap, ""{name = B}]{}{H} & & \Dfrak \ar[from = U, to = M, Rightarrow, shorten <= 4pt]{}{\alpha} \ar[from = M, to = B, Rightarrow, shorten <= 4pt, shorten >= 4pt]{}{\beta}
\end{tikzcd}
\]
then the vertical composite $\beta \circ \alpha\colon F \Rightarrow H\colon \Cfrak \to \Dfrak$ has object functors defined by, for all $X \in \Cfrak_0$
\[
(\beta \circ \alpha)_X = \beta_X \circ \alpha_X,
\]
while it has morphism natural transformations defined by, for all $f\colon X \to Y \in \Cfrak_1$:
\[
\begin{tikzcd}
(\beta \circ \alpha)_f = (\beta_Y \ast \alpha_f) \circ (\beta_f \ast \alpha_X).
\end{tikzcd}
\]
From this we get that $\underline{T}^2\colon \PC(F) \to \PC(F)$ arises as the functor associated to the pseudonatural transformation $T^2 = T \circ T\colon F \Rightarrow F$ described by fibre functors
\[
\quot{(T^2)}{X} := (\quot{T}{X})^2
\]
and fibre transformations
\[
\quot{T^2}{f} := (\quot{T}{X} \ast \quot{T}{f}) \circ (\quot{T}{f} \ast \quot{T}{Y})
\]
for any $f\colon X \to Y$ in $\Cscr_1$. We will use this observation below in proving the existence of both the verical lift $\underline{\ell}\colon \underline{T} \Rightarrow \underline{T}^2$ and the canonical flip $\underline{c}\colon \underline{T}^2 \Rightarrow \underline{T}^2$.

\begin{lem}\label{Lemma: Vertical lift modification}
Let $F$ be a tangent indexing functor on $X$. There is then a vertical lift modification $\ell\colon T \Rrightarrow T^2$ as in the diagram
\[
\begin{tikzcd}
\Cscr^{\op} \ar[rrr, bend left = 30, ""{name = U}]{}{F} \ar[rrr, bend right = 30, swap, ""{name = B}]{}{F} & & &\fCat \ar[from = U, to = B, Rightarrow, shorten <= 4pt, shorten >= 4pt, bend right = 30, swap, ""{name = L}]{}{T} \ar[from = U, to = B, Rightarrow, shorten <= 4pt, shorten >= 4pt, bend left = 30, ""{name = R}]{}{T^2} \ar[from = L, to = R, symbol = \underset{\ell}{\Rrightarrow}, swap]
\end{tikzcd}
\]
whose object-local natural transformations are given by the vertical lifts 
\[
\!\quot{\ell}{X}\colon \!\quot{T}{X} \to \!\quot{T^2}{X}
\]
of the tangent structure $(F(X),\!\quot{\Tbb}{X})$.
\end{lem}
\begin{proof}
To show that $\ell$ is a modification it suffices to check that the diagram of functors and natural transformations
\[
\xymatrix{
F(f) \circ \quot{T}{Y} \ar[rr]^-{F(f) \ast \quot{\ell}{Y}} \ar[d]_{\quot{T}{f}} & & F(f) \circ \quot{T^2}{Y} \ar[d]^{\quot{T^2}{f}} \\
\quot{T}{X} \circ F(f) \ar[rr]_-{\quot{\ell}{X} \ast F(f)} & & \quot{T^{2}}{X} \circ F(f)
}
\]
commutes for any $f\colon X \to Y$ in $\Cscr_1$. However, this is equivalent to asking that the diagram
\[
\xymatrix{
F(f) \circ \quot{T}{Y} \ar[d]_{\quot{T}{f}} \ar[rr]^-{F(f) \ast \quot{\ell}{Y}} & & F(f) \circ \quot{T^2}{Y} \ar[d]^{(\quot{T}{X} \ast \quot{T}{f}) \circ (\quot{T}{f} \ast \quot{T}{Y})} \\
\quot{T}{X} \circ F(f) \ar[rr]_-{\quot{\ell}{X} \ast F(f)} & & \quot{T^2}{X} \circ F(f)
}
\]
commutes, which holds because $(F(f),\quot{T}{f})$ is a (strong) tangent morphism and
\[
\quot{T^2}{f} = (\quot{T}{X} \ast \quot{T}{f}) \circ (\quot{T}{f} \ast \quot{T}{Y}).
\]
\end{proof}
\begin{cor}\label{Cor: Existence of Equivariant Vertical Lift}
Let $F$ be a tangent indexing functor. Then there is a vertical lift natural transformation $\underline{\ell}\colon \underline{T} \Rightarrow \underline{T^2}$.
\end{cor}
\begin{proof}
Apply Theorem \ref{Theorem: Defining functors and nat transforms between equivariant categories and functors} to the vertical lift modification $\ell\colon T \Rrightarrow T^2$.
\end{proof}
\begin{lem}\label{Lemma: Equivariant vertical lift is bundle morphism}
	Let $F$ be a tangent indexing functor on $X$ and let $A$ be an object of $\PC(F)$. Then $(\underline{\ell}_A,0_{A})$ is a bundle morphism in $\Cscr$.
\end{lem}
\begin{proof}
	It suffices to show that the squares
	\[
	\begin{tikzcd}
	\underline{T}A \ar[r]{}{\underline{\ell}_A} \ar[d, swap]{}{\underline{p}_A} & \underline{T}^2A \ar[d]{}{(\underline{T} \ast \underline{p})_{A}} & \underline{T}A \times_{A} \underline{T}A \ar[d, swap]{}{\underline{+}_{A}} \ar[rr]{}{\langle\underline{\ell}_A \circ \pi_1, \underline{\ell}_A \circ \pi_2 \rangle} & & \underline{T}^2A \times_{\underline{T}A}\underline{T}^2A \ar[d]{}{(\underline{T} \ast \underline{+})_A} & A \ar[d, swap]{}{ \underline{0}_A} \ar[r]{}{\underline{0}_A} & \underline{T}A \ar[d]{}{(\underline{T} \ast \underline{0})_A} \\
	A \ar[r, swap]{}{\underline{0}_A} & \underline{T}A & \underline{T}A \ar[rr, swap]{}{\underline{\ell}_A} & & \underline{T}^2A & \underline{T}A \ar[r, swap]{}{\underline{\ell}_A} & \underline{T}^2A
	\end{tikzcd}
	\]
	all commute in $\PC(F)$. However, because these amount to checking the commutativity of diagrams in $\PC(F)$, it suffices to check that the diagrams all commute object-locally. That is, it suffices to prove that for all $X \in \Cscr_0$, the squares
	\[
	\begin{tikzcd}
	\quot{T}{X}\left(\quot{A}{X}\right) \ar[r]{}{\quot{\ell}{X}_{\quot{A}{X}}} \ar[d, swap]{}{\quot{p}{X}_{\quot{A}{X}}} & \quot{T^2}{X}\left(\quot{A}{X}\right) \ar[d]{}{(\quot{T}{X} \ast \quot{p}{X})_{\quot{A}{X}}}  \\
	\quot{A}{X} \ar[r, swap]{}{\quot{0}{X}_{\quot{A}{X}}} & \quot{T}{X}\left(\quot{A}{X}\right) 
	\end{tikzcd}
	\]
	\[
	\begin{tikzcd}
	\quot{T}{X}\left(\quot{A}{X}\right) \times_{\quot{A}{X}} \quot{T}{X}\left(\quot{A}{X}\right) \ar[d, swap]{}{\quot{+}{X}_{\quot{A}{X}}} \ar[rrrr]{}{\langle\quot{\ell}{X}_{\quot{A}{X}} \circ \quot{\pi_1}{X}, \quot{\ell}{X}_{\quot{A}{X}} \circ \quot{\pi_2}{X} \rangle} & & & & \quot{T^2}{X}\left(\quot{A}{X}\right) \times_{\quot{T}{X}\left(\quot{A}{X}\right)}\quot{T^2}{X}\left(\quot{A}{X}\right) \ar[d]{}{(\quot{T}{X} \ast \quot{+}{X})_{\quot{A}{X}}} \\
	\quot{T}{X}\left(\quot{A}{X}\right) \ar[rrrr, swap]{}{\quot{\ell}{X}_{\quot{A}{X}}} & & & & \quot{T^2}{X}\left(\quot{A}{X}\right)
	\end{tikzcd}
	\]
	\[
	\begin{tikzcd}
	\quot{A}{X} \ar[d, swap]{}{\quot{0}{X}_{\quot{A}{X}}} \ar[r]{}{\quot{0}{X}_{\quot{A}{X}}} & \quot{T}{X}\left(\quot{A}{X}\right) \ar[d]{}{(\quot{T}{X} \ast \quot{0}{X})_{\quot{A}{X}}} \\
	\quot{T}{X}\left(\quot{A}{X}\right) \ar[r, swap]{}{\quot{\ell}{X}_{\quot{A}{X}}} & \quot{T^2}{X}\left(\quot{A}{X}\right)
	\end{tikzcd}
	\]
	all commute. However, this follows by Axiom 3 of Definition \ref{Defn: Tangent Category} applied to $(F(X),\quot{\Tbb}{X})$.
\end{proof}

\begin{lem}\label{Lemma: Canonical flip modification}
Let $F$ be a tangent indexing functor on $X$. There is then a canonical flip modification $c\colon T^2 \Rrightarrow T^2$ as in the diagram
\[
\begin{tikzcd}
\Cscr^{\op} \ar[rrr, bend left = 30, ""{name = U}]{}{F} \ar[rrr, bend right = 30, swap, ""{name = B}]{}{F} & & &\fCat \ar[from = U, to = B, Rightarrow, shorten <= 4pt, shorten >= 4pt, bend right = 30, swap, ""{name = L}]{}{T^2} \ar[from = U, to = B, Rightarrow, shorten <= 4pt, shorten >= 4pt, bend left = 30, ""{name = R}]{}{T^2} \ar[from = L, to = R, symbol = \underset{c}{\Rrightarrow}, swap]
\end{tikzcd}
\]
whose object-local natural transformations are given by the canonical flips 
\[
\quot{c}{X}\colon \quot{T^2}{X} \to \quot{T^2}{X}
\]
of the tangent structure $(F(X),\quot{\Tbb}{X})$.
\end{lem}
\begin{proof}
To show that $c$ determines a modification it suffices to prove that for any $f\colon X \to Y$ in $\Cscr_1$ the diagram of functors and natural transformations
\[
\xymatrix{
F(f) \circ \quot{T^2}{Y} \ar[d]_{F(f) \ast \quot{c}{Y}} \ar[r]^-{\quot{T^2}{f}} & \quot{T^2}{X} \circ F(f) \ar[d]^{\quot{c}{X} \ast F(f)} \\
F(f) \circ \quot{T^2}{Y} \ar[r]_-{\quot{T^2}{f}} & \quot{T^2}{X} \circ F(f)
}
\]
commutes. However, this is equivalent to asking that the diagram
\[
\xymatrix{
F(f) \circ \quot{T^2}{Y} \ar[d]_{F(f) \ast \quot{c}{Y}} \ar[rrr]^-{(\quot{T}{X} \ast \quot{T}{f}) \circ (\quot{T}{f} \ast \quot{T}{Y})} & & & \quot{T^2}{X} \circ F(f) \ar[d]^{\quot{c}{X} \ast F(f)} \\
F(f) \circ \quot{T^2}{Y} \ar[rrr]_-{(\quot{T}{X} \ast \quot{T}{f}) \circ (\quot{T}{f} \ast \quot{T}{Y})} & & & \quot{T^2}{X} \circ F(f)
}
\]
commutes, which holds because $(F(f), \quot{T}{f})$ is a tangent morphism and because
\[
\!\quot{T^2}{f} = (\!\quot{T}{X} \ast \!\quot{T}{f}) \circ (\!\quot{T}{f} \ast \!\quot{T}{Y}).
\]
\end{proof}
\begin{cor}\label{Cor: Equivariant canonical flip}
For any tangent indexing functor $F$ on $X$ there is a canonical flip natural transformation $\underline{c}\colon \underline{T}^2 \Rightarrow \underline{T}^2$.
\end{cor}
\begin{lem}\label{Lemma: Equivariant canonical flip is bundle morphism}
Let $F$ be a tangent indexing functor on $X$ and let $A$ be an object of $\PC(F)$. Then $(\underline{c}_A,\id_{\underline{T}A})$ is a bundle morphism in $\Cscr$.
\end{lem}
\begin{proof}
It suffices to show that the squares
\[
	\begin{tikzcd}
\underline{T}^2A \ar[r]{}{\underline{c}_A} \ar[d, swap]{}{\underline{T}\underline{p}_A} & \underline{T}^2A \ar[d]{}{\underline{p}_{\underline{T}A}} & \underline{T}^2A \times_{\underline{T}A} \underline{T}^2A \ar[d, swap]{}{(\underline{T} \ast \underline{+})_{A}} \ar[rr]{}{\langle\underline{c}_A \circ \pi_1, \underline{c}_A\circ \pi_2\rangle} & & \underline{T}^2A \times_{\underline{T}A}\underline{T}^2A \ar[d]{}{(\underline{+} \ast \underline{T})_A} & \underline{T}A \ar[d, swap]{}{(\underline{T} \ast 0)_A} \ar[r]{}{\id_{\underline{T}A}} & \underline{T}A \ar[d]{}{(0 \ast \underline{T})_A} \\
\underline{T}A \ar[r, swap]{}{\id_{\underline{T}A}} & \underline{T}A & \underline{T}^2A \ar[rr, swap]{}{\underline{c}_A} & & \underline{T}^2A & \underline{T}^2A \ar[r, swap]{}{\underline{c}_A} & \underline{T}^2A
\end{tikzcd}
\]
all commute in $\PC(F)$. However, because these amount to checking the commutativity of diagrams in $\PC(F)$, it suffices to check that the diagrams all commute object-locally. That is, it suffices to prove that for all $X \in \Cscr_0$, the squares
\[
	\begin{tikzcd}
\quot{T^2}{X}\left(\quot{A}{X}\right) \ar[r]{}{\quot{c}{X}_{\quot{A}{X}}} \ar[d, swap]{}{\quot{T}{X}(\quot{p}{X}_{\quot{A}{X}})} & \quot{T^2}{X}\left(\quot{A}{X}\right) \ar[d]{}{(\quot{p}{X} \ast \quot{T}{X})_{\quot{A}{X}}} \\ 
\quot{T}{X}\left(\quot{A}{X}\right) \ar[r, swap]{}{\id_{\quot{T}{X}\quot{A}{X}}} & \quot{T}{X}\left(\quot{A}{X}\right) 
\end{tikzcd}
\]
\[
\begin{tikzcd}
\quot{T^2}{X}\left(\quot{A}{X}\right) \times_{\quot{T}{X}\left(\quot{A}{X}\right)} \quot{T^2}{X}\left(\quot{A}{X}\right) \ar[d, swap]{}{(\quot{T}{X} \ast \quot{+}{X})_{\quot{A}{X}}} \ar[rrrr]{}{\langle\quot{c}{X}_{\quot{A}{X}} \circ \quot{\pi_1}{X}, \quot{c}{X}_{\quot{A}{X}}\circ \quot{\pi_2}{X}\rangle}& & & & \quot{T^2}{X}\left(\quot{A}{X}\right) \times_{\quot{T}{X}\left(\quot{A}{X}\right)}\quot{T^2}{X}\left(\quot{A}{X}\right) \ar[d]{}{(\quot{+}{X} \ast \quot{T}{X})_{\quot{A}{X}}}  \\
\quot{T^2}{X}\left(\quot{A}{X}\right) \ar[rrrr, swap]{}{\quot{c}{X}_{\quot{A}{X}}} & & & & \quot{T^2}{X}\left(\quot{A}{X}\right)
\end{tikzcd}
\]
\[
\begin{tikzcd}
\quot{T}{X}\left(\quot{A}{X}\right) \ar[d, swap]{}{(\quot{T}{X} \ast 0)_{\quot{A}{X}}} \ar[r]{}{\id_{\quot{T}{X}\quot{A}{X}}} & \quot{T}{X}\left(\quot{A}{X}\right) \ar[d]{}{(0 \ast \quot{T}{X})_{\quot{A}{X}}} \\
\quot{T^2}{X}\left(\quot{A}{X}\right) \ar[r, swap]{}{\quot{c}{X}_{\quot{A}{X}}} & \quot{T^2}{X}\left(\quot{A}{X}\right)
\end{tikzcd}
\]
all commute. However, this follows by Axiom 4 of Definition \ref{Defn: Tangent Category} applied to $(F(X),\quot{\Tbb}{X})$.
\end{proof}

\begin{prop}\label{Prop: Coherences of equivairaint vertical lift and canonical flip}
Let $F$ be a tangent indexing functor. Then the canonical flip and vertical flip transformations satisfy the following coherences: $\underline{c}^2 = \iota_{\underline{F}}$, $\underline{c} \circ \underline{\ell} = \underline{\ell}$, and the diagrams of functors and natural transformations
\[
\begin{tikzcd}
\underline{T} \ar[r]{}{\underline{\ell}} \ar[d, swap]{}{\underline{\ell}} & \underline{T}^2 \ar[d]{}{\underline{T} \ast \underline{\ell}} & \underline{T}^3 \ar[d, swap]{}{\underline{c} \ast \underline{T}} \ar[r]{}{\underline{T} \ast \underline{c}} & T^3 \ar[r]{}{\underline{c} \ast \underline{T}} & \underline{T}^3\ar[d]{}{\underline{c} \ast \underline{T}} \\
T^2 \ar[r, swap]{}{\underline{\ell} \ast \underline{T}} & \underline{T}^3 & \underline{T}^3 \ar[r, swap]{}{\underline{T} \ast \underline{c}} & \underline{T}^3 \ar[r, swap]{}{\underline{c} \ast \underline{T}} & \underline{T}^3
\end{tikzcd}
\begin{tikzcd}
\underline{T}^2 \ar[d, swap]{}{\underline{c}} \ar[r]{}{\underline{\ell} \ast \underline{T}} & \underline{T}^3 \ar[r]{}{\underline{T} \ast \underline{c}} & \underline{T}^3 \ar[d]{}{\underline{c} \ast \underline{T}} \\
\underline{T}^2 \ar[rr, swap]{}{\underline{T} \ast \underline{\ell}} & & \underline{T}^3
\end{tikzcd}
\]
all commute.
\end{prop}
\begin{proof}
It suffices through Theorem \ref{Theorem: Defining functors and nat transforms between equivariant categories and functors} to show that the corresponding identities and commuting diagrams hold for the corresponding pseudonatural transformations and modifications. In particular, it suffices to verify each corresponding identity in $F(X)$ for all $X \in \Cscr_0$. However, each identity holds because each object is given in terms of the corresponding functor/natural transformation in the tangent structure $(F(X),\quot{\Tbb}{X})$.
\end{proof}

As a last lonely proposition we show the universality of the vertical lift in $\PC(F)$. This is our final ingredient in proving that $\PC(F)$ carries the structure of a tangent category.
\begin{prop}\label{Prop: Equivariant universality vertical lift}
Let $F$ be a tangent indexing functor on $X$. Then for any $A \in \PC(F)_0$ the diagram
\[
\begin{tikzcd}
\underline{T}_2A \ar[rrrr]{}{\left(\underline{T} \ast \underline{+}\right)_A \circ \left\langle \underline{\ell}_A \circ \pi_1, (\underline{0} \ast \underline{T})_A \circ \pi_2\right\rangle} & & & & \underline{T}^2A \ar[shift left = 0.5ex, rr]{}{(\underline{T} \ast \underline{p})_A} \ar[shift left = -0.5, rr, swap]{}{\underline{0}_{A} \circ \underline{p}_{A} \circ \underline{p}_{\underline{T}A}} & & \underline{T}A
\end{tikzcd}
\]
is an equalizer in $\PC(F)$.
\end{prop}
\begin{proof}
Recall that because each pair $(F(f),\quot{T}{f})$ for $f \in \Cscr_1$ is a strong tangent morphism, all the pullbacks expressing the universality of the vertical lift are preserved by $F(f)$. Thus, in order to prove the given diagram is an equalizer it suffices by Theorem \ref{Thm: Limits in pseudcone categories} to check that for every $X \in \Cscr_0$, the diagram
\[
\begin{tikzcd}
\quot{T_2}{X}\left(\quot{A}{X}\right) \ar[rrrrrr]{}{\left(\quot{T}{X} \ast \quot{+}{X}\right)_{\quot{A}{X}} \circ \left\langle \quot{\ell}{X}_{\quot{A}{X}} \circ \pi_1, (\quot{0}{X} \ast \quot{T}{X})_{\quot{A}{X}} \circ \pi_2\right\rangle} & & & & & & \quot{T}{X}^2\left(\quot{A}{X}\right) \ar[shift left = 0.5ex, rrrr]{}{(\quot{T}{X} \ast \underline{p})_{\quot{A}{X}}} \ar[shift left = -0.5, rrrr, swap]{}{\underline{0}_{\quot{A}{X}} \circ \underline{p}_{\quot{A}{X}} \circ \underline{p}_{\quot{T}{X}\left(\quot{A}{X}\right)}} & & & & \quot{T}{X}\left(\quot{A}{X}\right)
\end{tikzcd}
\]
is an equalizer. However, this holds because $(F(X),\quot{\Tbb}{X}) = (F(X),\quot{T}{X},\quot{p}{X}, \quot{+}{X}, \quot{0}{X}, \quot{\ell}{X}, \quot{c}{X})$ is a tangent category.
\end{proof}

With this we can finally prove that when $F$ is a tangent indexing functor the pseudolimit category, $\PC(F)$ is a tangent category.

\begin{Theorem}\label{Thm: Pre-Equivariant Tangent Category}
If $F\colon \Cscr^{\op} \to \fCat$ is a tangent indexing functor then $\PC(F)$ is a tangent category where the tangent structure $\Tbb$ is given by $\Tbb = (\underline{T}, \underline{p}, \underline{0}, \underline{+}, \underline{\ell}, \underline{c})$ for the functors and natural transformations determined in Lemma \ref{Lemma: Constructing the tangent functor} and in Corollaries \ref{Cor: Bundle transformation equivariant tan functor}, \ref{Cor: Construction of equivariant zero transformation}, \ref{Cor: Bundle addition transformation equivariant category}, \ref{Cor: Existence of Equivariant Vertical Lift}, \ref{Cor: Equivariant canonical flip}.
\end{Theorem}
\begin{proof}
The existence of each functor and natural transformation is given in the statement of the theorem, so we only need to prove the desired coherences and structural properties of Definition \ref{Defn: Tangent Category}. However, one may prove that the axioms in Definition \ref{Defn: Tangent Category} hold for $(\PC(F),\Tbb)$ as follows:
\begin{enumerate}
	\item Part $1$ is verified by the existence of $\underline{T}$ (cf. Lemma \ref{Lemma: Constructing the tangent functor}), the existence of $\underline{p}$ (cf. Corollary \ref{Cor: Bundle transformation equivariant tan functor}), the existence of pullback powers of $\underline{p}_A\colon \underline{T}A \to A$ for all $A \in \PC(F)_0$ (cf. Corollary \ref{Cor: Tangent Pullback functors as equivariant functors}), and the commutativity of composite powers of $\underline{T}$ with pullback powers of $\underline{p}$ (cf. Corollary \ref{Cor: Tangent powers preserve pullback functors}).
	\item Part $2$ holds by the existence of the bundle addition map $\underline{+}$ (cf. Corollary \ref{Cor: Bundle addition transformation equivariant category}), the existence of the bundle unit $\underline{0}$ (cf. Corollary \ref{Cor: Construction of equivariant zero transformation}), and Proposition \ref{Prop: Equivariant Tangent over A is Additive Bundle over A}.
	\item Part $3$ is verified in two steps. First, the existence of $\underline{\ell}$ follows from Corollary \ref{Cor: Existence of Equivariant Vertical Lift}. Second, each pair $(\underline{\ell}_A,\underline{0})_A$ is a bundle map for all $A \in \PC(F)_0$ by Lemma \ref{Lemma: Equivariant vertical lift is bundle morphism}.
	\item Part $4$ is verified in two steps. First, the canonical flip $\underline{c}$ exists by Corollary \ref{Cor: Equivariant canonical flip}. Second, the pair $(\underline{c}_A, \id_{\underline{T}A})$ is a bundle map for any $A \in \PC(F)_0$ by Lemma \ref{Lemma: Equivariant canonical flip is bundle morphism}.
	\item Part $5$ holds by Proposition \ref{Prop: Coherences of equivairaint vertical lift and canonical flip}. 
	\item Part $6$ holds by Proposition \ref{Prop: Equivariant universality vertical lift}.
\end{enumerate}
Thus $(\PC(F),\Tbb)$ is a tangent category.
\end{proof}

\begin{example}
Let $F\colon \Cscr^{\op} \to \fCat$ be any pseudofunctor and equip each fibre category $F(X)$ with the trivial tangent structure $\quot{I}{X}$. Then the pseudolimit tangent structure $(\PC(F),{\Ibb})$ asserted by Theorem \ref{Thm: Pre-Equivariant Tangent Category} coincides with the trivial tangent structure $(\PC(F), \Ibb_{\PC(F)})$ on $\PC(F)$.
\end{example}
For further (families of) examples of pseudolimit tangent structures, see Definitions \ref{Defn: Equivariant Zariski category} and \ref{Defn: Descent Equivariant Manfiolds} below.

\begin{rmk}\label{Remark: What we actually need}
The general theme so far in this paper has been to take a pseudofunctor $F\colon\Cscr^{\op} \to \fCat$, assume each category $F(X)$ for $X \in \Cscr_0$ has some structure $S$ (such as a limit of shape $D$, a tangent structure, etc.), assume that each functor $F(f)$ for $f\colon X \to Y$ preserves this structure, and determine if and when the pseudolimit $\PC(F)$ has the same structure (cf.\! Theorems \ref{Thm: Limits in pseudcone categories}, \ref{Thm: Pre-Equivariant Tangent Category}). This is an instance of a more general approach for doing descent-theoretic $2$-and-$bi$-category theory which can be outlined/summarized as follows:
\begin{enumerate}
	\item Begin by choosing a structure $S$ which can apply to some (subclass, if need be) of categories. That is, $S$ applies to a subclass $\Scal$ of $\mathbf{Cat}_0$.
	\item Assume that we have a category $\Cscr$ and a pseudofunctor $F\colon\Cscr^{\op} \to \fCat$ for which: each category $F(X)$ has structure $S$ for all $X \in \Cscr_0$ and each functor $F(f)$ preserves $S$ or ``commutes with'' $S$, up to a natural isomorphism (such as a distributive law), for all morphisms $f$ in $\Cscr$.
	\item Assume that the isomorphisms which witness that each $F(f)$ preserves/commutes with the structure $S$ vary coherently in $\Cscr^{\op}$.
	\item Then deduce that the pseudolimit $\PC(F)$ satisfies the structure $S$ as well.
\end{enumerate}
Effectively, this comes down to witnessing that these are precisely the pseudofunctors $F$ and properties $S$ for which one can do effective descent. This has been studied in depth in \cite{GeneralGeoffThesis}. For instance, there it is shown using $S$ as a category being a (braided) monoidal category with translation functors being strong (braided) monoidal functors; $S$ being a regular category with translation functors regular functors; $S$ being that a category has all (co)limits of a fixed shape and the translation functors preserve these (co)limits; $S$ being that a category is triangulated and having translation functors all be triangulated functors; and more. Theorem \ref{Thm: Pre-Equivariant Tangent Category} continues this trend, but an exact study of which structures $S$ one can examine, internal to which $2$-categories/bicategories $\Kcal$, and how they relate to both effective descent and the Grothendieck construction is the subject of future work and future examination.
\end{rmk}

\subsection{Tangent Objects in the Hom-2-Category}
We close this section with a discussion of the theory of tangent objects and the relation of tangent indexing functors and tangent objects in the hom $2$-category $\Bicat(\Cscr^{\op},\fCat)$. Recently, in \cite{Marcello} Lanfranchi has defined the notion of a tangent object within a $2$-category as a natural generalization of what it means to be a tangent category; in fact, Lanfranchi's definition is motivated by characterizing the tangent objects of $\fCat$ as precisely tangent categories. 
We begin by showing that the tangent objects of $\Bicat(\Cscr^{\op},\fCat)$ can be described in terms of pseudofunctors $F:\Cscr^{\op} \to \fCat$ which factor through the forgetful $2$-functor
\[
\begin{tikzcd}
\Cscr^{\op}  \ar[rr]{}{F} \ar[dr, swap]{}{F} & & \fCat \\
    & \fTan_{\operatorname{strong}} \ar[ur, swap]{}{\Forget}
\end{tikzcd}
\]
and so we characterize the indexed tangent categories of \cite[Definition 3.8]{Marcello} as exactly the tangent objects in $\Bicat(\Cscr^{\op},\fCat)$\footnote{In \cite{Marcello}, the definition of an indexed tangent category involves a pseudofunctor $F\colon\Cscr^{\op} \to \fTan_{\operatorname{strong}}$ where $\Cscr = (\Cscr,\Tbb)$ is some given tangent category. However, the definition involves no other use of the tangent structure on $\Cscr$ and so may be omitted.}. To prove this, we begin with a convenient and important technical lemma.
\begin{lem}\label{Lemma: strong lax mor iff strong colax}
Let $(\Cscr,\Tbb)$ and $(\Dscr,\Sbb)$ be tangent categories. Then a morphism $(F,\alpha)\colon\Cscr \to \Dscr$ is a strong lax morphism of tangent categories if and only if $(F,\alpha^{-1})\colon\Cscr \to \Dscr$ is a strong colax morphism of tangent categories.
\end{lem}
\begin{proof}
$\implies$: If $(F,\alpha)$ is a strong lax morphism of tangent categories then
\[
\begin{tikzcd}
\Cscr \ar[r, ""{name = U}]{}{F} \ar[d, swap]{}{T} & \Dscr \ar[d]{}{S} \\
\Cscr \ar[r, swap, ""{name = D}]{}{F} & \Dscr
\ar[from = U, to = D, Rightarrow, shorten <= 4pt, shorten >= 4pt]{}{\alpha^{-1}}
\end{tikzcd}
\]
is a natural isomorphism. That this forms the natural transformation of a colax morphism of tangent categories follows by virtue of inverting the direction of all the maps $\alpha:F \circ T \Rightarrow S \circ F$ together with the corresponding coherences. 

$\impliedby$: If we instead know that $(F,\alpha)$ is a strong colax morphism of tangent categories, then
\[
\begin{tikzcd}
\Cscr \ar[r, ""{name = U}]{}{T} \ar[d, swap]{}{F} & \Cscr \ar[d]{}{F} \\
\Dscr \ar[r, swap, ""{name = D}]{}{S} & \Dscr
\ar[from = U, to = D, Rightarrow, shorten <= 4pt, shorten >= 4pt]{}{\alpha^{-1}}
\end{tikzcd}
\]
is a natural isomorphism. That this forms the natrual transformation of a lax morphism of tangent categories follows by virtue of inverting the direction of all the maps $\alpha:S \circ F \Rightarrow F \circ T$ together with the corresponding coherences.
\end{proof}

Because we require it in the remainder of this section and in Section \ref{Section: Eq Tan Mor}, we now recall the definition of what it means to be a tangent object in a $2$-category $\Cfrak$. This definition requires the use of \cite{LeungWeil} and the approach to tangent categories as categories $\Cscr$ together with a $\mathsf{Weil}_1$-actegory structure, i.e., a strong monoidal functor $\mathsf{Weil}_1 \to [\Cscr,\Cscr]$. We consequently also recall the definition of the category $\mathsf{Weil}_1$ as well.
\begin{dfn}[{\cite{LeungWeil}}]\label{Defn: Weil algebras}
The category $\mathsf{Weil}_1$ is the monoidal category where:
\begin{itemize}
    \item Objects: Finite tensor products of the $\N$-algebras
    \[
    W^n := \frac{\N[x_1, \cdots, x_n]}{(x_ix_j: 1 \leq i, j \leq n)}.
    \]
    \item Morphisms: As in the category of augmented $\N$-algebras.
    \item Composition: As in the category of augmented $\N$-algebras.
    \item Identities: The identity of a Weil algebra $A$ is $\id_A$.
    \item Monoidal Product: The monoidal product is the $\N$-tensor product $\otimes_{\N}$.
    \item Monoidal unit: The monoidal unit is $\N$.
    \item Unitors: The tensor units $A \otimes_{\N} \N \cong A$ and $A \cong \N \otimes_{\N} A$.
    \item Associators: The tensor associator: $(A \otimes_{\N} B) \otimes_{\N} C \cong A \otimes_{\N} (B \otimes_{\N} C)$.
\end{itemize}
\end{dfn}
The category $\mathsf{Weil}_1$ is equipped with a special algebra $W := \N[x]/(x^2)$ which largely generates the category. It is used as a placeholder for determining the tangent structure on $\Cscr$ by determining the functor that $W$ is mapped to by the strong monoidal functor $\alpha:\mathsf{Weil}_1 \to [\Cscr,\Cscr]$. The corresponding ways to translate the remainder of the tangent structure (namely the bundle projection, the sum, the zero, the lift, and the flip) are given by the following morphisms:
\begin{itemize}
    \item $p:W \to \N$ induced by sending $x \mapsto 0$; this is the map which realizes $W$ as an augmented $\N$-algebra.
    \item The zero morphism $z:\N \to W$ is the rig map which witnesses $W$ as an $\N$-algebra.
    \item The sum map $\operatorname{add}:W^2 \to W$ which sends the two generators $x_1$ and $x_2$ of $W^2$ to the generator $x$ of $W$.
    \item The vertical lift $\ell:W \to W \otimes_{\N} W$ is the map which sends the generator $x$ of $W$ to the element $x \otimes y$.
    \item The canonical flip $c:W \otimes_{\N} W \to W \otimes_{\N} W$ is defined to be the map
    \[
    \frac{\N[x_1,x_2]}{(x_1^2, x_2^2)} \to \frac{\N[x_1, x_2]}{(x_1^2, x_2^2)}
    \]
    given by $x_1 \mapsto x_2$ and $x_2 \mapsto x_1$.
\end{itemize}
Generally we will denote by $p_A:A \to \N$ the augmentation of the algebra $A$ for all Weil algebras $A$.
\begin{dfn}[{\cite{Marcello}, Definition 4.3}]
Let $\Cfrak$ be a $2$-category. We say that an object $X$ of $\Cfrak$ is a \emph{tangent object} with a tangent structure $\Tbb$ if there is a strong monoidal functor $F_{\Tbb}:\mathsf{Weil}_1 \to \Cfrak(X,X)$ for which:
\begin{enumerate}
    \item $F_{\Tbb}$ preserves the pullbacks of the form
    \[
    \begin{tikzcd}
    A \otimes (B \times C) \ar[r]{}{\id_A \otimes \operatorname{pr}_2} \ar[d, swap]{}{\id_A \otimes \operatorname{pr}_1} & A \otimes C \ar[d]{}{\id_A \otimes p_C} \\
    A \otimes B \ar[r, swap]{}{\id_A \otimes p_B} & A
    \end{tikzcd}
    \]
    for all Weil-algebras $A, B, C$ and that these pullbacks be pointwise, i.e., they are preserved by all functors $\Cfrak(f,X)$. Note that we also have omitted the unitor $A \otimes_{\N} \N \cong A$ in the description of the pullback above.
    \item The functor $F_{\Tbb}$ perserves the universality of the vertical lift, i.e., if $\xi$ is the pairing map of
    \[
    \begin{tikzcd}
    W^2 \ar[r]{}{\cong} \ar[d, swap]{}{\ell_{W^2}} & \N \otimes_{\N} W^2 \ar[r]{}{z \otimes \id_W} & W \otimes_{\N} W^2 \ar[d]{}{\id_W \otimes \operatorname{add}} \\
    W \otimes_{\N} W^2 \ar[rr, swap]{}{\id_W \otimes \operatorname{add}} & & W \otimes_{\N} W
    \end{tikzcd}
    \]
    then $F_{\Tbb}$ preserves the pullback square
    \[
    \begin{tikzcd}
    W^2 \ar[r]{}{\xi} \ar[d, swap]{}{x_1, x_2 \mapsto 0} & W \otimes_{\N} W \ar[d]{}{\id_W \otimes p_{W}} \\
    \N \ar[r, swap]{}{z} & W
    \end{tikzcd}   
    \]
    and sends it to a pointwise limit in $\Cfrak(X,X)$.
\end{enumerate}
\end{dfn}
\begin{rmk}
The definition of a tangent object is more general than we will require in this paper. What it means to give a tangent object in the $2$-category
$\Bicat(\Cscr^{\op},\fCat)$ is to require:
\begin{enumerate}
    \item The existence of a pseudofunctor $F\colon\Cscr^{\op} \to \fCat$,
    a pseudonatural transformation $T$ and a modification $p$,
\[
\begin{tikzcd}
\Cscr^{\op} \ar[rr, bend left = 30, ""{name = U}]{}{F} \ar[rr, bend right = 30, swap, ""{name = D}]{}{F}  & & \fCat
\ar[from = U, to = D, Rightarrow, shorten <= 4pt, shorten >= 4pt]{}{T}
\end{tikzcd}\quad \begin{tikzcd}
\Cscr^{\op} \ar[rrr, bend left = 30, ""{name = U}]{}{F} \ar[rrr, bend right = 30, swap, ""{name = B}]{}{F} & & &\fCat \ar[from = U, to = B, Rightarrow, shorten <= 4pt, shorten >= 4pt, bend right = 30, swap, ""{name = L}]{}{T} \ar[from = U, to = B, Rightarrow, shorten <= 4pt, shorten >= 4pt, bend left = 30, ""{name = R}]{}{\id_{F}} \ar[from = L, to = R, symbol = \underset{p}{\Rrightarrow}, swap]
\end{tikzcd}
\]
such that all iterated pullbacks against the cospan $T \xrightarrow{p} \id_{F} \xleftarrow{p} T$ exist in $\Bicat(\Cscr^{\op},\fCat)(F,F)$ and are preserved by all powers $T^m$. We define the pullback of $T \xrightarrow{p} \id_{F} \xleftarrow{p} T$ to be the pseudonatural transformation $T_2$.
\item The existence of a pair of modifications
\[
\begin{tikzcd}
\Cscr^{\op} \ar[rrr, bend left = 30, ""{name = U}]{}{F} \ar[rrr, bend right = 30, swap, ""{name = B}]{}{F} & & &\fCat \ar[from = U, to = B, Rightarrow, shorten <= 4pt, shorten >= 4pt, bend right = 30, swap, ""{name = L}]{}{T_2} \ar[from = U, to = B, Rightarrow, shorten <= 4pt, shorten >= 4pt, bend left = 30, ""{name = R}]{}{T} \ar[from = L, to = R, symbol = \underset{\operatorname{add}}{\Rrightarrow}, swap]
\end{tikzcd}\quad
\begin{tikzcd}
\Cscr^{\op} \ar[rrr, bend left = 30, ""{name = U}]{}{F} \ar[rrr, bend right = 30, swap, ""{name = B}]{}{F} & & &\fCat \ar[from = U, to = B, Rightarrow, shorten <= 4pt, shorten >= 4pt, bend right = 30, swap, ""{name = L}]{}{\id_{F}} \ar[from = U, to = B, Rightarrow, shorten <= 4pt, shorten >= 4pt, bend left = 30, ""{name = R}]{}{T} \ar[from = L, to = R, symbol = \underset{0}{\Rrightarrow}, swap]
\end{tikzcd}
\]
which make $(T, \operatorname{add}:T_2 \to T, 0:\id_F \to T)$ into a commutative monoid in $\Bicat(\Cscr^{\op},\fCat)(T,T)$.
\item The existence of modifications
\[
\begin{tikzcd}
\Cscr^{\op} \ar[rrr, bend left = 30, ""{name = U}]{}{F} \ar[rrr, bend right = 30, swap, ""{name = B}]{}{F} & & &\fCat \ar[from = U, to = B, Rightarrow, shorten <= 4pt, shorten >= 4pt, bend right = 30, swap, ""{name = L}]{}{T} \ar[from = U, to = B, Rightarrow, shorten <= 4pt, shorten >= 4pt, bend left = 30, ""{name = R}]{}{T^2} \ar[from = L, to = R, symbol = \underset{\ell}{\Rrightarrow}, swap]
\end{tikzcd}\quad
\begin{tikzcd}
\Cscr^{\op} \ar[rrr, bend left = 30, ""{name = U}]{}{F} \ar[rrr, bend right = 30, swap, ""{name = B}]{}{F} & & &\fCat \ar[from = U, to = B, Rightarrow, shorten <= 4pt, shorten >= 4pt, bend right = 30, swap, ""{name = L}]{}{T^2} \ar[from = U, to = B, Rightarrow, shorten <= 4pt, shorten >= 4pt, bend left = 30, ""{name = R}]{}{T^2} \ar[from = L, to = R, symbol = \underset{c}{\Rrightarrow}, swap]
\end{tikzcd}
\]
which satisfy the coherences in Lines (\ref{Eqn: First line of the pasting equations}) -- (\ref{Eqn: Fifth line fo pasting equations}) below.
\item The diagram
\[
\begin{tikzcd}
T_2 \ar[r]{}{\operatorname{add} \circ (\ell \times 0)} \ar[d, swap]{}{p \circ \pi_1} & T^2 \ar[d]{}{T \ast p} \\
\id_{F} \ar[r, swap]{}{0} & T
\end{tikzcd}
\]
is a (pointwise) pullback in $\Bicat(\Cscr^{\op},\fCat)(F,F)$.
\end{enumerate} 
Below are the coherences we require the tangent objects to satisfy:
\begin{align}
\ell \circ \operatorname{add} &= \left(\operatorname{add} \ast T\right) \circ \left(\ell \times \ell\right) & c \circ (T \ast \operatorname{add})&= (\operatorname{add} \ast T) \circ (c \times c) \label{Eqn: First line of the pasting equations} \\
(p \ast T) &= T \ast p & 0 \ast T&= c \circ (T \ast 0) \label{Eqn: Second line of pasting equations} \\
c \circ \ell &= \ell & c \circ c &= \id_{T^2} \label{Eqn: Third line of pasting equations} \\
(T \ast \ell) \circ \ell &= (\ell \ast T) \circ \ell & (\ell \ast T) \circ c &= (T \ast c) \circ (c \ast T) \circ \ell \label{Eqn: Fourth line of pasting equations} \\
(T \ast c) \circ (c \ast T) \circ (T \ast c) &= (c \ast T) \circ (T \ast c) \circ (c \ast T) \label{Eqn: Fifth line fo pasting equations}
\end{align}
\end{rmk}
\begin{rmk}
In the proof of the theorem below, we indicate where it is that we need to use pseudofunctors which land in $\fTan_{\operatorname{strong}}$ and not merely lax functors which land in $\fTan$. The issue arises in each case where we invoke Lemma \ref{Lemma: strong lax mor iff strong colax}, as this indicates that we require the naturality morphisms to be actual \emph{isomorphisms}.
\end{rmk}
\begin{Theorem}\label{Thm: Indexed tangent category}
Tangent objects in $\Bicat(\Cscr^{\op},\fCat)$ are determined by pseudofunctors $F\colon\Cscr^{\op} \to \fCat$ which factor through the forgetful functor:
\[
\begin{tikzcd}
\Cscr^{\op} \ar[rr]{}{F} \ar[dr, swap]{}{F} & & \fCat \\
 & \fTan_{\operatorname{strong}} \ar[ur, swap]{}{\Forget}
\end{tikzcd}
\]
\end{Theorem}
\begin{proof}
$\implies:$ Assume that $(F,T,p,0,\operatorname{add},\ell,c)$ is a tangent object in $\Bicat(\Cscr^{\op},\fCat)$. The object-local versions of the pasting diagrams and coherences required of a tangent object (cf. Lines (\ref{Eqn: First line of the pasting equations}) -- (\ref{Eqn: Fifth line fo pasting equations})) imply that for every object $X \in \Cscr^{\op}$, the category $F(X)$ is a tangent category with tangent functor $\quot{T}{X}$, bundle projection $\quot{p}{X}$, zero section $\quot{0}{X}$, addition $\quot{\operatorname{add}}{X}$, lift $\quot{\ell}{X}$, and canonical flip $\quot{c}{X}$. This shows us that the object-assignment $F_0:\Cscr^{\op}_0 \to \fCat_0$ factors through $\Forget_0:\fTan_0 \to \fCat_0$. To see that the morphisms $Ff:FY \to FX$ are the functor components of (strong) lax tangent morphisms, recall that because $T:F \Rightarrow F$ is a pseudonatural transformation for any $f:X \to Y$ in $\Cscr$, there is a corresponding invertible $2$-cell
\[
\begin{tikzcd}
FY \ar[r, ""{name = U}]{}{F(f)} \ar[d, swap]{}{\quot{T}{Y}} & FX \ar[d]{}{\quot{T}{X}} \\
FY \ar[r, swap, ""{name = D}]{}{F(f)} & FX
\ar[from = U, to = D, Rightarrow, shorten <= 4pt, shorten >= 4pt]{}{\quot{T}{f}}
\end{tikzcd}
\]
which, by virtue of the coherences implied in Lines (\ref{Eqn: First line of the pasting equations}) -- (\ref{Eqn: Fifth line fo pasting equations}), gives rise to a colax morphism of tangent categories. Moreover, because $T$ is a pseudonatural transformation, $\quot{T}{f}$ is invertible and so $(F(f),\quot{T}{f})$ is a strong colax morphism of tangent categories. However, applying Lemma \ref{Lemma: strong lax mor iff strong colax} shows that $(F(f),\quot{T}{f}^{-1})$ is a strong lax morphism of tangent categories and hence shows that $F_1$ factors through $\fTan$ as well. Thus $F:\Cscr^{\op} \to \fCat$ factors through $\Forget:\fTan_{\operatorname{strong}} \to \fCat$.

$\impliedby$: Assume that $F:\Cscr^{\op} \to \fCat$ factors as:
\[
\begin{tikzcd}
\Cscr^{\op} \ar[rr]{}{F} \ar[dr, swap]{}{F} & & \fCat \\
 & \fTan_{\operatorname{strong}} \ar[ur, swap]{}{\Forget}
\end{tikzcd}
\]
Because $F$ lands first in the $2$-category $\fTan_{\operatorname{strong}}$, for any morphism $f:X \to Y$ in $\Cscr$ there is a strong tangent morphism $(F(f),\quot{\alpha}{f})$. But then by Lemma \ref{Lemma: strong lax mor iff strong colax} the map $\quot{\alpha}{f}^{-1}$ fits into a $2$-cell
\[
\begin{tikzcd}
FY \ar[d, swap]{}{\quot{T}{Y}} \ar[r, ""{name = U}]{}{F(f)} & FX \ar[d]{}{\quot{T}{X}} \\
FY \ar[r, swap, ""{name = D}]{}{F(f)} & FX
\ar[from = U, to = D, Rightarrow, shorten <= 4pt, shorten >= 4pt]{}{\quot{\alpha}{f}^{-1}}
\end{tikzcd}
\]
which varies pseudofunctorially in $\Cscr^{\op}$. Arguing as in Lemmas \ref{Lemma: Constructing the tangent functor} shows that the data described by $T = (\quot{T}{X},\quot{\alpha}{f}^{-1})_{X \in \Cscr_0, f \in \Cscr_1}$ is then a pseudonatural transformation $T:F \Rightarrow F$. Similarly, arguing as in Lemmas \ref{Lemma: Construction of Bundle Map}, \ref{Lemma: Zero modification}, \ref{Lemma: Addition modification}, Proposition \ref{Prop: Equivariant Tangent over A is Additive Bundle over A}, Lemma \ref{Lemma: Vertical lift modification}, and Lemma \ref{Lemma: Canonical flip modification} shows that the object-local assignments for the bundle projections $\quot{p}{X}$, the zero maps $\quot{0}{X}$, the additions $\quot{\operatorname{add}}{X}$, the lifts $\quot{\ell}{X}$, and the canonical flips $\quot{c}{X}$ determine modifications $p$, $0$, $\operatorname{add}$, $\ell$, and $c$ which determine a commutative monoid in $\Bicat(\Cscr^{\op},\fCat)(F,F)$, satisfy the coherences in Lines (\ref{Eqn: First line of the pasting equations}) --  (\ref{Eqn: Fifth line fo pasting equations}), and have universal vertical lifts. As such, by extracting this information we find that $F$ determines a tangent object in $\Bicat(\Cscr^{\op},\fCat)$.
\end{proof}

We now show how our notion of a tangent indexing functor fits into this formulation and language via the theorem below.
\begin{Theorem}\label{Thm: Tangent objects in Bicat Hom category are tangent indexing functors}
Let $\Cscr$ be a $1$-category and consider the $2$-category $\Bicat(\Cscr^{\op},\fCat)$. Then a tuple\\ $(F,T,p,\operatorname{add},0,\ell,c)$ in $\Bicat(\Cscr^{\op},\fCat)$ is a tangent object in the sense of \cite[Definition 3.2]{Marcello} if and only if $F\colon \Cscr^{\op} \to \fCat$ is a tangent indexing functor.
\end{Theorem}
\begin{rmk}
Note that in Theorem \ref{Thm: Tangent objects in Bicat Hom category are tangent indexing functors},  the roles of the witness isomorphism of the pseudonatural transformation $\quot{T}{f}$ and its inverse $\quot{T}{f}^{-1}$ are reversed when passing between tangent indexing functors and tangent objects. This illustrates another reason why we work with pseudonatural transformations as opposed to lax transformations.
\end{rmk}
\begin{proof}
The theorem follows because to give a tangent indexing functor $F\colon\Cscr^{\op} \to \fCat$ is precisely to give a pseudofunctor $\underline{F}\colon\Cscr^{\op} \to \fTan_{\operatorname{strong}}$ by Proposition \ref{Prop: Tangent indexing functor is a pseudofunctor into Tanstrong} and because tangent objects in $\Bicat(\Cscr^{\op},\fCat)$ are exactly pseudofunctors $\underline{F}\colon\Cscr^{\op} \to \fTan_{\operatorname{strong}} \xrightarrow{\Forget} \fCat$ by Theorem \ref{Thm: Indexed tangent category}.
\end{proof}

Since tangent objects in $\fCat$ are tangent categories, we obtain now the following corollary to Theorem \ref{Thm: Pre-Equivariant Tangent Category}.

\begin{cor}
    The pseudolimit 2-functor $\PC\colon \Bicat(\Cscr^{\op}, \fCat)\to\fCat$ sends tangent objects to tangent objects.
\end{cor}

The fact that $\operatorname{pseudolim}$ lifts to a 2-functor on 2-categories of tangent objects is a direct consequence of Theorem \ref{Thm: PC of tangent objects in hom two category takes values in tangent categories}.

\section{Pseudolimit Tangent Morphisms}\label{Section: Eq Tan Mor}
So far we have seen in Theorem \ref{Thm: Pre-Equivariant Tangent Category} that tangent indexing functors $F\colon \Cscr^{\op} \to \fCat$ give rise to tangent structures on their pseudolimit $\PC(F)$. In this section we will examine the functoriality of this construction based on the strict $2$-functor $\PC(-)\colon \Bicat(\Cscr^{\op},\fCat) \to \fCat$ constructed in \cite[Lemma 4.1.11, Page 80]{GeneralGeoffThesis}; note that we have already seen the object assignment of $\PC(-)$ in the construction of the pseudocone categories and $1$-and-$2$-cell assignments of $\PC(-)$ in Theorem \ref{Theorem: Defining functors and nat transforms between equivariant categories and functors}. The content of the citation given here is that this construction is strictly $2$-functorial, i.e., $\PC(-)$ is a pseudofunctor with identity structure cells.

Our next main goal is to establish Theorem \ref{Thm: PC of tangent objects in hom two category takes values in tangent categories}, which states that the diagram
\[
\begin{tikzcd}
\Bicat(\Cscr^{\op},\fCat) \ar[rr]{}{\PC(-)} & & \fCat \\
\mathfrak{Tan}\left(\Bicat(\Cscr^{\op},\fCat)\right)  \ar[u]{}{\Forget} \ar[rr, swap]{}{\PC(-)} & & \fTan \ar[u, swap]{}{\Forget}
\end{tikzcd}
\]
exists, commutes strictly, and creates pseudolimits indexed by $1$-categories in the $2$-category $\fTan$. To do this, however, we will need to understand the $2$-category $\mathfrak{Tan}\left(\Bicat(\Cscr^{\op},\fCat)\right)$ of tangent objects, tangent morphisms, and tangent transformations in the $2$-category $\Bicat(\Cscr^{\op},\fCat)$ introduced in \cite{Marcello} as well as what they correspond to in terms of tangent indexing functors and the data $\Bicat(\Cscr^{\op}, \fCat)$ encodes.

We begin the task above by recalling what it means to be a morphism of tangent objects in the sense of \cite{Marcello}, as this will form the crux of our discussion of functoriality on $1$-cells and $2$-cells.
\begin{dfn}[{\cite{Marcello}, Definition 4.15}]\label{Defn: Tangent Mor}
If $\Cfrak$ is a $2$-category with tangent objects $(X,\Tbb)$ and $(Y,\Sbb)$, then a \emph{(lax) morphism of tangent objects} $(X,\Tbb) \to  (Y,\Sbb)$ is a pair $(f,\alpha)$ where:
\begin{enumerate}
    \item $f\colon X \to Y$ is a $1$-cell in $\Cfrak$;
    \item $\alpha$ is a $2$-cell:
    \[
    \begin{tikzcd}
    X \ar[r, ""{name = U}]{}{T} \ar[d, swap]{}{f} & X \ar[d]{}{f} \\
    Y \ar[r, swap, ""{name = D}]{}{S} & Y
    \ar[from = U, to = D, Rightarrow, shorten <= 4pt, shorten >= 4pt]{}{\alpha}
    \end{tikzcd} 
    \]
    \item The diagrams of $1$-cells and $2$-cells
    \[
    \begin{tikzcd}
    f \circ T \ar[r]{}{\alpha} \ar[dr, swap]{}{f \ast \!\quot{p}{T}} & S \circ f \ar[d]{}{\!\quot{p}{S} \ast f} \\
     & f
    \end{tikzcd}\quad
    \begin{tikzcd}
    f \ar[r]{}{f \ast \!\quot{0}{T}} \ar[dr, swap]{}{\!\quot{0}{S} \ast f} & f \circ T \ar[d]{}{\alpha} \\
 & S \circ f
    \end{tikzcd}\quad
    \begin{tikzcd}
    f \circ T_2 \ar[r]{}{\alpha_2} \ar[d, swap]{}{f \ast \!\quot{\operatorname{add}}{T}} & S_2 \circ f \ar[d]{}{\!\quot{\operatorname{add}}{S} \ast f} \\ 
    f \circ T \ar[r, swap]{}{\alpha} & S \circ f
    \end{tikzcd}
    \]
    \[
    \begin{tikzcd}
    f \circ T \ar[rrrr]{}{\alpha} \ar[d, swap]{}{f \ast \!\quot{\ell}{T}} & & & & S \circ f \ar[d]{}{\!\quot{\ell}{S} \ast f} \\
    f \circ T^2 \ar[rrrr, swap]{}{(S \ast \alpha) \circ (\alpha \ast T)} & & & & S^2 \circ f
    \end{tikzcd}
    \]
    \[
\begin{tikzcd}
    f \circ T^2 \ar[rrrr]{}{(S \ast \alpha) \circ (\alpha \ast T)} \ar[d, swap]{}{f \ast \!\quot{c}{T}} & & & & S^2 \circ f \ar[d]{}{\!\quot{c}{S} \ast f} \\
    f \circ T^2 \ar[rrrr, swap]{}{(S \ast \alpha) \circ (\alpha \ast T)} & & & & S^2 \circ f
    \end{tikzcd}
    \]
    each commute in the relevant hom-categories.
\end{enumerate}
\end{dfn}

Specializing this to the $2$-category $\Bicat(\Cscr^{\op},\fCat)$ we find the following characterization of (lax) tangent morphisms of tangent indexing functors. By Theorem \ref{Thm: Tangent objects in Bicat Hom category are tangent indexing functors}, this characterizes the tangent morphisms between tangent objects in the hom-$2$-category $\Bicat(\Cscr^{\op},\fCat)$.
\begin{prop}\label{Prop: tangent morphisms between tangent index functors}
Let $F,G\colon \Cscr^{\op} \to \fCat$ be tangent indexing functors with tangent object structures $(F,T,\!\quot{p}{T},\quot{0}{T},\!\quot{\operatorname{add}}{T},\!\quot{\ell}{T},\quot{c}{T})$ and $(G,S,\!\quot{p}{S}, \!\quot{0}{S}, \!\quot{\operatorname{add}}{S}, \!\quot{\ell}{S}, \!\quot{c}{S})$, respectively. To give a tangent morphism $(h,\alpha)\colon F \to G$ it is necessary and sufficient to give a pseudonatural transformation $h\colon F \Rightarrow G$ and a modification
\[
\begin{tikzcd}
\Cscr^{\op} \ar[rrr, bend left = 30, ""{name = U}]{}{F} \ar[rrr, bend right = 30, swap, ""{name = B}]{}{G} & & &\fCat 
\ar[from = U, to = B, Rightarrow, shorten <= 4pt, shorten >= 4pt, bend right = 30, swap, ""{name = L}]{}{h \circ T} 
\ar[from = U, to = B, Rightarrow, shorten <= 4pt, shorten >= 4pt, bend left = 30, ""{name = R}]{}{S \circ h} 
\ar[from = L, to = R, symbol = \underset{\alpha}{\Rrightarrow}, swap]
\end{tikzcd}
\]
such that for all objects $X$ in $\Cscr$, the equations
\begin{align}
\!\quot{h}{X} \ast \!\quot{\left(\quot{p}{T}\right)}{X} &= \!\quot{\left(\!\quot{p}{S} \ast h\right)}{X} \circ \!\quot{\alpha}{X} \label{Eqn: Necessary eqn 0 and bundle eqns}\\
\!\quot{\alpha}{X} \circ \!\quot{\left(h \ast \!\quot{0}{T}\right)}{X}&=\!\quot{\left(\!\quot{0}{S} \ast h\right)}{X} \\
\!\quot{\left(\!\quot{\operatorname{add}}{S} \ast h\right)}{X} \circ \!\quot{\alpha_2}{X} &= \!\quot{\alpha}{X} \circ \!\quot{\left(h \ast \quot{\operatorname{add}}{T}\right)}{X} \label{Eqn: Necessary eqn addition eqn}\\
\!\quot{\left(\!\quot{\ell}{G}\right)}{X} \circ\! \quot{\alpha}{X} &= \!\quot{\left(S \ast \alpha\right)}{X} \circ \!\quot{\left(\alpha \ast T\right)}{X} \circ \!\quot{\left(\!\quot{\ell}{F}\right)}{X} \label{Eqn: Necessary eqn lift eqn} \\
\!\quot{\left(S \ast \alpha\right)}{X} \circ \!\quot{\left(\alpha \ast T\right)}{X} \circ \!\quot{\left(\!\quot{c}{T}\right)}{X} &= \!\quot{\left(\!\quot{c}{S}\right)}{X} \circ \!\quot{\left(S \ast \alpha\right)}{X} \circ \!\quot{\left(\alpha \ast T\right)}{X} \label{Eqn: Necessary eqn flip eqn}
\end{align}
hold. Additionally, $(h,\alpha)$ is strong if and only if $\alpha$ is an isomorphism.
\end{prop}
\begin{proof}
To see that the equations are sufficient, we note that by requiring that each of the equations in lines (\ref{Eqn: Necessary eqn 0 and bundle eqns}) -- (\ref{Eqn: Necessary eqn flip eqn}) hold for each object $X \in \Cscr_0$, the coherence diagrams given in Item (3) of Definition \ref{Defn: Tangent Mor} are all satisfied. 

Alternatively, to see that giving a pseudonatural transformation $h\colon F \Rightarrow G$ and modification $\alpha\colon h \circ T \Rrightarrow S \circ h$ for which the equations given in Lines (\ref{Eqn: Necessary eqn 0 and bundle eqns}) -- (\ref{Eqn: Necessary eqn flip eqn}) is necessary. First we note that in order for Items (1) and (2) of Definition \ref{Defn: Tangent Morphism} to hold, we require a pseudonatural transformation $h$ and modification $\alpha$ of the typing declared in the statement of the proposition. Finally, each of the coherence diagrams in Item (3) of Definition \ref{Defn: Tangent Mor} are described object-wise by the equations described in Lines (\ref{Eqn: Necessary eqn 0 and bundle eqns}) -- (\ref{Eqn: Necessary eqn flip eqn}), which shows their necessity. The final claim follows immediately by inspection.
\end{proof}
The proposition above allows us to deduce that giving tangent morphisms between tangent objects in $\Bicat(\Cscr^{\op},\fCat)$ is equivalent to giving a pseudonatural transformation and modification pair which is object-locally a tangent functor for every object in the base category $\Cscr$. 

\begin{cor}\label{Cor: Tangent morphisms in hom two category}
Let $F, G\colon \Cscr^{\op} \to \fCat$ be tangent indexing functors with corresponding object-local tangent structures $(F(X),\!\quot{\Tbb}{X})$ and $(G(X),\!\quot{\Sbb}{X})$, respectively. To give a tangent morphism $(h,\rho)\colon F \to G$ is to give a pseudonatural transformation and modification pair
\[
\begin{tikzcd}
\Cscr^{\op} \ar[rr, bend left = 30, ""{name = U}]{}{F} \ar[rr, swap, bend right = 30, ""{name = D}]{}{G} & & \fCat
\ar[from = U, to = D, Rightarrow, shorten <= 4pt, shorten >= 4pt]{}{h}
\end{tikzcd}\quad
\begin{tikzcd}
\Cscr^{\op} \ar[rrr, bend left = 30, ""{name = U}]{}{F} \ar[rrr, bend right = 30, swap, ""{name = B}]{}{G} & & &\fCat 
\ar[from = U, to = B, Rightarrow, shorten <= 4pt, shorten >= 4pt, bend right = 30, swap, ""{name = L}]{}{h \circ T} 
\ar[from = U, to = B, Rightarrow, shorten <= 4pt, shorten >= 4pt, bend left = 30, ""{name = R}]{}{S \circ h} 
\ar[from = L, to = R, symbol = \underset{\alpha}{\Rrightarrow}, swap]
\end{tikzcd}
\]
such that for all $X \in \Cscr_0$ the pair $(\!\quot{h}{X}, \!\quot{\alpha}{X})$ is a morphism of tangent categories. Additionally, $(h,\alpha)$ is strong if and only if each map $(\quot{h}{X}, \quot{\alpha}{X})$ is strong.
\end{cor}
\begin{proof}
In both cases, we are given the existence of a pseudonatural transformation $h\colon F \Rightarrow G$ and a modification $\alpha\colon h \circ T \Rrightarrow S \circ h$, so it suffices to verify that being a tangent morphism is equivalent to having each functor/natural transformation pair $(\!\quot{h}{X},\!\quot{\alpha}{X})$ be a morphism of tangent categories. Observe, however, that the equations (\ref{Eqn: Necessary eqn 0 and bundle eqns}) -- (\ref{Eqn: Necessary eqn lift eqn}) precisely declare that each pair $(\!\quot{h}{X}, \!\quot{\alpha}{X})$ is a morphism of tangent categories, as they are exactly the coherences declared in Definition \ref{Defn: Tangent Morphism}. Thus declaring that the equations in Lines (\ref{Eqn: Necessary eqn 0 and bundle eqns}) -- (\ref{Eqn: Necessary eqn lift eqn}) hold for all $X \in \Cscr_0$ is equivalent to $(\!\quot{h}{X}, \!\quot{\alpha}{X})$ being a morphism of tangent categories for all objects. As in Proposition \ref{Prop: tangent morphisms between tangent index functors}, the final claim of the corollary is clear from inspection.
\end{proof}
\begin{Theorem}\label{Thm: Tangent morphisms pseudoconify to tangent functors}
Let $F,G\colon \Cscr^{\op} \to \fCat$ be tangent indexing functors and let $(h,\alpha)\colon F \Rightarrow G$ be a tangent morphism between them in $\Bicat(\Cscr^{\op},\fCat)$. Then the pair $(\underline{h},\underline{\alpha})\colon \PC(F) \to \PC(G)$ is a morphism of tangent categories.
\end{Theorem}
\begin{proof}
Note that the content of this theorem amounts to proving that the diagrams of functors and natural transformations
\begin{equation}\label{Eqn: Diagrams to check for tangent mor1}
\begin{tikzcd}
\underline{h} \circ \underline{T} \ar[r]{}{\underline{\alpha}} \ar[dr, swap]{}{\underline{h} \ast \underline{p}_{\Tbb}} & \underline{S} \circ \underline{h} \ar[d]{}{\underline{p}_{\Sbb} \ast \underline{h}} \\
 & \underline{h}
\end{tikzcd}\quad
\begin{tikzcd}
\underline{h} \ar[r]{}{\underline{h} \ast \underline{0}_{\Tbb}} \ar[dr, swap]{}{\underline{0}_{\Sbb} \ast \underline{h}} & \underline{h} \circ \underline{T} \ar[d]{}{\underline{\alpha}} \\
 & \underline{S} \circ \underline{h}
\end{tikzcd}\quad
\begin{tikzcd}
\underline{h} \circ \underline{T}_2 \ar[rr]{}{\underline{h} \ast \underline{\alpha}_2} \ar[d, swap]{}{\underline{h} \ast \underline{\operatorname{add}}_{\Tbb}} & & \underline{S}_2 \circ \underline{h} \ar[d]{}{\underline{\operatorname{add}}_{\Sbb} \ast \underline{h}} \\
\underline{h} \ast \underline{T} \ar[rr, swap]{}{\underline{\alpha}} & & \underline{S} \circ \underline{h}
\end{tikzcd}
\end{equation}
\begin{equation}\label{Eqn: Diagrams to check for tangent mor2}
\begin{tikzcd}
\underline{h} \circ \underline{T} \ar[rr]{}{\underline{\alpha}} \ar[d, swap]{}{\underline{h} \ast \underline{\ell}_{\Tbb}} & & \underline{S} \circ \underline{h} \ar[d]{}{\underline{\ell}_{\Sbb} \ast \underline{h}} \\
\underline{h} \circ \underline{T}^2 \ar[rr, swap]{}{(\underline{S} \ast \underline{h}) \circ (\underline{\alpha} \circ \underline{T})} & & \underline{S}^2 \circ \underline{h}
\end{tikzcd}\qquad
\begin{tikzcd}
\underline{h} \circ \underline{T}^2 \ar[rr]{}{(\underline{S} \ast \underline{h}) \circ (\underline{\alpha} \circ \underline{T})} \ar[d, swap]{}{\underline{h} \ast \underline{c}_{\Tbb}} & & \underline{S}^2 \circ \underline{h} \ar[d]{}{\underline{c}_{\Sbb} \ast \underline{h}} \\
\underline{h} \circ \underline{T}^2 \ar[rr, swap]{}{(\underline{S} \ast \underline{h}) \circ (\underline{\alpha} \circ \underline{T})} & & \underline{S}^2 \circ \underline{h}
\end{tikzcd}
\end{equation}
all must commute. However, to show this it suffices to prove that for any object $A = \lbrace \!\quot{A}{X} \; \left. \right| \; X \in \Cscr_0 \rbrace$ of $\PC(F)$, the $X$-local versions of the diagrams above,
\[
\begin{tikzcd}
\left(\!\quot{h}{X} \circ\! \quot{T}{X}\right)\left(\!\quot{A}{X}\right) \ar[r]{}{\!\quot{\alpha}{X}} \ar[dr, swap]{}{\left(\!\quot{h}{X} \ast\! \quot{{p}_{\Tbb}}{X}\right)} & \left(\!\quot{S}{X} \circ \!\quot{h}{X}\right)\left(\!\quot{A}{X}\right) \ar[d]{}{\left(\!\quot{{p}_{\Sbb}}{X} \ast \!\quot{h}{X}\right)} \\
 & \left(\!\quot{h}{X}\right)\left(\!\quot{A}{X}\right)
\end{tikzcd}\quad
\begin{tikzcd}
\!\quot{h}{X}\left(\!\quot{A}{X}\right) \ar[rr]{}{\left(\!\quot{h}{X} \ast \!\quot{{0}_{\Tbb}}{X}\right)} \ar[drr, swap]{}{\left(\!\quot{{0}_{\Sbb}}{X} \ast \!\quot{h}{X}\right)} & & \left(\!\quot{h}{X} \circ \!\quot{T}{X}\right)\left(\!\quot{A}{X}\right) \ar[d]{}{\!\quot{\alpha}{X}} \\
 & & \left(\!\quot{S}{X} \circ \!\quot{h}{X}\right)\left(\!\quot{A}{X}\right)
\end{tikzcd}
\]
\[
\begin{tikzcd}
\left(\!\quot{h}{X} \circ \left(\!\quot{T}{X}\right)_2\right)\left(\!\quot{A}{X}\right) \ar[rr]{}{\left(\!\quot{h}{X} \ast \left(\!\quot{\alpha}{X}\right)_2\right)} \ar[d, swap]{}{\left(\!\quot{h}{X} \ast \!\quot{\operatorname{add}_{\Tbb}}{X}\right)} & & \left(\left(\!\quot{S}{X}\right)_2 \circ \quot{h}{X}\right)\left(\!\quot{A}{X}\right) \ar[d]{}{\left(\!\quot{\operatorname{add}_{\Sbb}}{X} \ast \!\quot{h}{X}\right)} \\
\left(\!\quot{h}{X} \ast \!\quot{T}{X}\right)\left(\!\quot{A}{X}\right) \ar[rr, swap]{}{\!\quot{\alpha}{X}} & & \left(\!\quot{S}{X} \circ\! \quot{h}{X}\right)\left(\!\quot{A}{X}\right)
\end{tikzcd}
\]
\[
\begin{tikzcd}
\left(\!\quot{h}{X} \circ \!\quot{T}{X}\right)\left(\!\quot{A}{X}\right) \ar[rrr]{}{\!\quot{\alpha}{X}} \ar[d, swap]{}{\left(\!\quot{h}{X} \ast \!\quot{{\ell}_{\Tbb}}{X}\right)} & & & \left(\!\quot{S}{X} \circ \!\quot{h}{X}\right)\left(\!\quot{A}{X}\right) \ar[d]{}{\left(\!\quot{{\ell}_{\Sbb}}{X} \ast\! \quot{h}{X}\right)} \\
\left(\!\quot{h}{X} \circ \left(\!\quot{T}{X}\right)^2\right)\left(\!\quot{A}{X}\right) \ar[rrr, swap]{}{(\!\quot{S}{X} \ast \!\quot{h}{X}) \circ (\!\quot{\alpha}{X} \circ \!\quot{T}{X})} & & & \left(\left(\!\quot{S}{X}\right)^2 \circ\! \quot{h}{X}\right)\left(\!\quot{A}{X}\right)
\end{tikzcd}
\]
\[
\begin{tikzcd}
\left(\!\quot{h}{X} \circ \left(\!\quot{T}{X}\right)^2\right)\left(\!\quot{A}{X}\right) \ar[rrr]{}{(\!\quot{S}{X} \ast \!\quot{h}{X}) \circ (\!\quot{\alpha}{X} \circ \!\quot{T}{X})} \ar[d, swap]{}{\left(\!\quot{h}{X} \ast \!\quot{{c}_{\Tbb}}{X}\right)} & & & \left(\left(\!\quot{S}{X}\right)^2 \circ \!\quot{h}{X}\right)\left(\!\quot{A}{X}\right) \ar[d]{}{\left(\!\quot{{c}_{\Sbb}}{X} \ast \!\quot{h}{X}\right)} \\
\left(\!\quot{h}{X} \circ \left(\!\quot{T}{X}\right)^2\right)\left(\!\quot{A}{X}\right) \ar[rrr, swap]{}{(\!\quot{S}{X} \ast \!\quot{h}{X}) \circ (\!\quot{\alpha}{X} \circ \!\quot{T}{X})} & & & \left(\left(\!\quot{S}{X}\right)^2 \circ \!\quot{h}{X}\right)\left(\!\quot{A}{X}\right)
\end{tikzcd}
\]
all must commute; note that for readability, we have omitted the $\!\quot{A}{X}$ subscript for the component of each natural transformation. However, each of these diagrams holds by virtue of the fact that $(h,\alpha)$ is a tangent morphism, an application of Corollary \ref{Cor: Tangent morphisms in hom two category}, and the fact that each $\!\quot{A}{X} \in F(X)_0$. As such, it follows that the diagrams expressed in Lines (\ref{Eqn: Diagrams to check for tangent mor1}) and (\ref{Eqn: Diagrams to check for tangent mor2}) commute and so that $(\underline{h},\underline{\alpha})$ is a morphism of tangent categories.
\end{proof}

We now characterize the tangent transformations of \cite{Marcello} in $\Bicat(\Cscr^{\op},\fCat)$ and then prove that they become tangent natural transformations after applying the pseudocone $2$-functor $\PC(-)$. For this, however, we will start by recalling what it means to be a tangent transformation.
\begin{dfn}[{\cite{Marcello}, Definition 4.16}]
Let $\Cfrak$ be a $2$-category with tangent objects $(X,\Tbb), (Y,\Sbb)$ and with tangent morphisms $(h,\alpha), (k,\beta)\colon (X,\Tbb) \to (Y,\Sbb)$. A tangent transformation $\rho\colon (h,\alpha) \Rightarrow (k,\beta)$, described visually as
\[
\begin{tikzcd}
(X,\Tbb) \ar[rr, bend left = 30, ""{name = U}]{}{(h,\alpha)} \ar[rr, bend right = 30, swap, ""{name = D}]{}{(k,\beta)} & & (Y,\Sbb)
\ar[from = U, to = D, Rightarrow, shorten <= 4pt, shorten >= 4pt]{}{\rho}
\end{tikzcd}
\]
is given by a $2$-cell $\rho\colon h \Rightarrow k$ for which the composite $2$-cell
\[
\begin{tikzcd}
X \ar[rrrr, bend left = 60, ""{name = U}]{}{h \circ T} \ar[rrrr, ""{name = M}]{}[description]{S \circ h} \ar[rrrr, swap, bend right = 60, ""{name = D}]{}{S \circ k} & & & &  Y
\ar[from = U, to = M, Rightarrow, shorten <= 4pt, shorten >= 6pt]{}{\alpha}
\ar[from = M, to = D, Rightarrow, shorten <= 6pt, shorten >= 4pt]{}{\id_S \ast \rho}
\end{tikzcd}
\]
is equal to the composite $2$-cell:
\[
\begin{tikzcd}
X \ar[rrrr, bend left = 60, ""{name = U}]{}{h \circ T} \ar[rrrr, ""{name = M}]{}[description]{k \circ T} \ar[rrrr, bend right = 60, swap, ""{name = D}]{}{S \circ k} & & & & Y
\ar[from = U, to = M, Rightarrow, shorten <= 4pt, shorten >= 6pt]{}{\rho \ast \id_T}
\ar[from = M, to = D, Rightarrow, shorten <= 6pt, shorten >= 4pt]{}{\beta}
\end{tikzcd}
\]
\end{dfn}

\begin{prop}\label{Prop: Tangent two morphisms equivalent conditions in hom two cat}
To give a tangent transformation in $\Bicat(\Cscr^{\op},\fCat)$ of type
\[
\begin{tikzcd}
F \ar[rr, bend left = 30, ""{name = U}]{}{(h,\alpha)} \ar[rr, bend right = 30, swap, ""{name = D}]{}{(k,\beta)} & & G
\ar[from = U, to = D, Rightarrow, shorten <= 4pt, shorten >= 4pt]{}{\rho}
\end{tikzcd}
\]
for tangent indexing functors $F, G$ and tangent morphisms $(h,\alpha), (k,\beta)$ between them is equivalent to giving a modification $\rho\colon h \Rrightarrow k$ such that for all $X \in \Cscr_0$, the component natural transformation $\!\quot{\rho}{X}$ is a tangent natural transformation.
\end{prop}
\begin{proof}
This is similar to proving Corollary \ref{Cor: Tangent morphisms in hom two category}, and so some details are omitted. To this end, note that asking for the composites of modifications $(\id_{S} \ast \rho) \circ \alpha = \beta \circ (\rho \ast \id_T)$ is equivalent to asking that for all $X \in \Cscr_0$, the equations
\[
\left(\left(\id_{\!\quot{S}{X}}\right) \ast \!\quot{\rho}{X}\right) \circ \!\quot{\alpha}{X} = \!\quot{\left(\id_{S} \ast \rho\right)}{X} \circ \!\quot{\alpha}{X} = \!\quot{\beta}{X} \circ \!\quot{\left(\rho \ast \id_T\right)}{X} = \!\quot{\beta}{X} \circ \left(\!\quot{\rho}{X} \ast \left(\id_{\!\quot{T}{X}}\right)\right)
\]
hold. However, each such equation is equivalent to $\!\quot{\rho}{X}$ being a tangent natural transformation, which gives the proposition.
\end{proof}
\begin{Theorem}\label{Thm: tangent 2 morphisms are PCd to tangent transformations}
Let $F,G\colon\Cscr^{\op} \to \fCat$ be a tangent indexing functors with tangent morphisms $(h,\alpha), (k,\beta)\colon F \rightarrow G$ and with a tangent $2$-cell $\rho\colon (h,\alpha) \Rightarrow (k,\beta)$. Then the induced natural transformation $\underline{\rho}$ fitting into the diagram
\[
\begin{tikzcd}
\PC(F) \ar[rr, bend left = 30, ""{name = U}]{}{\underline{h}} \ar[rr, bend right = 30, swap, ""{name = D}]{}{\underline{k}} & & \PC(G)
\ar[from = U, to = D, Rightarrow, shorten <= 4pt, shorten >= 4pt]{}{\underline{\rho}}
\end{tikzcd}
\]
is a tangent natural transformation $\underline{\rho}\colon (\underline{h},\underline{\alpha}) \Rightarrow (\underline{k},\underline{\beta})$.
\end{Theorem}
\begin{proof}
Note that to ensure that $\underline{\rho}$ is a tangent transformation, we must show that the equation 
\[
(\id_{\underline{S}} \ast \underline{\rho}) \circ \underline{\alpha} = \underline{\beta} \circ (\underline{\rho} \ast \id_{\underline{T}})
\]
holds. To this end, let $A = \lbrace \!\quot{A}{X} \; \left. \right| \; X \in \Cscr_0 \rbrace$ be an object in $\PC(F)$. By Proposition \ref{Prop: Tangent two morphisms equivalent conditions in hom two cat} we know that each $\!\quot{\rho}{X}$ is a tangent natural transformation. As such, because morphisms in $\PC(F)$ are determined by their object-local actions, it suffices to show the above identity holds for all $X \in \Cscr_0$ and for all component morphisms of the natural transformation. We now compute that the $\!\quot{A}{X}$-component of $(\id_{\underline{S}} \ast \underline{\rho})  \circ \underline{\alpha}$ satisfies
\begin{align*}
\left((\id_{\underline{S}} \ast \underline{\rho}) \circ \underline{\alpha}\right)_{\!\quot{A}{X}} &= \left(\id_{\underline{S}} \ast \underline{\rho}\right)_{\!\quot{A}{X}} \circ \underline{\alpha}_{\!\quot{A}{X}} = \left(\left(\id_{\!\quot{S}{X}}\right) \ast \!\quot{\rho}{X}\right)_{\!\quot{A}{X}} \circ \left(\!\quot{\alpha}{X}\right)_{\!\quot{A}{X}} \\
&= \left(\!\quot{\beta}{X}\right)_{\!\quot{A}{X}} \circ \left(\!\quot{\rho}{X} \ast \left(\id_{\!\quot{T}{X}}\right)\right)_{\!\quot{A}{X}} 
 =\left(\underline{\beta}\right)_{\!\quot{A}{X}} \circ \left(\underline{\rho} \ast \id_{\underline{T}}\right)_{\!\quot{A}{X}} \\
 &= \left(\underline{\beta} \circ \left(\underline{\rho} \ast \id_{\underline{T}}\right)\right)_{\!\quot{A}{X}}
\end{align*}
exactly because $\!\quot{\rho}{X}$ is a tangent transformation. However, this implies that
\[
\left((\id_{\underline{S}} \ast \underline{\rho}) \circ \underline{\alpha}\right)_{A} = \left(\underline{\beta} \circ (\underline{\rho} \ast \id_{\underline{T}})\right)_{A}
\]
and so proves that $\underline{\rho}$ is a tangent transformation.
\end{proof}

We now collect our results above (namely Theorems \ref{Thm: Pre-Equivariant Tangent Category}, \ref{Thm: Tangent morphisms pseudoconify to tangent functors}, and \ref{Thm: tangent 2 morphisms are PCd to tangent transformations}) to give a strict $2$-functor 
\[
\PC\colon \mathfrak{Tan}\left(\Bicat\left(\Cscr^{\op},\fCat\right)\right) \to \fTan.
\]

\begin{Theorem}\label{Thm: PC of tangent objects in hom two category takes values in tangent categories}
The strict $2$-functor $\PC(-)$ lifts to a strict $2$-functor $\PC\colon \mathfrak{Tan}\left(\Bicat\left(\Cscr^{\op},\fCat\right)\right) \to \fTan$ and makes the diagram of $2$-categories
\[
\begin{tikzcd}
\Bicat\left(\Cscr^{\op},\fCat\right) \ar[rr]{}{\PC(-)} & & \fCat \\
\fTan\left(\Bicat\left(\Cscr^{\op},\fCat\right)\right) \ar[u]{}{\Forget} \ar[rr, swap]{}{\PC(-)}  &  &\fTan \ar[u, swap]{}{\Forget}
\end{tikzcd}
\]
commute strictly.
\end{Theorem}
\begin{proof}
The strictness of $\PC$ is described in \cite[Lemma 4.1.11, Page 80]{GeneralGeoffThesis}. That it sends tangent objects in $\Bicat(\Cscr^{\op},\fCat)$ to tangent categories is a combination of Theorems \ref{Thm:  Pre-Equivariant Tangent Category} and \ref{Thm:  Tangent objects in Bicat Hom category are tangent indexing functors}, while the statement that $\PC$ takes tangent morphisms to tangent functors is Theorem \ref{Thm: Tangent morphisms pseudoconify to tangent functors} and that $\PC$ takes tangent $2$-cells to tangent natural transformations is Theorem \ref{Thm: tangent 2 morphisms are PCd to tangent transformations}. As such, it follows that $\PC$ restricts to
\[
\PC\colon \mathfrak{Tan}\left(\Bicat\left(\Cscr^{\op},\fCat\right)\right) \to \fTan.
\]
Finally, the commutativity of the diagram
\[
\begin{tikzcd}
\Bicat\left(\Cscr^{\op},\fCat\right) \ar[rr]{}{\PC(-)} & & \fCat \\
\fTan\left(\Bicat\left(\Cscr^{\op},\fCat\right)\right) \ar[u]{}{\Forget} \ar[rr, swap]{}{\PC(-)} & & \fTan \ar[u, swap]{}{\Forget}
\end{tikzcd}
\]
is immediate by construction.
\end{proof}

We now close this section with a discussion of pseudolimits indexed by $1$-categories in $\fTan$. In particular, we'll show that if $F\colon \Cscr^{\op} \to \fCat$ is a tangent indexing functor then $\PC(F)$ is the pseudolimit of the induced diagram $\tilde{F}\colon \Cscr^{\op} \to \fTan$. Afterwards, we'll discuss how to use the strict $2$-functor $\PC(-)$ to show that the forgetful functor $\Forget\colon \fTan \to \fCat$ reflects and preserves (and hence creates) pseudolimits indexed by $1$-categories in $\fCat$.

\begin{Theorem}\label{Thm: Pseudolimits in Tancat}
Let $F\colon \Cscr^{\op} \to \fCat$ be a tangent indexing functor. The tangent category $\PC(F)$ is the pseudolimit in $\fTan$ and in $\fTan_{\operatorname{strong}}$ of shape $F$.
\end{Theorem}
\begin{rmk}
It is worth noting that in this proof we are showing that even though a tangent indexing functor $F$ is equivalently described as a pseudofunctor $\underline{F}:\Cscr^{\op} \to \fTan_{\operatorname{strong}}$, $\PC(F)$ is a pseudolimit in \emph{both} the $2$-category of $\fTan$ of tangent categories and (potentially) lax morphisms \emph{and} in $\fTan_{\operatorname{strong}}$.
\end{rmk}
\begin{proof}
Our strategy here is to mimic the proof of the fact that $\PC(F)$ is the pseudolimit in $\fCat$ presented in \cite[Theorem 2.3.16, Page 43]{GeneralGeoffThesis} and adapt it to the tangent categorical situation and then finally conclude from this that if all relevant were in $\fTan_{\operatorname{strong}}$, so too are all comparisons. We do this in four main steps:
\begin{enumerate}
    \item We show that for all $X \in \Cscr_0$, the functor $\operatorname{pr}_X\colon \PC(F) \to F(X)$ is a strict tangent functor.
    \item We show that for all morphisms $f\colon X \to Y$ in $\Cscr$, the $2$-cell
    \[
    \begin{tikzcd}
        & \PC(F) \ar[dr, ""{name = R}]{}{\operatorname{pr}_Y} \ar[dl, swap, ""{name = L}]{}{\operatorname{pr}_X} \\
    F(X) & & F(Y) \ar[ll]{}{F(f)}
    \ar[from = R, to = L, Rightarrow, shorten <= 10pt, shorten >= 10pt]{}{\!\quot{\alpha}{f}}
    \end{tikzcd}
    \]
    is a tangent natural transformation $\!\quot{\alpha}{f}\colon (F(f) \circ \operatorname{pr}_X, \!\quot{T}{f} \ast \operatorname{pr}_X) \Rightarrow (\operatorname{pr}_Y, \id)$.
    \item We will show that if there is a tangent category $(\Dscr,\Sbb)$ with tangent functors $\!\quot{G}{X}\colon \Dscr \to F(X)$ for all objects $X \in \Cscr_0$ and tangent natural transformations
    \[
    \begin{tikzcd}
        & \Dscr \ar[dr, ""{name = R}]{}{\!\quot{G}{Y}} \ar[dl, swap, ""{name = L}]{}{\!\quot{G}{X}} \\
    F(X) & & F(Y) \ar[ll]{}{F(f)}
    \ar[from = R, to = L, Rightarrow, shorten <= 10pt, shorten >= 10pt]{}{\!\quot{\beta}{f}}
    \end{tikzcd}
    \]
    which vary pseudofunctorially in $\Cscr^{\op}$, there is a tangent functor $G\colon \Dscr \to \PC(F)$ for which the $2$-cell
    \[
    \begin{tikzcd}
     & \Dscr \ar[d]{}[description]{G} \ar[ddr, bend left = 20]{}{\!\quot{G}{Y}} \ar[ddl, bend right = 20, swap]{}{\!\quot{G}{X}} \\
     & \PC(F) \ar[dr, ""{name = R}]{}{\operatorname{pr}_X} \ar[dl, swap, ""{name = L}]{}{\operatorname{pr}_Y} \\
     F(X) & & F(Y) \ar[ll]{}{F(f)}
     \ar[from = R, to = L, Rightarrow, shorten <= 10pt, shorten >= 10pt]{}{\!\quot{\alpha}{f}}
    \end{tikzcd}
    \]
    pastes to the $2$-cell:
     \[
    \begin{tikzcd}
        & \Dscr \ar[dr, ""{name = R}]{}{\!\quot{G}{Y}} \ar[dl, swap, ""{name = L}]{}{\!\quot{G}{X}} \\
    F(X) & & F(Y) \ar[ll]{}{F(f)}
    \ar[from = R, to = L, Rightarrow, shorten <= 10pt, shorten >= 10pt]{}{\!\quot{\beta}{f}}
    \end{tikzcd}
    \]
    \item Assume we have a tangent category $(\Dscr,\Sbb)$ with two tangent morphisms $(G,\beta)\colon \Dscr \to \PC(F)$ and $(H,\gamma)\colon \Dscr \to \PC(F)$ and a family of tangent transformations
    \[
    \begin{tikzcd}
    \Dscr \ar[rr, ""{name = U}]{}{(G,\beta)} \ar[d, swap]{}{(H, \gamma)} & & \PC(F) \ar[d]{}{\operatorname{pr}_X} \\
    \PC(F) \ar[rr, swap, ""{name = D}]{}{\operatorname{pr}_X} & & F(X)
    \ar[from = U, to = D, Rightarrow, shorten <= 4pt, shorten >= 4pt]{}{\!\quot{\rho}{X}}
    \end{tikzcd}
    \]
    for all $X \in \Cscr_0$ subject to the coherences that for all $f\colon X \to Y$ in $\Cscr$,
    \[
    \!\quot{\rho}{X} \circ \left(\!\quot{\alpha}{f} \ast G\right) = \left(\!\quot{\alpha}{f} \ast H\right) \circ \left(F(f) \ast \!\quot{\rho}{Y}\right).
    \]
    We will then show that there exists a unique tangent transformation $P\colon (G,\beta) \Rightarrow (H,\gamma)$ for which $\operatorname{pr}_X \ast P = \!\quot{\rho}{X}$.
\end{enumerate}
We begin by proving (1). Fix an object $X \in \Cscr_0$ and recall that the functor $\operatorname{pr}_X\colon \PC(F) \to F(X)$ is defined by
\[
\operatorname{pr}_X\left(\left\lbrace \!\quot{A}{Y} \; \left. \right| \; Y \in \Cscr_0 \right\rbrace, \left\lbrace \tau_f^{A} \; \left. \right| \; f \in \Cscr_1 \right\rbrace\right) = \!\quot{A}{X}.
\]
We then find that for all objects $A$ of $\PC(F)$,
\begin{align*}
    \left(\operatorname{pr}_X \circ \underline{T}\right)(A) &= \left(\operatorname{pr}_X\left(\underline{T}\left\lbrace \!\quot{A}{Y} \; \left. \right| \; Y \in \Cscr_0 \right\rbrace \right)\right) = \operatorname{pr}_X\left\lbrace \!\quot{T}{X}\left(\!\quot{A}{X}\right)\right\rbrace = \!\quot{T}{X}\left(\!\quot{A}{X}\right) \\
    &=\! \quot{T}{X}\left(\operatorname{pr}_X\left\lbrace \!\quot{A}{Y} \; \left. \right| \; Y \in \Cscr_0\right\rbrace\right) = \left(\!\quot{T}{X} \circ \operatorname{pr}_X\right)(A),
\end{align*}
proving that $\operatorname{pr}_X\colon (\PC(F),\Tbb) \to (F(X),\!\quot{\Tbb}{X})$ is a strict tangent functor.

We now prove (2). Let $f\colon X \to Y$ be a morphism in $\Cscr$ and consider the $2$-cell:
\[
  \begin{tikzcd}
        & \PC(F) \ar[dr, ""{name = R}]{}{\operatorname{pr}_Y} \ar[dl, swap, ""{name = L}]{}{\operatorname{pr}_X} \\
    F(X) & & F(Y) \ar[ll]{}{F(f)}
    \ar[from = R, to = L, Rightarrow, shorten <= 10pt, shorten >= 10pt]{}{\!\quot{\alpha}{f}}
    \end{tikzcd}
\]
Note that the natural transformation $\!\quot{\alpha}{f}$ is given via the assignment
\[
\left(\!\quot{\alpha}{f}\right)_A := \left\lbrace \tau_{f}^{A} \; \left. \right| \; f \in \Cscr_1\right\rbrace
\]
for all objects $A$ in $\PC(F)$; it is proved to be natural in \cite[Theorem 2.3.16, Page 81]{GeneralGeoffThesis}. With this in mind, we now must prove that $\!\quot{\alpha}{f}$ constitutes a tangent natural transformation $(F(f) \circ \operatorname{pr}_X, \!\quot{T}{f} \ast \id) \to (\operatorname{pr}_Y,\id)$. To this end, we compute that on one hand
\[
\id_{\!\quot{T}{X}} \circ \left(\!\quot{\alpha}{f} \ast \id_{\underline{T}}\right)_{A} = \id_{\!\quot{T}{X}} \circ \left(\!\quot{\alpha}{f}\right)_{\underline{T}(A)} = \tau_f^{\underline{T}A}.
\]
On the other hand,
\[
\left(\id_{\!\quot{T}{X}} \ast \!\quot{\alpha}{f}\right)_{A} \circ \left(\!\quot{T}{f} \ast \id_{\PC(F)}\right)_{A} = \!\quot{T}{X}\left(\tau_{f}^{A}\right) \circ \left(\!\quot{T}{f}\right)_{\!\quot{A}{Y}} =  \tau_f^{\underline{T}A}
\]
where the identity in the last equality sign is described in Remark \ref{Rmk: Explicit form of transition isomorphisms}. As such, it follows that $\!\quot{\alpha}{f}$ is a tangent natural transformation.

We now prove (3). Assume that we have a tangent category $(\Dscr,\Sbb)$ together with tangent functors $(\!\quot{G}{X},\!\quot{\alpha}{X})\colon \Dscr \to F(X)$ for all objects $X \in \Cscr_0$ and tangent natural transformations
    \[
    \begin{tikzcd}
        & \Dscr \ar[dr, ""{name = R}]{}{(\!\quot{G}{Y},\!\quot{\alpha}{Y})} \ar[dl, swap, ""{name = L}]{}{(\!\quot{G}{X},\!\quot{\alpha}{X})} \\
    F(X) & & F(Y) \ar[ll]{}{(F(f),\!\quot{T}{f})}
    \ar[from = R, to = L, Rightarrow, shorten <= 10pt, shorten >= 10pt]{}{\!\quot{\beta}{f}}
    \end{tikzcd}
    \]
    which vary pseudofunctorially in $\Cscr^{\op}$. Again guided by \cite[Theorem 2.3.16]{GeneralGeoffThesis}, we consider the comparison functor $G\colon \Dscr \to \PC(F)$ defined by the object assignment
    \[
    G(Z) := \left\lbrace \!\quot{G}{X}\left(Z\right) \; \left. \right| \; X \in \Cscr_0 \right\rbrace
    \]
    with transition isomorphisms
    \[
    T_{G(Z)} := \left\lbrace\! \quot{\beta}{f}_{Z} \; \left. \right| \; Z \in \Cscr_0 \right\rbrace.
    \]
    Let us now show that $G$ is a part of a morphism of tangent categories. To do this define the natural transformation $\alpha\colon G \circ S \Rightarrow \underline{T} \circ G$ by setting, for all $Z \in \Dscr_0$,
    \[
    \alpha_A := \left\lbrace \!\quot{\alpha_Z}{X} \; \left. \right| \; Z \in \Cscr_0\right\rbrace.
    \]
    If we can show that $\alpha$ is a morphism in the category $\PC(F)$, the fact that $\alpha$ is a natural transformation follows immediately from the naturality of the $\!\quot{\alpha}{X}$. Similarly, if we know that $\alpha_Z$ is a morphism in $\PC(F)$, then the fact that $\alpha$ fits into a tangent morphism $(G,\alpha)$ follows from the fact that each $(\!\quot{G}{X},\!\quot{\alpha}{X})$ is a morphism of tangent categories. To this end, note that we must prove that for any $f\colon X \to Y$ in $\Cscr$, the diagram
    \[
    \begin{tikzcd}
    F(f)\left(\left(\!\quot{G}{Y} \circ S\right)(Z)\right) \ar[d, swap]{}{\tau_{f}^{G(SZ)}} \ar[rrr]{}{F(f)\left(\!\quot{\alpha}{Y}_Z\right)} & & & F(f)\left(\left(\!\quot{T}{Y} \circ \!\quot{G}{Y}\right)(Z)\right) \ar[d]{}{\tau_f^{\underline{T}(GZ)}} \\
    \left(\!\quot{G}{X} \circ S\right)(Z) \ar[rrr, swap]{}{\!\quot{\alpha}{X}_Z} & & & \left(\!\quot{T}{X} \circ \!\quot{G}{X}\right)(Z)
    \end{tikzcd}
    \]
    commutes. Because on one hand we have that
    \[
    \tau_f^{G(SZ)} = \!\quot{\beta}{f}^{SZ} = \left(\!\quot{\beta}{f} \ast \id_{S}\right)_Z
    \]
    and that
    \[
    \tau_f^{\underline{T}(GZ)} = \!\quot{T}{X}\left(\!\quot{\beta}{f}^{Z}\right) \circ \left(\!\quot{T}{f}\right)_{\!\quot{G}{Y}(Z)}.
    \]
    So, using that $\!\quot{\beta}{f}$ is a tangent natural transformation $\left(F(f) \circ \!\quot{G}{Y},\!\quot{T}{f} \ast \!\quot{\alpha}{Y}\right) \Rightarrow \left(\!\quot{G}{X}, \!\quot{\alpha}{X}\right)$, we compute that
    \begin{align*}
    \!\quot{\alpha}{X}_Z \circ \tau_{f}^{G(SZ)} &= \!\quot{\alpha}{X}_Z \circ \!\quot{\beta}{f}^{SZ} = \!\quot{\alpha}{X}_Z \circ \left(\!\quot{\beta}{f} \ast \id_S\right)_{Z} = \left(\!\quot{\alpha}{X} \circ \left(\!\quot{\beta}{f} \ast \id_S\right)\right)_Z \\
    &= \left(\left(\!\quot{T}{X} \ast \!\quot{\beta}{f}\right) \circ \left(\!\quot{T}{f} \ast \!\quot{\alpha}{Y}\right)\right)_Z = \left(\!\quot{T}{X} \ast \!\quot{\beta}{f}\right)_Z \circ \left(\!\quot{T}{f} \ast \!\quot{\alpha}{Y}\right)_Z \\
    &= \!\quot{T}{X}\left(\!\quot{\beta}{f}^{Z}\right) \circ \left(\!\quot{T}{f} \ast \!\quot{\alpha}{Y}\right)_Z = \!\quot{T}{X}\left(\!\quot{\beta}{f}^{Z}\right) \circ \!\quot{T}{f}_{\!\quot{G}{Y}(Z)} \circ F(f)\left(\!\quot{\alpha}{Y}_Z\right) \\
    &= \tau_{f}^{\underline{T}(GZ)} \circ F(f)\left(\!\quot{\alpha}{Y}_Z\right).
    \end{align*}
    This shows that $\alpha_Z$ is a morphism in $\PC(F)$ and hence proves that $(G,\alpha)\colon \Dscr \to \PC(F)$ is a morphism of tangent categories. Finally, that $(G,\alpha)$ makes the pasting diagram
    \[
     \begin{tikzcd}
     & \Dscr \ar[d]{}[description]{G} \ar[ddr, bend left = 20]{}{\!\quot{G}{Y}} \ar[ddl, bend right = 20, swap]{}{\!\quot{G}{X}} \\
     & \PC(F) \ar[dr, ""{name = R}]{}{\operatorname{pr}_X} \ar[dl, swap, ""{name = L}]{}{\operatorname{pr}_Y} \\
     F(X) & & F(Y) \ar[ll]{}{F(f)}
     \ar[from = R, to = L, Rightarrow, shorten <= 10pt, shorten >= 10pt]{}{\!\quot{\alpha}{f}}
    \end{tikzcd}
    \]
    equal to
    \[
    \begin{tikzcd}
        & \Dscr \ar[dr, ""{name = R}]{}{(\!\quot{G}{Y},\!\quot{\alpha}{Y})} \ar[dl, swap, ""{name = L}]{}{(\!\quot{G}{X},\!\quot{\alpha}{X})} \\
    F(X) & & F(Y) \ar[ll]{}{(F(f),\quot{T}{f})}
    \ar[from = R, to = L, Rightarrow, shorten <= 10pt, shorten >= 10pt]{}{\!\quot{\beta}{f}}
    \end{tikzcd}
    \]
    is immediate from definition of all the morphisms in sight. As such, we have that $\PC(F)$ is a pseudolimit in $\fTan$.

    We now prove (4). Assume we have a tangent category $(\Dscr,\Sbb)$ with two tangent morphisms $(G,\beta)\colon \Dscr \to \PC(F)$ and $(H,\gamma)\colon \Dscr \to \PC(F)$ and a family of tangent transformations
    \[
    \begin{tikzcd}
    \Dscr \ar[rr, ""{name = U}]{}{(G,\beta)} \ar[d, swap]{}{(H, \gamma)} & & \PC(F) \ar[d]{}{\operatorname{pr}_X} \\
    \PC(F) \ar[rr, swap, ""{name = D}]{}{\operatorname{pr}_X} & & F(X)
    \ar[from = U, to = D, Rightarrow, shorten <= 4pt, shorten >= 4pt]{}{\!\quot{\rho}{X}}
    \end{tikzcd}
    \]
    for all $X \in \Cscr_0$ subject to the coherences that for all $f\colon X \to Y$ in $\Cscr$,
    \[
   \! \quot{\rho}{X} \circ \left(\!\quot{\alpha}{f} \ast G\right) = \left(\!\quot{\alpha}{f} \ast H\right) \circ \left(F(f) \ast \!\quot{\rho}{Y}\right).
    \]
    Let $A \in \Cscr_0$. To define $P\colon G \Rightarrow H$ we set the $A$-component morphism $P_A\colon GA \to HA$ to be given by
    \[
    P_A := \left\lbrace \left(\!\quot{\rho}{X}\right)_A\colon \left(\operatorname{pr}_X \circ G\right)(A) \to \left(\operatorname{pr}_X \circ H\right)(A) \; : \; X \in \Cscr_0 \right\rbrace\colon GA \longrightarrow HA.
    \]
    To see that this is a morphism in $\PC(F)$, recall that $\tau_f^{GA} = (\!\quot{\alpha}{f})_{GA}$ and that $\tau_f^{HA} = (\!\quot{\alpha}{f})_{HA}$. We then find that the $A$-component of the required coherence
    \[
    \left(\!\quot{\rho}{X}\right)_A \circ \left(\!\quot{\alpha}{f} \ast G\right)_A = \left(\!\quot{\alpha}{f} \ast H\right)_A \circ \left(F(f) \ast \!\quot{\rho}{Y}\right)_A
    \]
    expands to the form
    \[
    \left(\!\quot{\rho}{X}\right)_A \circ \tau_f^{GA} = \left(\!\quot{\rho}{X}\right)_A \circ \left(\!\quot{\alpha}{f}\right)_{GA} = \left(\!\quot{\alpha}{f}\right)_{HA} \circ F(f)\left(\left(\!\quot{\rho}{Y}\right)_A\right) = \tau_f^{HA} \circ F(f)\left(\left(\!\quot{\rho}{Y}\right)_A\right)
    \]
    which shows that the diagram
    \[
    \begin{tikzcd}
    F(f)\left(\left(\!\quot{G}{Y}\right)A\right) \ar[rr]{}{F(f)\left(\left(\!\quot{\rho}{Y}\right)_A\right)} \ar[d, swap]{}{\tau_f^{GA}} & & F(f)\left(\left(\!\quot{H}{Y}\right)A\right) \ar[d]{}{\tau_f^{HA}} \\
    \left(\!\quot{G}{X}\right)(A) \ar[rr, swap]{}{\left(\!\quot{\rho}{X}\right)(A)} & & \left(\!\quot{H}{X}\right)(A)
    \end{tikzcd}
    \]
    commutes. Thus $P_A\colon GA \to HA$ is a morphism in $\PC(F)$.
    
    To see that the assignment $A \mapsto P_A$ is natural in $\Dscr$, note that if $f\colon A \to B$ is a morphism in $\Dscr$, the naturality of each component natural transformation $\!\quot{\rho}{X}$ implies that the diagram
    \[
    \begin{tikzcd}
    \left(\!\quot{G}{X}\right)(A) \ar[r]{}{\left(\!\quot{\rho}{X}\right)_A} \ar[d, swap]{}{\left(\!\quot{G}{X}\right)(f)} & \left(\!\quot{H}{X}\right)(A) \ar[d]{}{\left(\!\quot{H}{X}\right)(f)} \\
    \left(\!\quot{G}{X}\right)(B) \ar[r, swap]{}{\left(\!\quot{\rho}{X}\right)_B} & \left(\!\quot{H}{X}\right)(B)
    \end{tikzcd}
    \]
    commutes. Because this holds for all $f \in \Dscr_1$, we have that $P$ is a natural transformation with component morphisms $P_A$. That this is a tangent transformation follows from the fact that each $\!\quot{\rho}{X}$ is a tangent transformation by an analogue of Proposition \ref{Prop: tangent morphisms between tangent index functors} and the fact that being a tangent transformation valued in $\PC(F)$ can be checked object-locally. Moreover, a quick calculation for each object $X \in \Cscr_0$ shows that for any object $A$ of $\Dscr$
    \[
    \operatorname{pr}_X \ast P = \operatorname{pr}_X(P) = \!\quot{\rho}{X}.
    \]
    That $P$ is unique with this property is immediate to verify, as if there were any other tangent transformation $\Sigma\colon G \Rightarrow H$ with $\!\quot{\rho}{X} = \operatorname{pr}_X \ast \Sigma$ then we immediately have that for all objects $A$ in $\Dscr$ and for all objects $X$ in $\Cscr$,
    \[
    \left(\!\quot{P}{X}\right)_A = \left(\!\quot{\rho}{X}\right)_A = \left(\operatorname{pr}_X \ast P\right)_A = \left(\operatorname{pr}_X \ast \Sigma\right)_A = \left(\!\quot{\Sigma}{X}\right)_A.
    \]
    Since each component of $\Sigma$ agrees with each component of $P$, it follows that $P = \Sigma$ and so $P$ is unique.

Finally, observe that if tangent morphisms $\quot{G}{X}$ are strong tangent morphisms, then so too is $G$ by the $2$-out-of-property for isomorphisms. Thus the comparison morphims $G$ is a strong tangent morphism when the diagram factors though $\fTan_{\operatorname{strong}},$ which shows $\PC(F)$ is a pseudolimit in $\fTan_{\operatorname{strong}}$ as well.
\end{proof}

\begin{cor}\label{Cor: forgetful functor creates 1-categorical indexed pseudolimits}
The forgetful functor $\Forget\colon \fTan_{\operatorname{strong}} \to \fCat$ reflects and preserves pseudolimits indexed by $1$-categories.
\end{cor}
\begin{proof}
An alternative way of phrasing this corollary is that for any tangent indexing functor $F\colon \Cscr^{\op} \to \fCat$ with corresponding pseudofunctor $\tilde{F}\colon \Cscr^{\op} \to \fTan_{\operatorname{strong}}$ induced by swapping the roles of $\!\quot{T}{f}$ with $\!\quot{T}{f}^{-1}$ (cf. Theorem \ref{Thm: Tangent objects in Bicat Hom category are tangent indexing functors}), $\PC(F)$ is the  pseudolimit of $F$ in $\fCat$ if and only if $\PC(F)$ is also the pseudolimit of $\PC(\tilde{F})$ in $\fTan$. However, both these facts are immediate by Remark \ref{Rmk: limiting_pseudocone} and the combination of Theorems \ref{Thm: PC of tangent objects in hom two category takes values in tangent categories} and \ref{Thm: Pseudolimits in Tancat}.
\end{proof}

\section{Equivariant Descent and Equivariant Categories on Smooth Manifolds and on Varieties}\label{Section: Review of EQCats}

We now change gears and discuss equivariant algebraic geometry and equivariant differential geometry so that we can apply the theory developed in the above sections and define the categories of equivariant descent schemes and equivariant descent manifolds. We also provide this section as a way of illustrating that the perspective on equivariant descent for both varieties and for manifolds is fibre-wise in nature. Note that because perpsectives are primarily descent-theoretic in nature, we provide a short topological descent discussion before focusing on manifolds.

To describe equivariant descent (tangent) categories on varieties and smooth manifolds, we begin by recalling some of the formalism developed in \cite{GeneralGeoffThesis} for this purpose. In essence, if $G$ is a smooth algebraic group over a field $K$ such an equivariant category on a left $G$-variety $X$ (or, in the topological formulation, a left $G$-space $X$ for a Lie group $G$ and a smooth real manifold $X$) has objects given by collections of equivariant descent data as it varies through a pseudofunctor defined on a class of principal $G$-fibrations. These fibrations serve to resolve the action $\alpha_X\colon G \times X \to X$ . The flexibility of allowing the equivariant category to be determined by a pseudofunctor is that it allows us to give flexible categorical structures with which to probe and study geometric and equivariant-geometric properties; cf. Example \ref{Example: Pre-equivariant pseudofunctors} below for a large list of some of the categories which appear in algebraic geometry, algebraic topology, differential geometry, and representation theory. 
\begin{assumption}
If $K$ is a field, all $K$-varieties $X$ are assumed to be \emph{quasi-projective}. That is\footnote{Here we use some scheme-theoretic facts to simplify the story of quasi-projectivity. Because we are working over a field, i.e., over $\Spec K$ we can use the definition of quasi-projectivity defined in \cite[Page 103]{Hartshorne}. This differs from Grothendieck's definition in \cite[Definition 5.3.1]{EGA2}, but the fact that $\Spec K$ is a Noetherian quasi-compact scheme gives that the definition presented here is equivalent to Grothendieck's definition. For explicit details, see \cite[Proposition 5.3.2]{EGA2}.}, there is a projective scheme $Y$ with closed immersion $i\colon Y \to \Pbb_K^n$ and an open immersion $j\colon X \to Y$ for which the structure map $X \to \Spec K$ factors as
\[
\begin{tikzcd}
X \ar[d] \ar[r]{}{j} & Y \ar[d]{}{i} \\
\Spec K & \Pbb_K^n \ar[l]
\end{tikzcd}
\]
\end{assumption}

We set the stage for the notion of equivariant category by defining the objects which we use to resolve the actions first in the variety-theoretic case and second in the topological case. Let $K$ be a field and let $G$ be an affine\footnote{In what follows, much of the theory we develop relies on the group $G$ being an affine algebraic group; i.e., an affine group object in the category of $K$-varieties. If we instead work with non-affine groups, the characterization of smooth free $G$-varieties as $G$-varieties with {\'e}tale locally trivializable actions $G \times \Gamma \to \Gamma$ fails. Instead we have to work with what are called locally isotrivial fibrations in order to have a theory which works properly; cf. \cite[Proposition 6.1.4, Page 141]{GeneralGeoffThesis}. Because of the extra detail being beyond the scope and nature of this paper, we have elected to work only with affine $G$. Interested readers can see the discussion around \cite[Proposition 6.1.4, Page 141]{GeneralGeoffThesis} for what goes into fixing this issue as well as the additional headaches that arise.} smooth algebraic group over $K$. The category of smooth free $G$-varieties, $\Sf(G)$, is defined as follows:
\begin{dfn}[\cite{LusztigCuspidal2}]\label{Defn: Free G Var}
We say that $\Gamma$ is a {\em smooth free $G$-variety} if $\Gamma$ is a left $G$-variety (over $K$) with a geometric quotient\footnote{In the category $\Var$, a quotient of $G$ on $X$ is simply a coequalizer of the action and projection maps $\alpha_X,\pi_2\colon G \times X \to X$ in the category of varieties. A {\em geometric} quotient of $G$ on $X$ is then a quotient with the extra property that the quotient topology of the $G$-action on $X$ lines up with the Zariski topology on $G \backslash X$ and that if $q\colon X \to G \backslash X$ is the quotient map then there is an isomorphism of sheaves $\Ocal_{G \backslash X} \cong q_{\ast}\Ocal_X^G$ of the structure sheaf on $G \backslash X$ with the pushforward of $G$-invariants on $\Ocal_X$.} (cf. \cite[Definition 0.6]{GIT}) 
$\quo_\Gamma\colon \Gamma \to G\backslash \Gamma$ for which $\quo_{\Gamma}$ is a smooth $G$-principal fibration; i.e., $\quo_{\Gamma}$ is {\'e}tale locally trivializable with fibre $G$.
\end{dfn}
\begin{dfn}
The category $\Sf(G)$ is defined as follows:
\begin{itemize}
	\item Objects: Smooth free $G$-varieties $\Gamma$ of pure dimension.
	\item Morphisms: Smooth $G$-equivariant maps $f\colon \Gamma \to \Gamma^{\prime}$ of constant fibre dimension.
	\item Composition and Identities: As in $\Var_{/K}$. 
\end{itemize}
\end{dfn}

We now describe some of the important technical properties $\Sf(G)$ has which justify it as a reasonable category with which to do equivariant descent for varieties. That it is sufficient is at least partially justified in practice by the various equivalences of categories described in \cite[Proposition 7.1.24, Theorem 7.1.26, Theorem 7.2.14, Theorem 9.0.38]{vooys2021equivariant}.
\begin{prop}[{\cite{GeneralGeoffThesis}, Proposition 6.1.4, Page 141}]\label{Prop: Smooth free quotients}
Let $\Gamma \in \Sf(G)_0$ and let $X$ be a (quasi-projective) left $G$-variety. Then the variety $\Gamma \times X$ is free with a smooth quotient morphism
\[
\quo_{\Gamma \times X}\colon \Gamma \times X \to G\backslash(\Gamma \times X).
\]
Furthermore, if $f\colon \Gamma \to \Gamma^{\prime}$ is a morphism in $\Sf(G)$ then there exists a unique smooth morphism $\overline{f}$ for which the diagram
\[
\xymatrix{
	\Gamma \times X \ar[rr]^-{f \times \id_X} \ar[d]_{\quo_{\Gamma \times X}} & & \Gamma^{\prime} \times X \ar[d]^{\quo_{\Gamma^{\prime} \times X}} \\
G \backslash(\Gamma \times X) \ar@{-->}[rr]_-{\exists! \overline{f}} & & G \backslash (\Gamma^{\prime} \times X)
}
\]
commutes in $\Var_{/K}$.
\end{prop}
\begin{rmk}
The above proposition is false in general if $X$ is not a quasi-projective variety or if $G$ is a non-affine algebraic group. There are ways around the non-affine-ness of $G$, but we cannot drop the assumption that our $G$-varieties are quasi-projective.
\end{rmk}

\begin{rmk}
It is possible to modify the definition of $\Sf(G)$ by instead taking the category $\Cscr$ of smooth free $G$-varieties which need not be of pure dimension and taking morphisms to be smooth $G$-equivariant morphisms which need not be of constant fibre dimension. If you do this you will arrive at a category which is not equivalent but is sufficient for doing equivariant homological algebra in the sense that the resulting equivariant derived category is equivalent to the equivariant derived category we define. The main difference is that computing the varieties $G \backslash \Gamma$ and $G \backslash (\Gamma \times X)$ is significantly simpler when $\Gamma$ is of constant dimension and having $f$ be of constant fibre dimension also simplifies computations of morphisms. The idea of working with the category $\Sf(G)$ explicitly is due to Lusztig (cf. \cite[Section 1.9]{LusztigCuspidal2}), although working with arbitrary pseudofunctors defined on $\Sf(G)$ (or categories isomorphic to $\Sf(G)$, at any rate)  was first seen in \cite{vooys2021equivariant}.
\end{rmk}
\begin{dfn}
For any $\Gamma \in \Sf(G)_0$ and any left $G$-variety $X$, we will write the quotient variety of $\Gamma \times X$ by $G$ as
\[
\!\quot{X}{\Gamma} := G \backslash (\Gamma \times X)
\]
in order to reduce the notational complexity of this work.
\end{dfn}
\begin{cor}
There is a functor $\quo_{(-)\times X}:\Sf(G) \to \Var_{/K}$ given on objects by $\Gamma \times X \mapsto \!\quot{X}{\Gamma}$ and on morphisms by $f \times \id_X \mapsto \overline{f}$.
\end{cor}

What we have just seen in the geometric case can be run mutatis mutandis in the topological case, save that there are many simplifying aspects of the theory. Our discussion of the resolution theory of topological groups acting on topological spaces follows the exposition of \cite{BernLun}, as their work established the modern formalism used for the equivariant derived category of a topological space\footnote{It is a pervasive error in the literature to state that the equivariant derived category of a variety was defined in \cite{BernLun}; however, this is not correct. Throughout \cite{BernLun} all groups are topological groups (and more often are topological groups which admit acyclic resolutions which are manifolds) and all spaces are topological spaces. That being said, most modern incarnations of the equivariant derived category of a variety can be traced back, in some way, to the ideas put forth in \cite{BernLun}.} and established the use of the descent-theoretic perspective to equivariant derived categories in general. For what follows we will recall the notion of free spaces for a topological group before moving on to discussing free resolutions and the analogous category and quotient functor in the topological case.
\begin{dfn}[{\cite{BernLun}, Definition 2.1.1}]
Let $L$ be a topological group. We say that a left $L$-space $F$ is \emph{free} if the quotient map $q\colon F \to L \backslash F$ is locally trivializable with fibre $L$. That is, there is an open cover $\lbrace U_i \to L \backslash F \; \left. \right| \; i \in I \rbrace$ of $F$ for which in the pullback
\[
\begin{tikzcd}
q^{-1}(U_i) \ar[r] \ar[d] & F \ar[d]{}{q} \\
U_i \ar[r] & L \backslash F
\end{tikzcd}
\]
we have a commuting diagram with an $L$-equivariant isomorphism:
\[
\begin{tikzcd}
q^{-1}(U_i) \ar[dr] \ar[r]{}{\cong} & L \times U_i \ar[d]{}{\operatorname{pr}_2} \\
 & U_i
 \end{tikzcd}
\]
\end{dfn}
\begin{rmk}
When $L$ is a Lie group and $L\backslash F$ is a manifold, what we have called free $L$-spaces here are sometimes known as principal $L$-bundles in the differential geometry literature; cf. \cite[Section 10]{KMSDiffGeo} for details.
\end{rmk}
We now define a category of free $L$-spaces, as well as the notion of what it means for a map of topological spaces to be $n$-acyclic for $n \in \N$. These notions are important for equivariant descent and the equivariant derived category on a topological space, as it is the descent through $n$-acyclic resolutions (as $n$ varies through $\N$) which control and define the equivariant derived category.
\begin{dfn}\label{Defn: The category of free G spaces}
Let $L$ be a topological group. We define the category $\Free(L)$ of free $L$-spaces by:
\begin{itemize}
    \item Objects: Free $L$-spaces $P$.
    \item Morphisms: $L$-equivariant surjective morphisms of spaces.
    \item Composition and Identities: As in $\Top$.
\end{itemize}
\end{dfn}

The definition above is well-suited to studying the equivariant descent theory of actions of topological groups on topological spaces by freely resolving the action by free spaces. However, we are primarily interested in the cases where the topological group $L$ is a Lie group (that is, a group object in $\SMan$) and when the space $M$ on which $L$ acts is a smooth manifold as well. In this case we wish to resolve the action of $L$ on $M$ by free $L$-spaces (that is, by principal $L$-bundles) in order to record how the differential structures of the actions vary. To do this, however, we necessitate knowing that the group $L$ admits free $L$-spaces which are themselves, smooth manifolds. Luckily, in most cases of interest to representation theorists, there are positive answers to this question.

\begin{rmk}
It is known from \cite[Lemma 3.1, Page 34]{BernLun} that if $L$ is either a closed subgroup of $\GL_n(\R)$ for some $n \in \N$ or if $L$ is a Lie group with finitely many connected components, then $L$ has free spaces which are smooth manifolds. In the case $L$ is a linear group, the Stiefel manifold $S_{n,k}$ of $k$-frames in $\R^{n+k}$ is a free $L$-space which is a smooth manifold for all $k$.
\end{rmk}

Because of the remark above, we make an assumption which allows us to ensure that we have free $L$-spaces which are smooth manifolds. This also will allow us to have a tractable handle on the principal bundle theory of the group $L$ which stays inside the category $\SMan$.

\begin{assumption}\label{Assume on Lie groups}
We assume that $L$ is a Lie group which is either a closed subgroup of some $\GL_n(\R)$ or a Lie group with finitely many connected components.
\end{assumption}

Under these assumptions, we have that $L$ admits free $L$-spaces which are smooth manifolds. In this case we will focus on the category of these free $L$-manifolds as opposed to all possible free $L$-spaces so as to stay within $\SMan$.
\begin{dfn}
Let $L$ be a Lie group satisfying Assumption \ref{Assume on Lie groups}. We define the category $\FMan(L)$ of free $L$-manifolds as follows:
\begin{itemize}
    \item Objects: Free $L$-spaces $M$ which are smooth manifolds and whose $L$-action on $M$ is smooth.
    \item Morphisms: $L$-equivariant surjective submersions.
    \item Composition and Identities: As in $\Free(L)$.
\end{itemize}
\end{dfn}

Because of the presence of surjective submersions in our definition of free $L$-manifolds, it will be helpful to recall that said maps are closed with respect to composition and post-compositional factors. While this is certainly well-known to experts, we prove it in this paper explicitly.
\begin{lem}\label{Lemma: Section Equivariant: Surjective Submersion Stability and permanence}
Let $f:X \to Y$ and $g:Y \to Z$ be morphisms in $\SMan$. Then:
\begin{enumerate}
    \item If $f$ and $g$ are surjective submersions, so too is $g \circ f$.
    \item If $g \circ f$ and $f$ are surjective submersions, so too is $g$.
\end{enumerate}
In particular, if $f$ is a surjective submersion then $g \circ f$ is a surjective submersion if and only if $g$ is a surjective submersion.
\end{lem}
\begin{proof}
Begin by observing that the final claim of the lemma follows by combining the statements in Items $(1)$ and $(2)$ simultaneously. Thus, it suffices to simply prove $(1)$ and $(2)$ in isolation.

$(1):$ Assume that $f\colon X \to Y$ and $g\colon Y \to  Z$ are surjective submersions. Then $g \circ f$ is surjective and for all $x \in X$,
\[
T(g \circ f)_x =D[g \circ f](x) = D[g](f(x)) \circ D[f](x)
\]
is a composite of surjective morphisms and hence also surjective. Thus $g \circ f$ is also a submersion.

$(2):$ Assume that $g \circ f$ and $f$ are surjective submersions. Because surjections are post-compositionally permanent, $g$ is necessarily a surjection as well. Additionally, consider that for all $y \in Y$, since $y = f(x)$ for some $x \in X$, we can write
\[
D[g](y) = Tg_y = Tg_{f(x)} = D[g](f(x)).
\]
Now because
\[
D[g](f(x)) \circ D[f](x) = D[g \circ f](x)
\]
is surjective it also follows that $D[g](f(x)) = D[g](y)$ is surjective for all $y \in Y$. Thus $g$ is a submersion.
\end{proof}

In the case that $M$ is a free $L$-manifold, the group action $L \times M \to M$ is free and acts on the fibres of $M$ by translation in the $L$-atlas of $M$. This allows us to put a smooth manifold structure on the quotient space $L \backslash M$ and also deduce that the morphism $\quo_{M}\colon M \to L \backslash M$ is a submersion of smooth manifolds, as the local trivialization of the quotient map $\quo_{M}\colon M \to L \backslash M$ gives the quotient map $\quo_{M}$ the structure of a differential $L\backslash M$-bundle. We collect the observations we require on free smooth $L$-manifolds in order to use later.

\begin{prop}\label{Prop: Obsrevations about free L manifolds}
Let $L$ be a Lie group satisfying Assumption \ref{Assume on Lie groups} and let $M$ be a smooth left $L$-space; i.e., a smooth manifold $M$ with a smooth group action $\alpha\colon L \times M \to M$. Then:
\begin{enumerate}
    \item If $F$ is a free $L$-manifold, then the quotient morphism $\quo_{F}\colon F \to L \backslash F$ is a submersion.
    \item For any manifold $F$ in $\FMan(L)$, the quotient space $L\backslash (F \times M)$ is a smooth manifold and the quotient morphism $\quo_{F \times M}\colon F \times M \to L \backslash (F \times M)$ is a submersion of manifolds.
    \item For any morphism $f\colon F \to E$ in $\FMan(L)$ there is a unique morphism $\overline{f}\colon L \backslash (F \times M) \to L \backslash (E \times M)$ fitting into the commutative square:
    \[
    \begin{tikzcd}
    F \times M \ar[rr]{}{f \times \id_M} \ar[d, swap]{}{\quo_{F \times M}} & & E \times M \ar[d]{}{\quo_{E \times M}} \\
    L \backslash (F \times M) \ar[rr, swap]{}{\overline{f}} & & L \backslash (E \times M)
    \end{tikzcd}
    \]
    In particular, there is a functor $\quo_{(-)\times M}\colon \FMan(L) \to \SMan$ given on objects by $F \mapsto L \backslash (F \times M)$ and on morphisms by $f \mapsto \overline{f}$.
    \item For any morphism $f\colon F \to E$ of free $L$-spaces in $\FMan(L)$, the map $\overline{f}\colon L \backslash (F \times M) \to L \backslash (E \times M)$ is a surjective submersion.
\end{enumerate}
\end{prop}
\begin{proof}
For (1) we note that the discussion of the fact that $L \backslash F$ is smooth is given directly prior to the statement of the proposition. That the map $\quo_{F}\colon F \to L\backslash F$ is a submersion follows from \cite[Lemma 10.3]{KMSDiffGeo} and the fact that the local trivialization of the quotient map makes $F$ into a differential $L \backslash F$-bundle.

The content of (2) is precisely \cite[Theorem 10.7.(1)]{KMSDiffGeo}.

To prove (3) note that the manifold $L \backslash (F \times M)$ is the coequalizer of $\alpha_{F \times M},\pi_2\colon L \times (F \times M) \to F \times M$ in $\SMan$ by virtue of (2). As such, the existence of the map $\overline{f}$ follows immediately from the universal property of the coequalizers. Finally, the functoriality of this construction in $\FMan(L)$ is immediate from (1), (2), and the existence of uniqueness of the maps $\overline{f}$.

To prove (4) note that since the diagram
   \[
    \begin{tikzcd}
    F \times M \ar[rr]{}{f \times \id_M} \ar[d, swap]{}{\quo_{F \times M}} & & E \times M \ar[d]{}{\quo_{E \times M}} \\
    L \backslash (F \times M) \ar[rr, swap]{}{\overline{f}} & & L \backslash (E \times M)
    \end{tikzcd}
    \]
    commutes with both $\quo_{E \times M}$ and $f \times \id_M$ surjective submersions, $\quo_{E \times M} \circ (f \times \id_M)$ is a surjective submersion by Part $(1)$ of Lemma \ref{Lemma: Section Equivariant: Surjective Submersion Stability and permanence}. Because
    \[
    \quo_{E \times M} \circ (f \times \id_M) = \overline{f} \circ \quo_{F \times M}
    \]
    is a surjective submersion with $\quo_{F \times M}$ a surjective submersion as well, Part $(2)$ of Lemma \ref{Lemma: Section Equivariant: Surjective Submersion Stability and permanence} gives that $\overline{f}$ is a surjective submersion, as was claimed.
\end{proof}

We now recall the definition of a pre-equivariant pseudofunctor and then the equivariant category on a variety.
\begin{dfn}\label{Defn: Pre-Equivariant Pseudofunctor}
Let $G$ be a smooth algebraic group over a field $K$ with $X$ a left $G$-variety and let $L$ be a Lie group satisfying Assumption \ref{Assume on Lie groups} with $M$ a smooth left $L$-manifold with smooth $L$-action. A {\em pre-equivariant pseudofunctor} on $X$ is a pair $(F,\overline{F})$ of pseudofunctors $F\colon \Cscr^{\op} \to \fCat$ and $\overline{F}\colon \Var_{/K}^{\op} \to \fCat$ making the diagram
\[
\xymatrix{
\Sf(G)^{\op} \ar[rr]^-{\quo_{(-) \times X}^{\op}} \ar[drr]_{F} & & \Var_{/K}^{\op} \ar[d]^{\overline{F}} \\
 & & \fCat
}
\]
commute. Similarly, a {\em pre-equivariant pseudofunctor} on $M$ is a pair $(E, \overline{E})$ of pseudofunctors $E\colon \FMan(L)^{\op} \to \fCat$ and $\overline{E}\colon \SMan^{\op} \to \fCat$ making the diagram
\[
\xymatrix{
\FMan(L)^{\op} \ar[rr]^-{\quo_{(-) \times M}^{\op}} \ar[drr]_-{E} & & \SMan^{\op} \ar[d]^{\overline{E}} \\
 & & \fCat
}
\]
commute.
\end{dfn}
\begin{rmk}
The term ``pre-equivariant'' in the definition above seems strange at first glance. However, we call the pair $(F,\overline{F})$ a pre-equivariant pseudofunctor for the following (arguably aesthetic) reasons and as a way of organizing one's thoughts.
 
In practice, an equivariant category (such as a category of equivariant sheaves, the equivariant derived category, a category of equivariant perverse sheaves, equivariant local systems, equivariant quasi-coherent sheaves, equivariant $\Dcal$-modules, etc.) is determined by three aspects: A pseudofunctor defined on a ``sufficiently global'' category (such as the category of varieties, smooth varieties, smooth manifolds, etc.), all possible equivariant descent data recorded by the quotient functor $\quo_{(-) \times X}\colon \Sf(G) \to \Var_{/K}$ or $\quo_{(-) \times M}\colon \FMan(L) \to \SMan$, and the fact that in forming the equivariant category itself we take what amounts to all possible effective descent data mediating between the group action resolutions and the pseudofunctor itself. In this sense, by giving the pseudofunctors $F$ and $\overline{F}$, we furnish ourselves with the first two pieces of equivariant descent data. It's only after taking global sections (as in the Cartesian fibrations of the elements of $F$) that we arrive with an actual equivariant category, so a pre-equivariant pseudofunctor is ``pre-equivariant'' in the sense that it is the information with which we can begin to consider equivariant descent, but have not yet imposed any descent conditions and as such have not yet made the information completely equivariant.
\end{rmk}

\begin{dfn}
Let $F = (F,\overline{F})\colon \Sf(G)^{\op} \to \fCat$ be a pre-equivariant pseudofunctor on $X$. The \emph{equivariant category} of $F$ over $X$ is defined to be the category $F_G(X) := \PC(F)$.
\end{dfn}

\begin{example}\label{Example: Pre-equivariant pseudofunctors}
Here are some examples of equivariant categories which appear in equivariant algebraic geometry, equivariant differential geometry, and equivariant algebraic topology. As before, $G$ here is a smooth algebraic group over a field $K$, $X$ is a left $G$-variety, $L$ is a Lie group satisfying Assumption \ref{Assume on Lie groups}, and $M$ is a smooth left $L$-space with $\SMan$ group action $L \times M \to M$. Note that we are illustrating the uses of the equivariant category formalism and that these categories need not necessarily carry interesting tangent structures.
\begin{enumerate}
	\item  If $\Dbb,\Dbb^b$ are the pseudofunctors
 \[
 \xymatrix{
\Sf(G)^{\op} \ar[rr]^-{\quo_{(-) \times X}^{\op}} \ar[drr]_-{\Dbb} & & \Var_{/K}^{\op} \ar[d]^{D_c(-)} \\
  & & \fCat
 }\qquad
  \xymatrix{
\Sf(G)^{\op} \ar[rr]^-{\quo_{(-) \times X}^{\op}} \ar[drr]_-{\Dbb^b} & & \Var_{/K}^{\op} \ar[d]^{D_c^b(-)} \\
  & & \fCat
 }
 \]
 with
	\[
	\Dbb(\Gamma) := D_c(\XGamma), \quad \Dbb^b(\Gamma) = D_c^b(\XGamma)
	\]
	on objects $\Gamma$ and 
	\[
	\Dbb(f) = L\of^{\ast}, \quad \Dbb^b(f \times \id_X) = L\of^{\ast},
	\]
	then $\PC(\Dbb^{b}) = \Dbb^b_G(X) = D_G^b(X)$ and $\PC(\Dbb) = \Dbb_G(X) = D_G(X)$ are the bounded and unbounded equivariant derived categories, respectively. Note that the category $D_c^b(\XGamma)$ is the derived category of bounded constructible {\'e}tale sheaves on $\XGamma$; cf. \cite[Definition II.4.7]{FreitagKiehl} for details regarding constructible sheaves.
 \item The differential-geometric/algebraic-topological equivariant derived category on $M$ is described as follows. Define the pre-equivariant pseudofunctors $\Dbb, \Dbb^b$ by
 \[
 \xymatrix{
\FMan(L)^{\op} \ar[rr]^-{\quo_{(-) \times X}^{\op}} \ar[drr]_-{\Dbb} & & \SMan^{\op} \ar[d]^{D_c(-)} \\
  & & \fCat
 }\qquad
 \xymatrix{
\FMan(L)^{\op} \ar[rr]^-{\quo_{(-) \times L}^{\op}} \ar[drr]_-{\Dbb^b} & & \SMan^{\op} \ar[d]^{D^b_c(-)} \\
  & & \fCat
 }
 \]
 where $\Dbb(F) := D_c(L \backslash (F \times M))$, $\Dbb^b(F) := D_c^b(L \backslash (F \times M))$, and $\Dbb(f) = \overline{f}^{\ast}$ while $\Dbb^b(f) = \overline{f}^{\ast}$. Then the equivariant categories $D_L(M) =: \Dbb_L(M)$ and $D_L^b(M) =: \Dbb^b_L(M)$ are the topological equivariant derived and bounded derived categories on $M$, respectively.
	\item If $\underline{\DbQl{-}}$ is the pre-equivariant pseudofunctor
 \[
 \xymatrix{
\Sf(G)^{\op} \ar[rr]^-{\quo_{(-) \times X}^{\op}} \ar[drr]_-{\underline{\DbQl{-}}} & & \Var_{/K}^{\op} \ar[d]^{D_c^b(-;\overline{\Q}_{\ell})} \\
  & & \fCat
 }
 \]
 with
	\[
	\underline{\DbQl{\Gamma}} = \DbQl{\XGamma}
	\]
	on objects and
	\[
	\underline{\DbQl{f}} = \of^{\ast}
	\]
	then $\underline{\DbQl{X}}_G = \DbeqQl{X}$ is the equivariant derived category of $\ell$-adic sheaves on $X$ (for $\ell > 0$ an integer prime with $\gcd(\operatorname{char}(X),\ell) = 1$ if $\operatorname{char}(X) > 0$).
	\item If $\Sbb$ is the pre-equivariant pseudofunctor
 \[
\xymatrix{
\Sf(G)^{\op} \ar[rr]^-{\quo_{(-) \times X}^{\op}} \ar[drr]_-{\Sbb} & & \Var_{/K}^{\op} \ar[d]^{\Ab(-;\overline{\Q}_{\ell})} \\
  & & \fCat
 }
 \]
 where
	\[
	\Sbb(\Gamma) := \Ab(\XGamma;\overline{\Q}_{\ell})
	\]
	on objects and
	\[
	\Sbb(f) := \of^{\ast}
	\]
	on morphisms then $\Sbb_G(X) =: \Shv_G(X;\overline{\Q}_{\ell})$ is the category of equivariant $\ell$-adic sheaves on $X$ for $\ell > 0$ an integer prime with $\gcd(\operatorname{char}(X),\ell) = 1$ if $\operatorname{char}(X) > 0$. Note that $\Shv_G(X;\overline{\Q}_{\ell})$ is equivalent to the usual category of equivariant $\ell$-adic sheaves by \cite[Proposition 7.1.26]{vooys2021equivariant}.
	\item If $\Pbb$ is the pre-equivariant pseudofunctor
 \[
 \xymatrix{
\Sf(G)^{\op} \ar[rr]^-{\quo_{(-) \times X}^{\op}} \ar[drr]_-{\Pbb} & & \Var_{/K}^{\op} \ar[d]^{\Per(-;\overline{\Q}_{\ell})} \\
  & & \fCat
 }
 \]
 with
	\[
	\Pbb(\Gamma) := \Per(\XGamma;\overline{\Q}_{\ell})
	\]
	and
	\[
	\Pbb(f) := {}^{p}\of^{\ast}
	\]
	then $\Pbb_G(X) =: \Per_G(X;\overline{\Q}_{\ell})$ is the category of equivariant perverse $\ell$-adic sheaves on $X$. Note that $\Per_G(X;\overline{\Q}_{\ell})$ is equivalent to the usual category of equivariant $\ell$-adic perverse sheaves by \cite[Theorem 7.1.28]{vooys2021equivariant}.
	\item Let $\Lbb$ be the pre-equivariant pseudofunctor given by
 \[
\xymatrix{
\Sf(G)^{\op} \ar[rr]^-{\quo_{(-) \times X}^{\op}} \ar[drr]_-{\Lbb} & & \Var_{/K}^{\op} \ar[d]^{\Shv((-)_{\operatorname{\acute{e}t}})_{\operatorname{lcf}}} \\
  & & \fCat
 }
 \]
 where
	\[
	\Lbb(\Gamma) := \Shv(\XGamma_{\text{{\'e}t}})_{\text{lcf}}
	\]
	on objects and
	\[
	\Lbb(f) = \of^{\ast}
	\]
	on morphisms for $\Shv(\XGamma_{\text{\'e}t})_{\text{lcf}}$ the category of locally constant finite {\'e}tale sheaves on $\XGamma$. Then $\Lbb_G(X) =: \Shv_G(X_{\text{{\'e}t}})_{\text{lcf}}$ is the category of equivariant locally constant finite {\'e}tale sheaves on $X$.
	\item Let $\Q(-):\SfResl_G(X)^{\op} \to \fCat$ be the pre-equivariant pseudofunctor
 \[
 \xymatrix{
\Sf(G)^{\op} \ar[rr]^-{\quo_{(-) \times X}^{\op}} \ar[drr]_-{\Q} & & \Var_{/K}^{\op} \ar[d]^{\QCoh(-)} \\
  & & \fCat
 }
 \]
	\[
	\Q(\Gamma):= \QCoh(\XGamma)
	\]
	on objects and
	\[
	\Q(f) := \of^{\ast}.
	\]
	Then $\Q_G(X) =: \QCoh_G(X)$ is the category of equivariant quasi-coherent sheaves on $X$.
	\item Recall that if $K$ is a complete valued field, a $K$-analytic space is a locally ringed space $X = (\lvert X \rvert, \Ocal_X)$ for which there is an open cover $\Ufrak = \lbrace U_i \; \left. \right| \; i \in I \rbrace$ of $X$ by locally ringed spaces $U_i = (\lvert U_i \rvert, \Ocal_{U_i})$ (so in particular there is a gluing
	\[
	X \cong \bigcup_{i \in I} U_i
	\]
	in the category $\LocRingSpac$ of locally ringed spaces) where for each $i \in I$ there is an isomorphism of locally ringed spaces $U_i \cong Z_i$ where $Z_i$ is an analytic $K$-variety. More explicitly, there is an open subset $V_i \subseteq K^{n_i}$ for some $n_i \in \Nbb$ and a collection of analytic functions $f_1, \cdots, f_{m_i}$ on $V_i$ for which $\lvert Z_i\rvert = \lbrace x \in V_i \; \left. \right| \; f_1(x) = \cdots = f_{m_i}(x) = 0 \rbrace$ equipped with the subspace topology and $\Ocal_{Z_i} \cong \Ocal_{V_i}/\Iscr(f_1, \cdots, f_{m_i})$ where $\Ocal_{V_i}$ is the sheaf of analytic functions on $V_i$ and $\Iscr(f_1, \cdots, f_{m_i})$ is the ideal sheaf of $\Ocal_{Z_i}$ generated by the $f_j$. It is well-known that to each finite-type $K$-scheme $X$, there is an associated $K$-analytic space $X_{\text{an}}$ which is functorial in $\Sch_{/K}^{\text{f.t.}}$; i.e., there is a functor
	\[
	(-)_{\text{an}}\colon \Sch^{\text{f.t.}}_{/K} \to \mathbf{AnSp}_{K}
	\]
	where $\mathbf{AnSp}_K$ is the category of $K$-analytic spaces. We now construct a pre-equivariant pseudofunctor $\operatorname{AnCoh}\colon \Sf(G)^{\op} \to \fCat$ as follows. Let $\Coh\colon \mathbf{AnSp}_{K}^{\op} \to \fCat$ be the coherent module pseudofunctor (so for each $K$-analytic space $Y$, $\Coh(Y)$ is the category of coherent sheaves over $Y$ and for each morphism $f\colon Y \to Z$ of analytic spaces $\Coh(f) = f^{\ast}\colon \Coh(Z) \to \Coh(Y)$ is the pullback functor of coherent sheaves. Now define $\operatorname{AnCoh}$ by:
	\[
	\xymatrix{
	\Sf(G)^{\op} \ar[r]^-{\quo^{\op}} \ar[dr]_{\operatorname{AnCoh}} & \Var_{/K}^{\op} \ar[d]^{\Coh \circ (-)_{\operatorname{an}}} \\ 
	 & \fCat
	 }
	\]
	Then $\operatorname{AnCoh}_G(X)$ describes the category of analytic coherent sheaves on $X$ (when regarded as an analytic space).
   \item Define the pre-equivariant pseudofunctor $\Ccal$ by
   \[
   \xymatrix{
\FMan(L)^{\op} \ar[rr]^-{\quo_{(-) \times M}^{\op}} \ar[drr]_-{\Ccal} & & \SMan^{\op} \ar[d]^{C^{\infty}\Shv(-)} \\
  & & \fCat
 }
   \]
   where
   \[
    \Ccal(F) := C^{\infty}\Shv(L \backslash (F \times M))
   \]
   is the category of sheaves of $C^{\infty}$-modules on $L \backslash (F \times M)$ and where
   \[
   \Ccal(f) := f^{\ast}.
   \]
   Then $\Ccal_L(M) =: C^{\infty}\Shv_L(M)$ describes the category of $L$-equivariant $C^{\infty}$-modules on $M$.
   \item Let $\Vbb$ be the pre-equivariant pseudofunctor
   \[
   \xymatrix{
    \FMan(L)^{\op} \ar[rr]^-{\quo_{(-) \times M}^{\op}} \ar[drr]_{\Vbb} & & \SMan_{\operatorname{sub}}^{\op} \ar[d]^{\VVec} \\
    & & \fCat
   }
   \]
   where $\SMan_{\operatorname{sub}}^{\op}$ is the category of smooth manifolds with submersions as morphisms\footnote{That $\quo_{(-) \times M}$ lands in $\SMan_{\operatorname{sub}}$ follows from Part (4) of Proposition \ref{Prop: Obsrevations about free L manifolds}.} and where
   \[
   \Vbb(F) := \VVec(L \backslash (F \times M))
   \]
   is the category of (smooth) vector bundles over $L \backslash (F \times M)$ and where the functor
   \[
   \Vbb(f) := \overline{f}^{\ast}
   \]
   is given by pullback. Then the category $\Vbb_L(M) =: \VVec_L(M)$ describes the category of equivariant vector bundles on $M$.
   \item Let $\Scal$ be the pre-equivariant pseudofunctor
   \[
   \xymatrix{
\Sf(G)^{\op} \ar[rr]^-{\quo_{(-) \times X}^{\op}} \ar[drr]_-{\Scal} & & \Var_{/K}^{\op} \ar[d]^{\Sch_{/(-)}} \\
  & & \fCat
 }
   \]
   where 
   \[
   \Scal(\Gamma) := \Sch_{/\XGamma}
   \]
   and where
   \[
   \Scal(f) := \of^{\ast}.
   \]
   Then $\Scal_G(X)$ defines and describes the category of equivariant descent data in schemes over $X$. 
\item Similarly to the last example, if $\Mcal$ is the pre-equivariant pseudofunctor
   \[
   \xymatrix{
\FMan(L)^{\op} \ar[rr]^-{\quo_{(-) \times M}^{\op}} \ar[drr]_-{\Mcal} & & \SMan_{\operatorname{sub}}^{\op} \ar[d]^{\SMan\downarrow{(-)}} \\
  & & \fCat
 }
   \]
   where $\SMan_{\operatorname{sub}}$ is the category of smooth manifolds with submersions as morphisms,
   \[
   \Mcal(F) := \SMan\downarrow\left({(L \backslash (F \times M))}\right)
   \]
   is the category of smooth manifolds over $L \backslash (F \times M)$, and where
   \[
   \Mcal(f) := \overline{f}^{\ast}
   \]
   is given by pullback\footnote{Note that by Part (4) of Proposition \ref{Prop: Obsrevations about free L manifolds}, the map $\overline{f}$ is a submersion and so the pullback functor $\overline{f}^{\ast}$ exists.}. Then the category $\Mcal_{L}(M) = \SMan_{L}(M)$ describes the category of descent equivariant smooth manifolds over $M$.
\end{enumerate}
\end{example}
\begin{rmk}
Each of the categories $D_G(X),$ $\Dbeq{G}{X}$, $D_L(M),$ $D_L(M)^b$, $\DbeqQl{X}$, $\Per_G(X;\overline{\Q}_{\ell})$, $\Shv_G(X;\overline{\Q}_{\ell}),$ $\QCoh_G(X),$ $\VVec_L(M),$ $C^{\infty}\Shv_L(M)$ are additive categories. The categories $D_G(X),$ $D_L^b(M)$, $D_L(M)$, $\Dbeq{G}{X},$ and $\DbeqQl{X}$ are all triangulated categories with standard and perverse $t$-structures whose hearts are equivalent to the usual categories of equivariant ($\ell$-adic) sheaves and perverse sheaves; cf. \cite[Example 5.1.23, Corollaries 5.1.27, 5.1.28]{GeneralGeoffThesis}. The categories $\Per_G(X;\overline{\Q}_{\ell}), \Shv_G(X;\overline{\Q}_{\ell})$, and $\QCoh_G(X)$ are all Abelian categories; cf. \cite[Corollary 3.1.13]{GeneralGeoffThesis}. In general, to show an equivariant category $F_G(X)$ is enriched in models of a Lawvere theory, Abelian, additive, triangulated, etc., it suffices to show that each category $F(\Gamma)$ has the desired structure and that the fibre functors $F(f)$ all preserve this structure up to isomorphism; cf., for instance, \cite[Lemma 2.4.5, Proposition 3.1.12, Theorem 3.2.3, Proposition 3.2.7, Proposition 3.3.6, Theorem 5.1.10, Theorem 5.1.21]{GeneralGeoffThesis} --- these various results show that the categories $F_G(X)$ satisfy the flavour of structural theory discussed.
\end{rmk}

\section{The Descent Equivariant Zariski Tangent Structure}\label{Section: Example of equivariant tangent structure on a variety}
In a recent paper, \cite{GeoffJS} have proved that the category $\Sch_{/S}$ admits a tangent structure for any base scheme $S$. The tangent functor $T\colon \Sch_{/S} \to \Sch_{/S}$ sends an $S$-scheme $X$ to the (Zariski) tangent fibre of $X$ relative to $S$ constructed by Grothendieck in \cite[Section 16.5]{EGA44}. In particular, for an $S$-scheme $X$ we have\footnote{The version of the tangent functor and maps constructed in \cite{GeoffJS} is given for affine schemes, but because the $\Sym$ functor commutes with tensors on $\QCoh(X)$ and because the sheaf of differentials $\Omega_{X/S}^1$ is a quasi-coherent sheaf, everything regarding this functor may be checked Zariski-locally, i.e., affine-locally on both the target and the base. We'll describe this more in detail later, but it is worth noting now.}
\[
T(X) := T_{X/S} = \Spec\left(\Sym(\Omega^1_{X/S})\right).
\]
For an affine scheme $S = \Spec A$ and an affine $S$-scheme $X = \Spec B$, the scheme $T_{X/S} = T_{B/A} = \Spec(\Sym(\Omega^1_{B/A}))$ is an affine scheme. The ring
\[
C = \Sym_B\left(\Omega_{B/A}^{1}\right)
\]
is generated by symbols $b, \mathrm{d}b$ for $b \in B$ generated by the rules that addition and multiplication for symbols from $b \in B$ are as in $B$ and the Leibniz rule
\[
\mathrm{d}(bb^{\prime}) = b^{\prime}\,\mathrm{d}(b) + b\,\mathrm{d}(b^{\prime})
\]
holds with $(\mathrm{d}b)\,(\mathrm{d}b^{\prime}) = 0$ and $\mathrm{d}(a) = 0$ for $a \in A$. With this definition we find that 
\[
T_2(X) = (T_{B/A})_{2} = \Spec\left(\Sym_B(\Omega_{B/A}^1) \otimes_B \Sym_B(\Omega_{B/A}^{1})\right)
\]
and that
\[
T^2(X) = T_{T_{B/A}/A} = \Sym_{\Sym_B(\Omega^1_{B/A})}\left(\Omega^{1}_{\Sym_B(\Omega_{B/A}^1)/B}\right).
\]
It is worth noting for what follows that the algebra $C$ with $T^2_{B/A} = \Spec C$,
\[
C = \Sym_{\Sym_B(\Omega_{B/A}^1)}\left(\Omega^1_{\Sym_B(\Omega_{B/A}^1)/B}\right),
\] 
is generated by symbols $b, \mathrm{d}b, \delta b,$ and $(\delta\mathrm{d})(b)$ for all $b \in B$. Essentially, there is a new derivational neighborhood $\delta$ of $\Sym(\Omega_{B/A}^1)$ which gives us a notion of $2$-jets and a distinct ``orthogonal'' direction of $1$-jets (the $\delta$-direction --- the idea is that $T^2$ roughly corresponds to the mixed partials $\partial^2/\partial x\partial y$ and $\partial^2/\partial y \partial x$ when both partials are regarded as nil-square operators).

\begin{Theorem}[\cite{GeoffJS}]\label{Thm: Zariski tangent structure}
For affine schemes $\Spec B \to \Spec A$ the tangent structure $(\Sch_{/\Spec A}, \mathbb{T})$ is generated by the maps and functors on affine schemes:
\begin{itemize}
	\item The tangent functor is given by $T_{\Spec B/\Spec A} = T_{B/A} = \Spec\left(\Sym_B(\Omega_{B/A}^1)\right)$.
	\item The bundle map $p_B\colon T_{B/A} \to \Spec B$ is the spectrum of the ring map
	\[
	q_B\colon B \to \Sym_B(\Omega_{B/A}^1)
	\]
	generated by $b \mapsto b$.
	\item The zero map $0_B\colon T_{B/A} \to \Spec B$ is the spectrum of the ring map
	\[
	\zeta_B\colon \Sym_B(\Omega_{B/A}^{1}) \to B
	\]
	given by $b \mapsto b, \mathrm{d}b \mapsto 0$. 
	\item The bundle addition map $+_B\colon (T_{B/A})_{2} \to T_{B/A}$ is the spectrum of the map
	\[
	\operatorname{add}_B\colon \Sym(\Omega^{1}_{B/A}) \to \Sym_B(\Omega^{1}_{B/A}) \otimes_B \Sym_B(\Omega^{1}_{B/A})
	\]
	given by $b \mapsto b \otimes 1_B, \mathrm{d}b \mapsto \mathrm{d}b \otimes 1 + 1 \otimes \mathrm{d}b$.
	\item The vertical lift $\ell_B\colon  T_{B/A} \to T^2_{B/A}$ is given as the spectrum of the ring map
	\[
	v_B\colon \Sym_{\Sym_B(\Omega^1_{B/A})}\left(\Omega^1_{\Sym_B(\Omega_{B/A}^1)/B}\right) \to \Sym_B\left(\Omega^1_{B/A}\right)
	\]
	generated by $b \mapsto b, \mathrm{d}b \mapsto 0, \delta b \mapsto 0$, $\delta\mathrm{d}(b) \mapsto \mathrm{d}b$.
	\item The canoncial flip is the map $c_B\colon T^2_{B/A} \to T^2_{B/A}$ generated as the spectrum of the ring map
	\[
	\gamma_B\colon \Sym_{\Sym_B(\Omega^1_{B/A})}\left(\Omega^1_{\Sym_B(\Omega_{B/A}^1)/B}\right) \to \Sym_{\Sym_B(\Omega^1_{B/A})}\left(\Omega^1_{\Sym_B(\Omega_{B/A}^1)/B}\right)
	\]
	which interchanges $1$-jets, i.e., the map is generated by $b \mapsto b, \mathrm{d}b \mapsto \delta b, \delta b \mapsto \mathrm{d}b$ and $(\delta\mathrm{d})b \mapsto (\delta\mathrm{d})b$.
\end{itemize}
\end{Theorem}
\begin{rmk}
In order to reduce notational complexity, we will omit the subscripts on the $\Sym$ functors if no confusion is likely to arise from said omission.
\end{rmk}

\begin{dfn}\label{Defn: Zariski tangent structure on Sch over X}
We define the Zariski tangent structure on a scheme $S$ to be the tangent structure $\Tbb_{\Zar{S}} = (T_{-/S}, p, 0, +, \ell, c)$ on the category $\Sch_{/S}$ described by Theorem \ref{Thm: Zariski tangent structure} when $S$ is affine and the one induced by quasi-coherent gluing of the affine tangent structures when $S$ is nonaffine.
\end{dfn}

This is, in some sense, the ``canonical'' tangent structure on $\Sch_{/S}$, as the tangent scheme $T_{X/S}$ captures $S$-derivations of $X$ in the following sense: If $S = \Spec K$ for a field $K$ and if $\Spec A$ is an affine $K$-scheme, then the $A$-points of the tangent scheme $T_{X/S}(A)$ satisfy
\[
T_{X/S}(A) = \Sch_{/K}(\Spec A,T_{X/S}) \cong \Sch_{/K}\left(\Spec\left(\frac{A[x]}{(x^2)}\right), X\right)
\]
so in particular $K$-points of $T_{X/S}$ give the $1$-differentials $\Omega_{X/K}^1$. Moreover, for any closed point $x \in \lvert X \rvert$ we have a canonical isomorphism
\[
T_{X/S}(x) \cong \KAlg\left(\frac{\mfrak_{x}}{\mfrak^{2}_{x}},K\right)
\]
of $T_{X/S}(x)$ with the Zariski tangent space of $X$ over $K$.
\begin{rmk}
The tangent functor $T_{(-)/S}:\Sch_{/S} \to \Sch_{/S}$ is a representable tangent functor. If $W_S$ is the object
\[
W_S := S \times_{\Spec \Z} \Spec\frac{\Z[x]}{(x^2)} = S \times_{\Spec \Z} \Z[\epsilon]
\]
then there is an adjunction $(-) \times_S W_S \dashv T_{(-)/S}$\footnote{When $S = \Spec A$ is an affine scheme, $W_S \cong \Spec A[\epsilon] \cong \Spec A[x]/(x^2)$. By the algebraic geometry pullback yoga, we can compute the pullback $X \times_S S[\epsilon]$ affine-locally by the pullbacks $X_j \times_{A_i} \Spec A_i[\epsilon]$ for $A_i$ an affine open of $S$ and $X_j$ an affine open of $X$ and then glue them along the corresponding open immersions (this is precisely the content, and method of proof, of \cite[Theorem II.3.3]{Hartshorne}). Each of the functors $X_j \times_{A_i} \Spec A_i[\epsilon]$ admits the right adjoints $T_{X_j/A_i}$ which in turn may be glued along the adjoint transposes to not only give $T_{X/S}$ but also construct the unit and counit morphisms.}. When $S = \Spec R$ is an affine scheme and we have $X = \Spec A$ and $Y = \Spec B$ affine $S$-schemes, we have isomorphisms
\begin{align*}
\Sch_{/S}(Y,T_{X/S}) &\cong \Sch_{/S}(Y, [W_S,X]) \cong \Sch_{/S}(Y \times_S W_S, X) \\
&\cong \Sch_{/S}(\Spec B \times_{R} \Spec R[\epsilon], \Spec A) \cong \Sch_{/S}\left(\Spec(B \otimes_R R[\epsilon]), \Spec A\right) \\
&\cong \Sch_{/S}\left(\Spec B[\epsilon], \Spec A\right) 
 \cong \mathbf{CAlg}_{R}(A, B[\epsilon]) \\
 &\cong \operatorname{Der}_R(A, B) \cong \RMod(\Omega^{1}_{A/R},B),
\end{align*}
so the tangent scheme functor may be thought of as a way of geometrizing the story of algebraic derivations.
\end{rmk}

With this tangent construction we now want to establish, for a smooth algebraic group $G$ and a left $G$-variety $X$, that there is an equivariant Zariski tangent structure on $X$; i.e., that the category of descent equivariant schemes on $X$ is a tangent category. In light of Theorem \ref{Thm: Pre-Equivariant Tangent Category} it suffices to prove that the pre-equivariant pseudofunctor 
\[
\begin{tikzcd}
\Sf(G)^{\op} \ar[rr]{}{\quo_{(-) \times X}} \ar[drr, swap]{}{F} & & \Var_{/K}^{\op} \ar[d]{}{\Sch_{/(-)}} \\
 & & \fCat
\end{tikzcd}
\]
defined by
\[
F(\Gamma) := \Sch_{/\XGamma}
\]
on objects and via the pullback functors
\[
F(f) := \overline{f}^{\ast}\colon \Sch_{/\XGammap} \to \Sch_{/\XGamma}
\]
on morphisms is a tangent pre-equivariant indexing functor on $X$. In particular, this involves showing four main ingredients:
\begin{enumerate}
	\item The pullback functors $\of^{\ast}$ preserve tangents in the sense that for any $Z \in \Sch_{/\XGammap}$, if $Z^{\prime} := Z \times_{\XGammap} \XGamma$ then there is an isomorphism
	\[
	T_{Z^{\prime}/\XGamma} \cong T_{Z/\XGammap} \times_{\XGammap} \XGamma;
	\]
	\item The pullback functors $\of^{\ast}$ preserve all tangent pullbacks/limits that arise in $\Sch_{/\XGammap}$.
	\item For any $f\colon \Gamma \to \Gamma^{\prime} \in \Sf(G)_0$ there is a natural isomorphism $\!\quot{T}{f}$
	\[
	\begin{tikzcd}
	\Sch_{/\XGammap} \ar[rr, bend left = 30, ""{name = U}]{}{\of^{\ast} \circ T_{-/\XGammap}} \ar[rr, bend right = 30, swap, ""{name = L}]{}{T_{-/\XGamma} \circ \of^{\ast}} & & \Sch_{/\XGamma} \ar[from = U, to = L, Rightarrow, shorten <= 4pt, shorten >= 4pt]{}{\!\quot{T}{f}}
	\end{tikzcd}
	\]
	for which the pair $(\of^{\ast},\!\quot{T}{f})$ is a strong tangent morphism.
	\item The $(T_{-/\XGamma},\!\quot{T}{f}^{-1})$ constitute a pseudonatural transformation $T\colon F \Rightarrow F$.
\end{enumerate}
We will now endeavour to show that these all hold in turn. However, we will do this at a slightly more general level: we will show that the pullback functors are always strong tangent morphisms for any morphism of schemes, as this significantly makes the theory easier to establish and follow notationally. We begin with the observation that the pullback functors do indeed preserve the tangent functors and all tangent limits that arise.
\begin{prop}\label{Prop: Pullback functor of schemes preserves tangent functor}
For any morphism $f\colon X \to Y$ of schemes and for any $Y$-scheme $Z$, there is a canonical natural isomorphism of schemes
\[
T_{f^{\ast}Z/X} = T_{Z \times_Y X/X} \cong T_{Z/Y} \times_Y X.
\]
\end{prop}
\begin{proof}
This is \cite[Equation IV.16.5.12.2]{EGA44}.
\end{proof}
In what follows when given a scheme map $f\colon X \to Y$ we define the natural isomorphism
\[
\!\quot{T}{f}\colon f^{\ast} \circ T_{-/Y} \xRightarrow{\cong} T_{-/X} \circ f^{\ast}
\]
by setting each $\!\quot{T}{f}_Z$, for $Z \to Y$ a $Y$-scheme, to be the isomorphism
\[
(f^{\ast} \circ T_{-/Y})(Z) = T_{Z/Y} \times_Y X \xrightarrow[\!\quot{T}{f}_{Z}]{\cong} T_{f^{\ast}Z/X} = (T_{-/X} \circ f^{\ast})(Z)
\]
described in Proposition \ref{Prop: Pullback functor of schemes preserves tangent functor}.
\begin{cor}\label{Cor: Affine scheme tangent iso}
	For a morphism $f\colon X \to Y$ of affine schemes and any affine $Y$-scheme $Z$, if $X \cong \Spec A$, $Y \cong \Spec C$, and $Z \cong \Spec B$ then the isomorphism $\!\quot{T_Z}{f}$ is induced by the ring isomorphism $\theta_B$
	\[
	\Sym\left(\Omega^1_{A \otimes_C B/A}\right) \xrightarrow[\theta_B]{\cong} \Sym\left(\Omega^1_{B/C}\right) \otimes_C A
	\]
	given on pure tensors by
	\[
	b \otimes a \mapsto b \otimes a, \mathrm{d}(b \otimes a) \mapsto \mathrm{d}b \otimes a.
	\]
\end{cor}

\begin{lem}\label{Lemma: Pullback functor of schemes preserves tangent limits}
For any morphism $f\colon X \to Y$ of schemes, if $Z$ is a limit in $\Sch_{/Y}$ then $f^{\ast}Z$ is a limit in $\Sch_{/X}$. In particular, $f^{\ast}$ preserves all tangent pullbacks and equalizers.
\end{lem}
\begin{proof}
This is immediate from the fact that the functor $f^{\ast}Z = Z \times_Y X$ is a pullback and hence commutes with all limits.
\end{proof}

With Proposition \ref{Prop: Pullback functor of schemes preserves tangent functor} and Lemma \ref{Lemma: Pullback functor of schemes preserves tangent limits}, in order to prove that $(f^{\ast},\!\quot{T}{f})$ is a strong tangent morphism it suffices to prove that the five commutative diagrams displayed in Definition \ref{Defn: Tangent Morphism} actually commute. We will do this by proving six technical lemmas below --- one for being able to check things affine locally and then one for each diagram of Definition \ref{Defn: Tangent Morphism} --- before concluding in Proposition \ref{Prop: Pullback functor is strong tangent morphism} that $(f^{\ast},\!\quot{T}{f})$ is always a strong tangent morphism.

Our first technical lemma establishes that for any morphism of schemes $f\colon X \to Y$ and for any $Y$-scheme $Z$, checking any of the commuting diagrams in Definition \ref{Defn: Tangent Morphism} may be done affine locally on the base and target. In particular, it suffices in all situations to establish the tangent morphism diagrams for affine schemes.
\begin{lem}\label{Lemma: Do tangent morphism stuff affine locally}
To establish that $(f^{\ast},\!\quot{T}{f})$ is a tangent morphism it suffices to assume that $f\colon X \to Y$ is a morphism $f\colon \Spec B \to \Spec A$ between affine schemes and establish the diagrams of Definition \ref{Defn: Tangent Morphism} for affine schemes over $\Spec A$.
\end{lem}
\begin{proof}
We begin by first noting for any scheme map $f\colon X \to Y$ and any scheme $Z$ over $Y$, the sheaf of $1$-differentials $\Omega_{Z/Y}^1$ is a quasi-coherent sheaf on $Z$ and is given affine-locally on $Z$ as $\Omega_{U/Y}^1$ for $U$ an affine open of $Z$. As such, when working with $\Omega_{Z/Y}^1$, provided all constructions in sight are quasi-coherent it suffices to work with sheaves on affine schemes and glue appropriately. The symmetric algebra construction $\Sym(\Omega^1_{U/Y})$ is a ring over $\Gamma(U)$ built out of a sheaf of quasi-coherent $\Ocal_U$-algebras; in particular, $\Sym(\Omega_{U/Y}^1)$ arises as a $\Gamma(S(\Omega_{U/Y}^1))$ where $S(\Omega_{U/Y}^1)$ is the sheaf of symmetric algebras on $\Omega_{U/Y}^1$ and $\Spec(S(\Omega_{U/Y}^1)) \to U$ is affine over $U$.\footnote{Because $U$ is affine, this implies that $\Spec(S(\Omega_{U/Y}^1))$ is an affine scheme. Checking global sections gives $\Gamma(S(\Omega_{U/Y}^1)) \cong \Sym(\Omega_{U/Y}^1)$. For details, see \cite[Proposition 1.3.1, Corollaire I.3.2, Definition 1.7.8]{EGA2}.} In particular, $\Sym(\Omega^1_{U/Y})$ is a quasi-coherent construction. Finally, since $\Sym$ commutes with tensor products (cf. \cite[Section 1.7.5]{EGA2}, \cite[Exercise II.5.16.e]{Hartshorne}) and since pullbacks are built affine-locally in any category of relative schemes $\Sch_{/S}$, calculating the sheaves which represent and construct $f^{\ast}T_{Z/Y}$ and $T_{Z/Y} \times_Y X$ may be done affine locally on $Z^{\prime} = Z \times_{Y} X$. However, since the set
\[
\left\lbrace U_i \times_{W_k} V_j \; \left. \right| \; U_i \to Z, V_j \to X, W_k \to Y \text{affine open}, i \in I, j \in J, k \in \Lambda \right\rbrace
\]
is an affine open cover of $Z \times_Y X$ and $S(f^{\ast}\Omega_{Z/Y}^1)$  and the sheaves of quasi-coherent $\Ocal_{Z \times_Y X}$-algebras $S(\Omega_{(Z \times_Y X)/X}^1)$ may be calculated and then glued over this cover as
\[
S\left(f^{\ast}\Omega_{U_i/{W_k}}^1\right) = S\left(f^{-1}\Omega^1_{U_i/W_k} \otimes_{f^{-1}\Ocal_{W_k}} \Ocal_{V_j}\right), \qquad S\left(\Omega^1_{(U_i \times_{W_k} V_j)/V_j}\right) = S\left(\Omega^1_{f^{\ast}U_i/W_k}\right)
\]
respectively, the lemma follows.
\end{proof}

\begin{rmk}
We thank the anonymous referee for the following argument to show Proposition \ref{Prop: Pullback functor is strong tangent morphism} below by making use of the fact that the tangent structure on $\mathbf{Aff}_{/\Spec A}$ is given by the dual tangent structure in the sense of \cite[Proposition 5.14]{GeoffRobinDiffStruct}. Recall that \cite[Proposition 5.14]{GeoffRobinDiffStruct} shows that if $(\Cscr,\Tbb)$ is a tangent category such that there is a left adjoint $L \dashv T$ for the tangent functor $T$ of $\Tbb$, then there is a corresponding tangent structure $\Lbb$ on $\Cscr^{\op}$ with tangent functor $L$. 

Observe that $\mathbf{Aff}_{/\Spec A} \simeq \Cring_{/A}^{\op} = \mathbf{CAlg}_{A}^{\op}$ for any commutative ring $A$ with unity and that if we have a morphism $f:A \to B$ of rings, then the functor $f^{\ast}:\mathbf{Aff}_{/\Spec B} \to \mathbf{Aff}_{\Spec A}$ is given as (the opposite of) the functor $(-) \otimes_A B:\mathbf{CAlg}_A \to \mathbf{CAlg}_B$. These functors all have right adjoints
\[
\begin{tikzcd}
\mathbf{CAlg}_B \ar[rr, swap, bend right = 20, ""{name = D}]{}{\operatorname{Res}_f} & & \mathbf{CAlg}_A \ar[ll, swap, bend right = 20, ""{name = U}]{}{(-) \otimes_A B}
\ar[from = U, to = D, symbol = \dashv]
\end{tikzcd}
\]
where $\operatorname{Res}_f$ is the algebra restriction of scalars functor, i.e., $\operatorname{Res}_f$ sends a commutative $B$-algebra $\nu:B \to R$ to the commutative $A$-algebra generated by pre-composing the structure map $\nu$ with $f$: $A \xrightarrow{f} B \xrightarrow{\nu} R$.

Note now that the dual numbers tangent structure $(\mathbf{CAlg}_A, \Dbb)$ which has $D(A) := A[\epsilon] \cong A[x]/(x^2)$. This tangent functor $D$ additionally has a left adjoint $T_{(-)/A}$. Observe now that we have an invertible $2$-cell
\[
\begin{tikzcd}
\mathbf{CAlg}_B \ar[r, ""{name = U}]{}{D_B} \ar[d, swap]{}{\operatorname{Res}_f}  & \mathbf{CAlg}_B \ar[d]{}{\operatorname{Res}_f} \\
\mathbf{CAlg}_A \ar[r, swap, ""{name = D}]{}{D_A} & \mathbf{CAlg}_A
\ar[from = U, to = D, Rightarrow, shorten <= 4pt, shorten >= 4pt]{}{\alpha_f}
\end{tikzcd}
\]
where $D_A$ denotes the dual numbers tangent structure on $\mathbf{CAlg}_A$ and where $D_B$ denotes the dual numbers tangent structure on $\mathbf{CAlg}_B$. The pair $(\operatorname{Res}_f, \alpha^{-1})$ is thus a strong tangent morphism from $(\mathbf{CAlg}_B, D_B)$ to the tangent category $(\mathbf{CAlg}_A,D_A)$. Upon taking mates, we arrive at the $2$-cell:
\[
\begin{tikzcd}
\mathbf{CAlg}_A^{\op} \ar[d, swap]{}{f^{\ast}} \ar[rr, ""{name = U}]{}{T_{(-)/A}} & & \mathbf{CAlg}_A^{\op} \ar[d]{}{f^{\ast}} \\
\mathbf{CAlg}_{B}^{\op} \ar[rr, swap, ""{name = D}]{}{T_{(-)/B}} & & \mathbf{CAlg}_B^{\op}
\ar[from = U, to = D, Rightarrow, shorten <= 4pt, shorten >= 4pt]{}{\hat{\alpha}_f}
\end{tikzcd}
\]
Upon recognizing that $\hat{\alpha}_f = \quot{T}{f}$ and observing that the $\alpha_f$ are pseudofunctorial in $\Cring^{\op} \simeq \mathbf{Aff}$, it follows that the functor/natural isomorphism pair $(f^{\ast},\quot{T}{f})$ is pseudofunctorial in $\Cring \simeq \mathbf{Aff}^{\op}$. Thus, upon gluing and performing Zariski descent, it follows that for any morphism of schemes $g:X \to Y$,  the pair of functors $(g^{\ast},\quot{T}{g})\colon(\Sch_{/Y},\Tbb_{\Zar{Y}}) \to (\Sch_{/X},\Tbb_{\Zar{X}})$ give strong tangent morphisms which are pseudofunctorial in $\Sch$.
\end{rmk}

We now check the five necessary technical conditions for $(f^{\ast},\!\quot{T}{f})$ to be a tangent morphism. First we will establish that $f^{\ast}$ and $\!\quot{T}{f}$ commute suitably with the bundle morphism $p\colon T \Rightarrow \id$ and the zero morphism $0\colon \id \Rightarrow T$.
\begin{lem}\label{Lemma: Pullback functor commutes with bundle map}
For any morphism of schemes $f\colon X \to Y$ and any $Y$-scheme $Z$ the diagram
\[
\xymatrix{
f^{\ast}T_{Z/Y} \ar[dr]_{(f^{\ast} \circ p)_Z} \ar[r]^-{\!\quot{T_Z}{f}} & T_{f^{\ast}Z/X} \ar[d]^{(p \ast f^{\ast})_Z} \\
 & f^{\ast}Z
}
\]
commutes.
\end{lem}
\begin{proof}
By Lemma \ref{Lemma: Do tangent morphism stuff affine locally} it suffices to prove this for affine schemes. Write $X = \Spec A, Y = \Spec C,$ and $Z = \Spec B$ in the cospan of schemes:
\[
\xymatrix{
 & \Spec B \ar[d] \\
 \Spec A \ar[r]_-{f} & \Spec C
}
\]
To show that the diagram
\[
\xymatrix{
f^{\ast}T_{B/C} \ar[r]^-{\!\quot{T_B}{f}} \ar[dr]_{f^{\ast}(p_B)} & T_{A \otimes_C B/A} \ar[d]^{p_{f^{\ast}B}}\\
 & f^{\ast}\Spec B
}
\]
commutes in $\Sch_{/X}$ it suffices to prove that the diagram
\[
\xymatrix{
\Sym\left(\Omega^1_{A \otimes_C B/A}\right) \ar[rr]^-{\theta_B} & & \Sym\left(\Omega_{B/C}\right) \otimes_C A \\
 & B \otimes_C A \ar[ur]_{q_B \otimes \id_A} \ar[ul]^{q_{A \otimes_C B}}
}
\]
commutes in $A \downarrow \Cring$. To check this we find on one hand for any pure tensor $b \otimes a$,
\[
(q_B \otimes \id_A)(b \otimes a) = b \otimes a
\]
while on the other hand
\[
(\theta_B \circ q_{B \otimes_C A})(b \otimes a) = \theta_B(b \otimes a) = b \otimes a,
\]
so the diagram indeed commutes.
\end{proof}

\begin{lem}\label{Lemma: Pullback functor commutes with zero map}
	For any morphism of schemes $f\colon X \to Y$ and any $Y$-scheme $Z$ the diagram
	\[
	\xymatrix{
		 T_{f^{\ast}Z/X}  & f^{\ast}T_{Z/Y} \ar[l]_-{\!\quot{T_Z}{f}} \\
		& f^{\ast}Z \ar[ul]^{(0 \ast f^{\ast})_Z} \ar[u]_{(f^{\ast} \circ 0)_Z}
	}
	\]
	commutes.
\end{lem}
\begin{proof}
	By Lemma \ref{Lemma: Do tangent morphism stuff affine locally} it suffices to prove this for affine schemes. Write $X = \Spec A, Y = \Spec C,$ and $Z = \Spec B$ in the cospan of schemes:
	\[
	\xymatrix{
		& \Spec B \ar[d] \\
		\Spec A \ar[r]_-{f} & \Spec C
	}
	\]
	To show that the diagram
	\[
	\xymatrix{
	T_{f^{\ast}Z/X}  & f^{\ast}T_{Z/Y} \ar[l]_-{\!\quot{T_Z}{f}} \\
	& f^{\ast}Z \ar[ul]^{(0 \ast f^{\ast})_Z} \ar[u]_{(f^{\ast} \circ 0)_Z}
	}
	\]
	commutes in $\Sch_{/X}$ it suffices to prove that the diagram
	\[
	\xymatrix{
		\Sym\left(\Omega^1_{A \otimes_C B/A}\right) \ar[dr]_{\zeta_{A \otimes_C B}} \ar[rr]^-{\theta_B} & & \Sym\left(\Omega_{B/C}\right) \otimes_C A \ar[dl]^{\zeta_B \otimes \id_A}\\
		& B \otimes_C A  
	}
	\]
	commutes in $A \downarrow \Cring$. To check this we find on one hand for any pure tensor $b \otimes a$,
	\[
	\zeta_{B \otimes_C A}(b \otimes a) = b \otimes a
	\]
	while for any differential $\mathrm{d}(b \otimes a)$,
	\[
	\zeta_{B \otimes_C A}\left(\mathrm{d}\big(b \otimes a\big)\right) = 0.
	\]
	Now we calculate that on the other hand
	\[
	\left((\zeta_B \otimes \id_A) \circ \theta_B\right)(b \otimes a) = (\zeta_B \otimes \id_A)(b \otimes a) = b \otimes a
	\]
	while
	\[
	\left((\zeta_B \otimes \id_A) \circ \theta_B\right)\big(\mathrm{d}(b \otimes a)\big) = (\zeta_B \otimes \id_A)\big(\mathrm{d}b \otimes a\big) = 0 \otimes a = 0,
	\]
	so the diagram indeed commutes.
\end{proof}

We will now show that the pullback functors $f^{\ast}$ commute suitably with the bundle addition morphisms $+\colon T_2 \Rightarrow T$.

\begin{lem}\label{Lemma: Pullback commutes with addition}
For any morphism $f\colon X \to Y$ of schemes and for any $Y$-scheme $Z$, the diagram
\[
\xymatrix{
f^{\ast}(T_{Z/Y})_2 \ar[r]^-{(\!\quot{T_{2}}{f})_Z} \ar[d]_{(f^{\ast} \ast +)_Z} & (T_{f^{\ast}Z/X})_2 \ar[d]^{(+ \ast f^{\ast})_{Z}} \\
f^{\ast}T_{Z/Y} \ar[r]_-{\!\quot{T}{f}_Z} & T_{f^{\ast}Z/Y}
}
\]
commutes.
\end{lem}
\begin{proof}
By Lemma \ref{Lemma: Do tangent morphism stuff affine locally} it suffices to prove this for affine schemes. Write $X = \Spec A, Y = \Spec C,$ and $Z = \Spec B$ in the cospan of schemes:
\[
\xymatrix{
	& \Spec B \ar[d] \\
	\Spec A \ar[r]_-{f} & \Spec C
}
\]
To prove that the diagram
\[
\xymatrix{
	f^{\ast}(T_{Z/Y})_2 \ar[r]^-{(\!\quot{T_{2}}{f})_Z} \ar[d]_{(f^{\ast} \ast +)_Z} & (T_{f^{\ast}Z/X})_2 \ar[d]^{(+ \ast f^{\ast})_{Z}} \\
	f^{\ast}T_{Z/Y} \ar[r]_-{\!\quot{T}{f}_Z} & T_{f^{\ast}Z/Y}
}
\]
commutes it suffices to prove that the diagram
\[
\xymatrix{
\Sym\left(\Omega^1_{B \otimes_C A/A}\right) \otimes_{B \otimes_C A} \Sym\left(\Omega^1_{B \otimes_C A/A}\right) \ar[rrr]^-{(\theta_2)_{B}} & & & \left(\Sym\left(\Omega^1_{B/C}\right) \otimes_B \Sym\left(\Omega_{B/C}^1\right) \right) \otimes_{C} A   \\
\Sym\left(\Omega_{B \otimes_C A/A}^1\right) \ar[rrr]_-{\theta_B}  \ar[u]^{\operatorname{add}_{B \otimes_C A}} & & & \Sym\left(\Omega_{B/C}^1\right) \otimes_C A \ar[u]_{\operatorname{add}_B \otimes \id_A}
}
\]
commutes in $A \downarrow \Cring$. Note that by construction we have that $(\theta_2)_B$ acts on a pure tensor $(b \otimes a) \otimes (b^{\prime} \otimes a^{\prime})$ by the assignment
\[
(\theta_2)_B\big((b \otimes a) \otimes (b^{\prime} \otimes a^{\prime})\big) = b \otimes b^{\prime} \otimes aa^{\prime}
\]
and for a differential $\mathrm{d}(b \otimes a) \otimes \mathrm{d}(b^{\prime} \otimes a^{\prime})$ via
\[
(\theta_2)_B\big(\mathrm{d}(b \otimes a) \otimes \mathrm{d}(b^{\prime} \otimes a^{\prime})\big) = \mathrm{d}b \otimes \mathrm{d}b^{\prime} \otimes aa^{\prime}.
\]
We begin by calculating the top half of the square above, i.e., we compute that for any pure tensor $b \otimes a$ in $B \otimes_C A$ that
\begin{align*}
\big((\theta_2)_B \circ \operatorname{add}_{B \otimes_C A}\big)(b \otimes a) &= (\theta_2)_B(b \otimes a \otimes 1_{B \otimes_C A}) = b \otimes 1_B \otimes a 
\end{align*}
while on differentials we have
\begin{align*}
\big((\theta_2)_B \circ \operatorname{add}_{B \otimes_C A}\big)\left(\mathrm{d}(b \otimes a)\right) &= (\theta_2)_B\left(\mathrm{d}(b \otimes a) \otimes 1 + 1 \otimes \mathrm{d}(b \otimes a)\right) = \left(\mathrm{d}b \otimes 1_B\right) \otimes a + \left(1_B \otimes \mathrm{d}b\right) \otimes a.
\end{align*}
Let us now compute the bottom half of the square above. Fix a pure tensor $b \otimes a$ and observe that on one hand we have
\[
\big((\operatorname{add}_B \otimes \id_A) \circ \theta_B\big)(b \otimes a) = (\operatorname{add}_B \otimes \id_A)(b \otimes a) = b \otimes 1_B \otimes a
\]
while on the other hand for differentials we get
\[
\big((\operatorname{add}_B \otimes \id_A) \circ \theta_B\big)\left(\mathrm{d}(b \otimes a)\right) = (\operatorname{add}_B \otimes \id_A)(\mathrm{d}b \otimes a) = (\mathrm{d}b \otimes 1_B) \otimes a + (1_B \otimes \mathrm{d}b) \otimes a.
\]
It follows that $(\theta_2)_B \circ \operatorname{add}_{B \otimes_C A} = (\operatorname{add}_B \otimes \id_A) \circ \theta_B$ so the diagram indeed commutes.
\end{proof}

We now establish that the pullback functor commutes suitably with the vertical lift natural transformation $\ell\colon T \Rightarrow T^2$ and the canonical flip natural transformation $c\colon T^2 \Rightarrow T^2$.

\begin{lem}\label{Lemma: Pullback commutes with vertical lift}
	For any morphism $f\colon X \to Y$ of schemes and for any $Y$-scheme $Z$, the diagram
	\[
	\xymatrix{
	f^{\ast}T_{Z/Y} \ar[r]^-{\!\quot{T}{f}_Z} \ar[d]_{(f^{\ast} \ast \ell)_{Z}} & T_{f^{\ast}Z/X} \ar[d]^{(\ell \ast f^{\ast})_{Z}} \\
	f^{\ast}T^2_{Z/Y} \ar[r]_-{\!\quot{T^2}{f}_Z} & T^2_{f^{\ast}Z/X}
	}
	\]
	commutes.
\end{lem}
\begin{proof}
	By Lemma \ref{Lemma: Do tangent morphism stuff affine locally} it suffices to prove this for affine schemes. Write $X = \Spec A, Y = \Spec C,$ and $Z = \Spec B$ in the cospan of schemes:
	\[
	\xymatrix{
		& \Spec B \ar[d] \\
		\Spec A \ar[r]_-{f} & \Spec C
	}
	\]
	To prove that the diagram
	\[
\xymatrix{
	f^{\ast}T_{Z/Y} \ar[r]^-{\!\quot{T}{f}_Z} \ar[d]_{(f^{\ast} \circ \ell)_{Z}} & T_{f^{\ast}Z/X} \ar[d]^{(\ell \circ f^{\ast})_{Z}} \\
	f^{\ast}T^2_{Z/Y} \ar[r]_-{\!\quot{T^2}{f}_Z} & T^2_{f^{\ast}Z/X}
}
\]
commutes it suffices to prove that the diagram 
\[
\xymatrix{
\Sym\left(\Omega_{\Sym(\Omega_{B \otimes_C A/A}^1)/A}^1\right) \ar[d]_{v_{B \otimes_C A}} \ar[rrr]^-{\sigma_B} & & & \Sym\left(\Omega^1_{\Sym(\Omega_{B/C}^1)/C}\right) \otimes_C A \ar[d]^{v_{B} \otimes \id_A} \\
\Sym\left(\Omega^1_{B \otimes_C A/A}\right) \ar[rrr]_-{\theta_B} & & & \Sym\left(\Omega^1_{B/C}\right) \otimes_C A
}
\]
commutes in $A \downarrow \Cring$. Note that the morphism $\sigma_B$ is the ring map corresponding to the scheme isomorphism\footnote{For potentially obvious reasons it is best to avoid the notation $\theta^2_B$.} $\!\quot{T^2_B}{f}$ and is given on generators by
\[
b \otimes a \mapsto b \otimes a, \quad \mathrm{d}(b \otimes a) \mapsto \mathrm{d}b \otimes a, \quad \delta(b \otimes a) \mapsto \delta b \otimes a, \quad \delta\mathrm{d}(b\otimes a) \mapsto \delta\mathrm{d}b \otimes a.
\]
Fix a pure tensor $b \otimes a$. We now compute the top half of the diagram on all generators:
\begin{align*}
\big((v_B \otimes \id_A) \circ \sigma_B\big)(b \otimes a) &= (v_B \otimes \id_A)(b \otimes a) = b \otimes a; \\
\big((v_B \otimes \id_A) \circ \sigma_B\big)\big(\mathrm{d}(b \otimes a)\big) &= (v_B \otimes \id_A)(\mathrm{d}b \otimes a) = 0 \otimes a = 0; \\
\big((v_B \otimes \id_A) \circ \sigma_B\big)\big(\delta(b \otimes a)\big) &= (v_B \otimes \id_A)(\delta b \otimes a) = 0 \otimes a = 0; \\
\big((v_B \otimes \id_A) \circ \sigma_B\big)\big(\delta\mathrm{d}(b \otimes a)\big) &= (v_B \otimes \id_A)(\delta\mathrm{d}b \otimes a) = \mathrm{d}b \otimes a.
\end{align*}
Alternatively, the bottom half of the diagram is computed on all generators as follows:
\begin{align*}
\big(\theta_B \circ v_{B \otimes_C A}\big)(b \otimes a) &= \theta_B(b \otimes a) = b \otimes a; \\
\big(\theta_B \circ v_{B \otimes_C A}\big)\big(\mathrm{d}(b \otimes a)\big) &= \theta_B(0) = 0; \\
\big(\theta_B \circ v_{B \otimes_C A}\big)\big(\delta(b \otimes a)\big) &= \theta_B(0) = 0; \\
\big(\theta_B \circ v_{B \otimes_C A}\big)\big(\delta\mathrm{d}(b \otimes a)\big) &= \theta_B\big(\mathrm{d}(b \otimes a)\big) = \mathrm{d}b \otimes a.
\end{align*}
Because all assignments agree on all generators, it follows that the square commutes.
\end{proof}
\begin{lem}\label{Lemma: Pullback commutes with canonical flip}
	For any morphism $f\colon X \to Y$ of schemes and for any $Y$-scheme $Z$, the diagram
	\[
	\xymatrix{
	f^{\ast}T^2_{Z/Y} \ar[r]^-{\!\quot{T^2}{f}_Z} \ar[d]_{(f^{\ast} \ast c)_{Z}} & T^2_{f^{\ast}Z/X} \ar[d]^{(c \ast f^{\ast})_{Z}} \\
	f^{\ast}T^2_{Z/Y} \ar[r]_-{\!\quot{T^2}{f}_Z} & T^2_{f^{\ast}Z/X}
}
	\]
	commutes.
\end{lem}
\begin{proof}
	By Lemma \ref{Lemma: Do tangent morphism stuff affine locally} it suffices to prove this for affine schemes. Write $X = \Spec A, Y = \Spec C,$ and $Z = \Spec B$ in the cospan of schemes:
	\[
	\xymatrix{
		& \Spec B \ar[d] \\
		\Spec A \ar[r]_-{f} & \Spec C
	}
	\]
	To prove that the diagram
		\[
	\xymatrix{
		f^{\ast}T^2_{Z/Y} \ar[r]^-{\!\quot{T^2}{f}_Z} \ar[d]_{(f^{\ast} \ast c)_{Z}} & T^2_{f^{\ast}Z/X} \ar[d]^{(c \ast f^{\ast})_{Z}} \\
		f^{\ast}T^2_{Z/Y} \ar[r]_-{\!\quot{T^2}{f}_Z} & T^2_{f^{\ast}Z/X}
	}
	\]
	commutes it suffices to prove that the diagram
	\[
	\xymatrix{
		\Sym\left(\Omega_{\Sym(\Omega_{B \otimes_C A/A}^1)/A}^1\right) \ar[d]_{\gamma_{B \otimes_C A}} \ar[rrr]^-{\sigma_B} & & & \Sym\left(\Omega^1_{\Sym(\Omega_{B/C}^1)/C}\right) \otimes_C A \ar[d]^{\gamma_{B} \otimes \id_A} \\
		\Sym\left(\Omega_{\Sym(\Omega_{B \otimes_C A/A}^1)/A}^1\right)  \ar[rrr]_-{\sigma_B} & & & \Sym\left(\Omega^1_{\Sym(\Omega_{B/C}^1)/C}\right) \otimes_C A 
	}
	\]
	commutes. To show this we calculate both paths of the square on all generators. For the top path we get
	\begin{align*}
	\big((\gamma_B \otimes \id_A) \circ \sigma_B\big)(b \otimes a) &= (\gamma_B \otimes \id_A)(b \otimes a) = b \otimes a; \\
	\big((\gamma_B \otimes \id_A) \circ \sigma_B\big)\big(\mathrm{d}(b \otimes a)\big) &= (\gamma_B \otimes \id_A)(\mathrm{d}b \otimes a) = \delta b \otimes a; \\
	\big((\gamma_B \otimes \id_A) \circ \sigma_B\big)\big(\delta(b \otimes a)\big) &= (\gamma_B \otimes \id_A)(\delta b \otimes a) = \mathrm{d}b \otimes a; \\
	\big((\gamma_B \otimes \id_A) \circ \sigma_B\big)\big(\delta\mathrm{d}(b \otimes a)\big) &= (\gamma_B \otimes \id_A)(\delta\mathrm{d}b \otimes a) = \delta\mathrm{d}b \otimes a.
	\end{align*}
	Alternatively, for the bottom path we calculate
	\begin{align*}
	\big(\gamma_B \circ \gamma_{B \otimes_C A}\big)(b \otimes a) &= \sigma_B(b \otimes a) = b \otimes a; \\
	\big(\sigma_B \circ \gamma_{B \otimes_C A}\big)\big(\mathrm{d}(b \otimes a)\big) &= \sigma_B(\delta(b \otimes a)) = \delta b \otimes a; \\
	\big(\sigma_B \circ \gamma_{B \otimes_C A}\big)\big(\delta(b \otimes a)\big) &= \theta_B(\mathrm{d}(b \otimes a)) = \mathrm{d}b \otimes a\; \\
	\big(\sigma_B \circ \gamma_{B \otimes_C A}\big)\big(\delta\mathrm{d}(b \otimes a)\big) &= \sigma_B\big(\delta\mathrm{d}(b \otimes a)\big) = \delta\mathrm{d}b \otimes a.
	\end{align*}
	Because all assignments agree on all generators, it follows that the square commutes.
\end{proof}

We now finally establish that for any scheme map $f\colon X \to Y$ the pullback functor $f^{\ast}$ and tangent commutativity isomorphism $\!\quot{T}{f}\colon f^{\ast} \circ T_{-/Y} \Rightarrow T_{-/X} \circ f^{\ast}$ constitute a strong tangent morphism.

\begin{prop}\label{Prop: Pullback functor is strong tangent morphism}
Let $f\colon X \to Y$ be a morphism of schemes and equip each category $(\Sch_{/X},\Tbb_{\Zar{X}})$ and $(\Sch_{/Y},\Tbb_{\Zar{Y}})$ with their respective Zariski tangent structures. Then the functor and natural transformation pair $(f^{\ast},\!\quot{T}{f})\colon (\Sch_{/Y},\Tbb_{\Zar{Y}}) \to (\Sch_{/X},\Tbb_{\Zar{X}})$ is a strong morphism of tangent categories.
\end{prop}
\begin{proof}
The fact that the natural transformation $\!\quot{T}{f}$ is an isomorphism is Proposition \ref{Prop: Pullback functor of schemes preserves tangent functor} and that $f^{\ast}$ preserves all tangent pullbacks and equalizers follows from Lemma \ref{Lemma: Pullback functor of schemes preserves tangent limits}. Finally, the fact that $(f^{\ast},\!\quot{T}{f})$ satisfy the commuting diagrams of Definition \ref{Defn: Tangent Morphism} follows from Lemmas \ref{Lemma: Pullback functor commutes with bundle map}, \ref{Lemma: Pullback functor commutes with zero map}, \ref{Lemma: Pullback commutes with addition}, \ref{Lemma: Pullback commutes with vertical lift}, and \ref{Lemma: Pullback commutes with canonical flip}.
\end{proof}

We now show that the information $(T_{-/X},\!\quot{T}{f}^{-1})$ forms a pseudonatural transformation $T\colon \Sch_{/-} \Rightarrow \Sch_{/-}$ of the pseudofunctor $\Sch_{/-}\colon \Sch \to \fCat$ given by 
\[
\Sch_{/-}(X) := \Sch_{/X}; \qquad \Sch_{/-}(f) := f^{\ast}\colon \Sch_{/Y} \to \Sch_{/X}, f^{\ast}Z := Z \times_Y X.
\]
on objects and morphisms. For any composable pair of morphisms $X \xrightarrow{f} Y \xrightarrow{g} Z$, the compositor natural isomorphisms $\phi_{f,g}\colon f^{\ast} \circ g^{\ast} \Rightarrow (g \circ f)^{\ast}$ are induced by the cancelation and reassociation isomorphisms
\[
\xymatrix{
f^{\ast}\left(g^{\ast}W\right) \ar@{=}[r] \ar[d]_{\phi_{f,g}^{W}}^{\text{def'n}} & f^{\ast}\left(W \times_Z Y\right) \ar@{=}[r] & (W \times_Z Y) \times_Y X \ar[d]^{\cong} \\
(g \circ f)^{\ast}W & W \times_Z X \ar@{=}[l] & \ar[l]^-{\cong} W \times_Z (Y \times_Y X)
}
\]
for any $Z$-scheme $W$. Affine-locally, if $X = \Spec A, Y = \Spec B, Z = \Spec C$, and $W = \Spec D$ this is the composition of spectra of canonical isomorphisms
\[
(D \otimes_C B) \otimes_B A \cong D \otimes_C (B \otimes_B A) \cong D \otimes_C A.
\]
In what follows, we will denote the natural isomorphism $D \otimes_C A \xrightarrow{\cong} (D \otimes_C B) \otimes_B A$ by 
\[
\omega_D^{A,B,C}\colon D \otimes_C A \xrightarrow{\cong} (D \otimes_C B) \otimes_B A;
\] 
note that $\omega_{D}^{A,B,C}(d\otimes a) = (d \otimes 1_B) \otimes a$ for any pure tensor $d \otimes a \in D \otimes_C A$.
Now to show that $(T_{-/X},\!\quot{T}{f}^{-1})$ constitute a pseudonatural transformation of $\Sch_{/-}$, because $\Sch$ is a $1$-category regarded as a $2$-category by inserting trivial $2$-cells, it suffices to prove that the pasting diagram
\[
\begin{tikzcd}
\Sch_{/Z} \ar[r, ""{name = UpLeft}]{}{g^{\ast}} \ar[d, swap]{}{T_{-/Z}} & \Sch_{/Y} \ar[d]{}{T_{-/Y}} \ar[r, ""{name = UpRight}]{}{f^{\ast}} & \Sch_{/X} \ar[d]{}{T_{/X}} \\
\Sch_{/Z} \ar[rr, bend right = 40, swap, ""{name = DownMid}]{}{(g \circ f)^{\ast}} \ar[r, swap, ""{name = DownLeft}]{}{g^{\ast}} & \Sch_{/Y} \ar[r, swap, ""{name = DownRight}]{}{f^{\ast}} & \Sch_{/X} \ar[from = UpLeft, to = DownLeft, Rightarrow, shorten <= 4pt, shorten >= 4pt]{}{\!\quot{T}{g}^{-1}} \ar[from = UpRight, to = DownRight, Rightarrow, shorten <= 4pt, shorten >= 4pt]{}{\!\quot{T}{f}^{-1}} \ar[from = 2-2, to = DownMid, Rightarrow, shorten <= 2pt, shorten >= 4pt]{}{\phi_{f,g}}
\end{tikzcd}
\] 
is equivalent to the pasting diagram:
\[
\begin{tikzcd}
 & \Sch_{/Y} \ar[dr]{}{f^{\ast}} \\
\Sch_{/Z} \ar[d, swap]{}{T_{-/Z}} \ar[ur]{}{g^{\ast}} \ar[rr, ""{name = Mid}]{}[description]{(g \circ f)^{\ast}} & & \Sch_{/X} \ar[d]{}{T_{-/X}} \\
\Sch_{-/Z} \ar[rr, swap, ""{name = Down}]{}{(g \circ f)^{\ast}} & & \Sch_{/X}
\ar[from = 1-2, to = Mid, Rightarrow, shorten <= 4pt, shorten >= 4pt]{}{\phi_{f,g}} \ar[from = Mid, to = Down, Rightarrow, shorten <= 6pt, shorten >= 4pt]{}{\!\quot{T}{g \circ f}^{-1}}
\end{tikzcd}
\]
It is worth noting that because the natural transformations $\!\quot{T}{f}_W\colon f^{\ast} \circ T_{-/Y} \Rightarrow T_{-/X} \circ f^{\ast}$ are generated affine-locally by the isomorphisms
\[
\theta_B\colon \Sym\left(\Omega_{B \otimes_C A/A}^1\right) \to \Sym\left(\Omega^1_{B/C}\right) \otimes_C A
\]
defined on pure tensors by $b \otimes a \mapsto b \otimes a$ and $\mathrm{d}(b \otimes a) \mapsto \mathrm{d}b \otimes a$, the map $\!\quot{T^{-1}}{f}$ is generated affine-locally by the inverse isomorphisms
\[
\theta_B^{-1}\colon \Sym\left(\Omega_{B/C}^1\right) \otimes_C A \to \Sym\left(\Omega^1_{B \otimes_C A/A}\right)
\]
defined by $b \otimes a \mapsto b \otimes a$ and $\mathrm{d}b \otimes a \mapsto \mathrm{d}(b \otimes a)$ on pure tensors. Furthermore, it follows from Lemma \ref{Lemma: Do tangent morphism stuff affine locally} and the fact that the isomorphisms above are all Zariski-locally defined, it even suffices to prove the equivalence of the pasting diagrams above on affine schemes using the maps $\theta^{-1}_B$.

\begin{prop}\label{Prop: Tangent functor and inverse iso is pseudonat}
For any composable pair of morphisms of schemes $X \xrightarrow{f} Y \xrightarrow{g} Z$, the pasting diagram
\[
\begin{tikzcd}
\Sch_{/Z} \ar[r, ""{name = UpLeft}]{}{g^{\ast}} \ar[d, swap]{}{T_{-/Z}} & \Sch_{/Y} \ar[d]{}{T_{-/Y}} \ar[r, ""{name = UpRight}]{}{f^{\ast}} & \Sch_{/X} \ar[d]{}{T_{/X}} \\
\Sch_{/Z} \ar[rr, bend right = 40, swap, ""{name = DownMid}]{}{(g \circ f)^{\ast}} \ar[r, swap, ""{name = DownLeft}]{}{g^{\ast}} & \Sch_{/Y} \ar[r, swap, ""{name = DownRight}]{}{f^{\ast}} & \Sch_{/X} \ar[from = UpLeft, to = DownLeft, Rightarrow, shorten <= 4pt, shorten >= 4pt]{}{\!\quot{T}{g}^{-1}} \ar[from = UpRight, to = DownRight, Rightarrow, shorten <= 4pt, shorten >= 4pt]{}{\!\quot{T}{f}^{-1}} \ar[from = 2-2, to = DownMid, Rightarrow, shorten <= 2pt, shorten >= 4pt]{}{\phi_{f,g}}
\end{tikzcd}
\] 
is equivalent to the pasting diagram:
\[
\begin{tikzcd}
& \Sch_{/Y} \ar[dr]{}{f^{\ast}} \\
\Sch_{/Z} \ar[d, swap]{}{T_{-/Z}} \ar[ur]{}{g^{\ast}} \ar[rr, ""{name = Mid}]{}[description]{(g \circ f)^{\ast}} & & \Sch_{/X} \ar[d]{}{T_{-/X}} \\
\Sch_{-/Z} \ar[rr, swap, ""{name = Down}]{}{(g \circ f)^{\ast}} & & \Sch_{/X}
\ar[from = 1-2, to = Mid, Rightarrow, shorten <= 4pt, shorten >= 4pt]{}{\phi_{f,g}} \ar[from = Mid, to = Down, Rightarrow, shorten <= 6pt, shorten >= 4pt]{}{\!\quot{T}{g \circ f}^{-1}}
\end{tikzcd}
\]
In particular, $(T_{-/X},\!\quot{T}{f}^{-1})$ describe the data of a pseudonatural transformation 
\[
T\colon \Sch_{/-} \Rightarrow \Sch_{/-}\colon \Sch \to \fCat.
\]
\end{prop}
\begin{proof}
After expanding out the pasting diagrams, we see that we are asking to show that the morphism $\alpha$ encoding the first pasting diagram defined via the diagram
\[
\begin{tikzcd}
T_{-/X} \circ f^{\ast} \circ g^{\ast} \ar[rr, Rightarrow]{}{\!\quot{T^{-1}}{f} \ast g^{\ast}} \ar[d, swap, Rightarrow]{}{\alpha} & & f^{\ast} \circ T_{-/Y} \circ g^{\ast} \ar[d, Rightarrow]{}{f^{\ast} \circ \!\quot{T^{-1}}{g}} \\
(g \circ f)^{\ast} \circ T_{-/Z} & & f^{\ast} \circ g^{\ast} \circ T_{-/Z} \ar[ll, Rightarrow]{}{\phi_{f,g} \ast T_{-/Z}}
\end{tikzcd}
\]
is equal to the morphism $\beta$ encoding the second pasting diagram defined via the diagram:
\[
\begin{tikzcd}
T_{-/X} \circ f^{\ast} \circ g^{\ast} \ar[drr, Rightarrow, swap]{}{\beta} \ar[rr, Rightarrow]{}{T_{-/X} \ast \phi_{f,g}} & & T_{-/X} \circ (g \circ f)^{\ast} \ar[d, Rightarrow]{}{\!\quot{T^{-1}}{g \circ f}} \\
 & & (g \circ f)^{\ast} \circ T_{-/Z}
\end{tikzcd}
\]
To show this we let $W$ be an arbitrary $Z$-scheme. From the discussion prior to the statement of the proposition, we know it suffices to prove the equality of $\alpha$ and $\beta$ on affine schemes. For this we assume that $X \cong \Spec A$, $Y \cong \Spec B$, $Z \cong \Spec C$, and $W \cong \Spec R$ in the diagram
\[
\xymatrix{
 & & \Spec R \ar[d] \\
\Spec A \ar[r]_-{f} & \Spec B \ar[r]_-{g} & \Spec C
}
\]
of schemes. 

We will now calculate how $\alpha$ acts on its domain. It is a straightfoward unwinding of definitions to check that $\alpha_W = \alpha_R$ is the spectrum of the composite of ring maps 
\[
\Gamma(\alpha^{\sharp}_R) =  \left(\theta_{R \otimes_C B}^{A,B}\right)^{-1} \circ \left(\left(\theta_R^{B,C}\right)^{-1} \otimes \id_A\right) \circ \omega_{\Sym(\Omega^1_{R/C})}^{A,B,C}
\] 
which is seen visually as
\[
\Sym\left(\Omega_{R/C}^1\right) \otimes_C A \xrightarrow{} \left(\Sym\left(\Omega_{R/C}^1\right) \otimes_C B\right) \otimes_B A \xrightarrow{} \left(\Sym\left(\Omega_{R \otimes_C B/B}\right)\right) \otimes_{B} A \to \Sym\left(\Omega_{(R \otimes_C B) \otimes_B A/A}\right).
\]
Fix now a pure tensor $r \otimes a$ in $\Sym(\Omega_{R/C}^1) \otimes_C A$ with $r \in R$ and $a \in A$. We then calculate that
\begin{align*}
\Gamma(\alpha_R^{\sharp})(r \otimes a) &= \left(\left(\theta_{R \otimes_C B}^{A,B}\right)^{-1} \circ \left(\left(\theta_R^{B,C}\right)^{-1} \otimes \id_A\right) \circ \omega_{\Sym(\Omega^1_{R/C})}^{A,B,C}\right)(r \otimes a) \\
&= \left(\left(\theta_{R \otimes_C B}^{A,B}\right)^{-1} \circ \left(\left(\theta_R^{B,C}\right)^{-1} \otimes \id_A\right)\right)\big((r \otimes 1_B) \otimes a\big) \\
&= (r \otimes 1_B) \otimes a
\end{align*}
while for differential pure tensors $\mathrm{d}r \otimes a$ we get that
\begin{align*}
\Gamma(\alpha_R^{\sharp})(\mathrm{d}r \otimes a) &= \left(\left(\theta_{R \otimes_C B}^{A,B}\right)^{-1} \circ \left(\left(\theta_R^{B,C}\right)^{-1} \otimes \id_A\right) \circ \omega_{\Sym(\Omega^1_{R/C})}^{A,B,C}\right)(\mathrm{d}r \otimes a) \\
&= \left(\left(\theta_{R \otimes_C B}^{A,B}\right)^{-1} \circ \left(\left(\theta_R^{B,C}\right)^{-1} \otimes \id_A\right)\right)\big((\mathrm{d}r \otimes 1_B) \otimes a\big) = \left(\theta_{R \otimes_C B}^{A,B}\right)^{-1}\big(\mathrm{d}(r \otimes 1_B) \otimes a\big) \\
&= \mathrm{d}\big((r \otimes 1_B) \otimes a\big).
\end{align*}
Because the pure tensors of the form $r \otimes a, \mathrm{d}r \otimes a$ for $r \in R$ and $a \in A$ linearly generate $\Sym(\Omega_{R/C}^1) \otimes_C A$, this completely determines the action of $\Gamma(\alpha_R^{\sharp})$ after extending linearly.

We now calculate the action of $\beta_W = \beta_R$ on its domain. By construction, $\beta_{R}$ is the composite of the ring maps
\[
\Gamma(\beta_{R}^{\sharp}) := \Sym(\Omega^1(\omega_R^{A,B,C}))\circ \left(\theta_{R}^{A,C}\right)^{-1}
\]
which is seen visually as
\[
\Sym\left(\Omega_{R/C}^1\right) \otimes_C A \xrightarrow{} \Sym\left(\Omega^1_{R \otimes_C A/A}\right) \xrightarrow{} \Sym\left(\Omega^1_{(R \otimes_C B) \otimes_B A}\right).
\]
We then compute that on pure tensors $r \otimes a$ for $r \in R$ and $a \in A$,
\begin{align*}
\Gamma(\beta_{R}^{\sharp})(r \otimes a) &= \left(\Sym(\Omega^1(\omega_R^{A,B,C}))\circ \left(\theta_{R}^{A,C}\right)^{-1}\right)(r \otimes a) = \Sym(\Omega^1(\omega_R^{A,B,C}))(r \otimes a) \\
&= (r \otimes 1_B) \otimes a.
\end{align*}
while for differentials pure tensors $\mathrm{d}r \otimes a$ we get
\begin{align*}
\Gamma(\beta_{R}^{\sharp})(\mathrm{d}r \otimes a) &= \left(\Sym(\Omega^1(\omega_R^{A,B,C}))\circ \left(\theta_{R}^{A,C}\right)^{-1}\right)(\mathrm{d}r \otimes a) = \Sym(\Omega^1(\omega_R^{A,B,C}))\big(\mathrm{d}(r \otimes a)\big) \\
&= \mathrm{d}\big((r \otimes 1_B) \otimes a\big).
\end{align*}
Because the pure tensors of the form $r \otimes a, \mathrm{d}r \otimes a$ for $r \in R$ and $a \in A$ linearly generate $\Sym(\Omega_{R/C}^1) \otimes_C A$, this completely determines the action of $\Gamma(\beta_R^{\sharp})$ after extending linearly. Furthermore, because $\Gamma(\alpha_R^{\sharp})$ and $\Gamma(\beta_R^{\sharp})$ coincide on generators, it follows that they coincide and we get that $\Gamma(\alpha_R^{\sharp}) = \Gamma(\beta_R^{\sharp})$. This implies that $\alpha_R = \beta_R$ and so the two pasting diagrams are equivalent by once again invoking Lemma \ref{Lemma: Do tangent morphism stuff affine locally}. Furthermore, because $\Sch$ is regarded as a $2$-category with trivial $2$-cells, the equivalence of these pasting diagrams implies that the data $(T_{-/X},\!\quot{T}{f}^{-1})$ defines a pseudonatural transformation $T\colon \Sch_{/-} \Rightarrow \Sch_{/-}$.
\end{proof}

With all these facts, we now have the ingredients to define and construct the (equivariant) tangent structure on the equivariant category of schemes over a base variety $X$ by way of the Zariski tangent indexing functor.

\begin{dfn}\label{Defn: Equivariant Zariski category}
Let $G$ be a smooth algebraic group and let $X$ be a (quasi-projective) left $G$-variety. Define the pre-equivariant pseudofunctor $\Scal$ to be the one given in Item (11) of Example \ref{Example: Pre-equivariant pseudofunctors}; i.e., $\Scal = (\Scal,\Sch_{/(-)})$ is the pre-equivariant pseudofunctor:
\[
\begin{tikzcd}
\Sf(G)^{\op} \ar[rr]{}{\quo_{(-) \times X}^{\op}} \ar[drr, swap]{}{\Scal} & & \Var_{/K}^{\op} \ar[d]{}{\Sch_{/(-)}} \\
 & & \fCat
\end{tikzcd}
\]
We then define the category of {\em descent equivariant schemes}\footnote{The idea here is that the category $(\Sch_{/X})_G$ describes all possible equivariant descent information as recorded by the actual schemes over $X$.} to be the equivariant category $(\Sch_{/X})_G := \Scal_G(X)$. We also define the {\em descent equivariant Zariski tangent category on $X$} to be the pseudolimit tangent category $((\Sch_{/X})_G,\Tbb)$.
\end{dfn}
The fact that the definition of the descent equivariant Zariski tangent category is actually a tangent category is justified by the theorem below.
\begin{Theorem}\label{Thm: Equivariant Zariski Tangent Structure}
Let $G$ be a smooth algebraic group and let $X$ be a (quasi-projective) left $G$-variety. Then the equivariant Zariski tangent category $((\Sch_{/X})_G,\Tbb)$ is a tangent category.
\end{Theorem}
\begin{proof}
By Theorem \ref{Thm: Pre-Equivariant Tangent Category} it suffices to show that the pre-equivariant pseudofunctor $\Sbb$ is a tangent pre-equivariant indexing functor. However:
\begin{enumerate}
	\item That each category $\Scal(\Gamma)$ is a tangent category is Theorem \ref{Thm: Zariski tangent structure};
	\item That each pair $(\Scal(f), \!\quot{T}{f}) = (\of^{\ast},\!\quot{T}{\of})$ is a strong tangent morphism is Proposition \ref{Prop: Pullback functor is strong tangent morphism};
	\item That the data $(\!\quot{T}{\Gamma},\!\quot{T}{f}^{-1}) = (T_{-/\XGamma},\!\quot{T^{-1}}{\of})$ constitutes a pseudonatural transformation $T\colon \Scal \Rightarrow \Scal$ follows from Proposition \ref{Prop: Tangent functor and inverse iso is pseudonat}.
\end{enumerate}
Thus $(\Scal_G(X), \Tbb_G)$ is a tangent category and we are done.
\end{proof}

We now conclude this section by showing just below in Propositions \ref{Prop: Scheme gluing equivalence of tan cats} and \ref{Prop: Equivariantifying the gluing equiv for schemes} that the descent equivariant Zariski tangent structure on $\Sch_{/X}$ is compatible with the gluings of schemes in the sense that:
\begin{enumerate}
    \item First, if we have a scheme $X$ with an open affine cover $\Ucal = \lbrace f_i\colon U_i \to X \; \left. \right| \; i \in I \rbrace$ then there is an equivalence of tangent categories 
    \[
    \left(\Sch_{/X},T_{\Zar{X}}\right) \simeq \operatorname{pseudolim} \left(\Sch_{/U_i}, T_{\Zar{U_i}}\right).
    \]
    \item Second, these gluing pseudolimits are compatible with the equivariant Zariski cover in the sense that if we have an open affine cover $\lbrace f_i\colon U_i \to X \; \left. \right| \; i \in I \rbrace$ by \emph{equivariant} afffine opens then there is an equivalence
    \[
    \left(\left(\Sch_{/X}\right)_G,\Tbb_X\right) \simeq \operatorname{pseudolim} \left(\left(\Sch_{/U_i}\right)_G,\Tbb_{U_i}\right)
    \]
    where each $\Tbb_{(-)}$ is the tangent structure defined in Theorem \ref{Thm: Equivariant Zariski Tangent Structure}.
\end{enumerate}
\begin{prop}\label{Prop: Scheme gluing equivalence of tan cats}
Let $X$ be a scheme and let $\Ucal = \lbrace f_i\colon U_i \to X \; \left. \right| \; i \in I \rbrace$ be an affine open cover of $X$. Then there is an equivalence of tangent categories
\[
\left(\Sch_{/X},\Tbb_{\Zar{X}}\right) \simeq \operatorname{pseudolim}\left(\Sch_{/U_i},\Tbb_{\Zar{U_i}}\right).
\]
\end{prop}
\begin{proof}
Begin by letting $\Ubb$ denote the sublattice of $\Open(X)$ generated by the open immersions $f_i\colon U_i \to X$\footnote{Alternatively, this is the same as asking to consider the sieve $\Uscr$ on $X$ in the presheaf topos $[\Open(X)^{\op},\Set]$ generated by the Zariski cover $\Ucal$ in the Zariski (Grothendieck) pretopology.}. Recall that each pullback diagram of open inclusions
\[
\begin{tikzcd}
U_i \times_{X} U_j \ar[r]{}{\operatorname{pr}_1} \ar[d, swap]{}{\operatorname{pr}_2} & U_i \ar[d]{}{f_i} \\
U_j \ar[r, swap]{}{f_j} & X
\end{tikzcd}
\]
induces an invertible $2$-cell
\[
\begin{tikzcd}
(\Sch_{/X}) \ar[rr, ""{name = U}]{}{(f_i^{\ast},\quot{T}{f_i})} \ar[d, swap]{}{(f_j^{\ast},\quot{T}{f_j})} & & \Sch_{/U_i} \ar[d]{}{(\operatorname{pr}_1^{\ast},\quot{T}{\operatorname{pr}_1})} \\
\Sch_{/U_j} \ar[rr, ""{name = D}, swap]{}{(\operatorname{pr}_2^{\ast}, \quot{T}{\operatorname{pr}_2})} & & \Sch_{/U_i \times_X U_j}
\ar[from = U, to = D, Rightarrow, shorten <= 4pt, shorten >= 4pt]{}{\gamma_{ij}}
\ar[from = U, to = D, Rightarrow, shorten <= 4pt, shorten >= 4pt, swap]{}{\cong}
\end{tikzcd}
\]
of pullback functors. Furthermore, as we have indicated above, each pullback functor is a strong tangent morphism (in the lax direction) by Proposition \ref{Prop: Pullback functor is strong tangent morphism} when we equip each scheme category with the corresponding Zariski tangent structure. A routine check using the universal properties of the various $\Sym$ functors involved in defining the tangent functors shows that the transformation $\gamma_{ij}$ is a tangent transformation. Using that the Zariski topology is subcanonical as a Grothendieck topology and that $\Ucal$ is an affine open cover of $X$ then implies that we have an equivalence of categories
\[
\Sch_{/X} \simeq \operatorname{pseudolim} \Sch_{/U_i}.
\]

Now observe that we can write the pseudolimit $\operatorname{pseudolim}\Sch_{/U_i}$ as a pseudocone category for a tangent indexing functor as follows. Consider the pseudofunctor
\[
T_{\Zar{(-)}}\colon\Ubb^{\op} \to \fCat
\]
given by sending objects $V \mapsto \Sch_{/V}$ and sending morphisms of open immersions $g\colon V \to W$ to the pullback tangent morphism
\[
(g^{\ast},\quot{T}{g})\colon\left(\Sch_{/W},\Tbb_{\Zar{W}}\right) \to \left(\Sch_{/V},\Tbb_{\Zar(V)}\right)
\]
asserted by Proposition \ref{Prop: Pullback functor is strong tangent morphism}. By Proposition \ref{Prop: Pullback functor of schemes preserves tangent functor} this determines a tangent indexing functor. Observe also that effective Zariski descent then gives an equivalence of categories
\[
\operatorname{pseudolim} \Sch_{/U_i} \simeq \PC(T_{\operatorname{Zar(-)}})
\]
and hence by composing equivalences we get
\[
\Sch_{/X} \simeq \PC(T_{\Zar{(-)}})
\]
in $\fCat$ induced by the assignment
\[
\begin{tikzcd}
Z \ar[d] \\
X
\end{tikzcd} \mapsto \begin{tikzcd}
Z \times_{U_i} X \ar[d] \\
U_i
\end{tikzcd}
\]
for all $X$-schemes $Z \to X$ and all open immersions $f_i\colon U_i \to X$ in $\Ucal$. Now, since the above assignment is induced as a pseudocone over $T_{\Zar{(-)}}$, it follows from Theorem \ref{Thm: Pseudolimits in Tancat} that $\Sch_{/X}$ equipped with the Zariski tangent structure is equivalent to the desired pseudolimit in $\fTan$.
\end{proof}
\begin{prop}\label{Prop: Equivariantifying the gluing equiv for schemes}
Let $\Ucal := \lbrace f_i\colon U_i \to X \; \left. \right| \; i \in I \rbrace$ be an affine open cover of a $K$-variety $X$ for which each map $f_i$ is $G$-equivariant for a smooth algebraic group $G$. Then there is an equivalence of tangent categories
\[
\left(\left(\Sch_{/X}\right)_G,\Tbb_{X}\right) \simeq \operatorname{pseudolim} \left(\left(\Sch_{/U_i}\right)_G, \Tbb_{U_i}\right).
\]
\end{prop}
\begin{proof}
Because the cover $\Ucal$ is $G$-equivariant, it follows from a routine argument that for any smooth free $G$-variety $\Gamma$, the set
\[
\Ucal_{\Gamma} = \left\lbrace f_i \times \id_{\Gamma}: U_i \times \Gamma \to X \times \Gamma \; \left. \right| \; i \in I\right\rbrace
\]
is a $G$-equivariant open cover of $X \times \Gamma$. As such, by Proposition \ref{Prop: Scheme gluing equivalence of tan cats} we know that for every smooth free $G$-variety $\Gamma$ we have equivalences of categories
\[
\Sch_{/X \times \Gamma} \simeq \operatorname{pseudolim} \Sch_{/U_i \times \Gamma}
\]
and hence, after taking $G$-quotients,
\[
\Sch_{/\quot{X}{\Gamma}} \simeq \operatorname{pseudolim} \Sch_{/\quot{U_i}{\Gamma}}
\]
for all open immersions $f_i\colon U_i \to X$. Furthermore, these equivalences are pseudonatural in $\Sf(G)$ and so we can proceed as in the proof of Proposition \ref{Prop: Scheme gluing equivalence of tan cats} by using Theorem \ref{Theorem: Defining functors and nat transforms between equivariant categories and functors} to produce functors between the our categories.
\end{proof}

\section{A Sketch of the Descent Equivariant Tangent Structure on Smooth Manifolds}\label{Section: Descent Equivariant Manfiold stuffs}

Perhaps unsurprisingly, we can play a similar game in the case of smooth manifolds as what we did with varieties and produce a tangent category of descent equivariant smooth manifolds. However, in order to define such a tangent category over a given left $L$-manifold $M$, we need, analogously to the scheme case, the result that the pre-equivariant pseudofunctor $\Mcal = (\Mcal, \SMan_{/(-)})$
\[
\begin{tikzcd}
\FMan(L)^{\op} \ar[rr]{}{\quo_{(-) \times M}^{\op}} \ar[drr, swap]{}{\Mcal} & & \SMan_{\operatorname{sub}}^{\op} \ar[d]{}{\SMan_{/(-)}} \\
 & & \fCat
\end{tikzcd}
\]
defines a tangent indexing functor. In this shorter section we sketch how to go about this by first recalling some general tangent category theory regarding tangent display morphisms, relative/vertical tangent structures, and how the two concepts interact. To motivate this discussion, let us pose the following question. If $f\colon X \to Y$ is a morphism in a tangent category $(\Cscr,\Tbb)$ for which all pullbacks against $f$ exist and are preserved by all powers $T^m$ for $m \in \N$, then what tangent structures can we equip each of $\Cscr_{/X}$ and $\Cscr_{/Y}$ with to make the functor $f^{\ast}\colon\Cscr_{/Y} \to \Cscr_{/X}$ into a strong tangent morphism?

Recall that when given a tangent category $\Cscr$, there are two main techniques for building tangent structures on the slice categories $\Cscr_{/X}$. The first is the direct approach as described by \cite[Proposition 2.5]{GeoffRobinDiffStruct} and \cite{Rosicky}: simply define the tangent functor on the slice category $T\colon\Cscr_{/X} \to \Cscr_{/X}$ by the assignment:
\[
\begin{tikzcd}
Z \ar[d]{}{h} \\
X
\end{tikzcd} \mapsto \begin{tikzcd}
TZ \ar[d]{}{} \ar[r]{}{Th} & TX \ar[dl]{}{p_X} \\
 X
\end{tikzcd}
\]
The other family of tangent structures on the slice categories $\Cscr_{/X}$ can be defined less often (in the sense that the construction required need not always exist), but is the one we will need to consider in order to make the pullback functor $f^{\ast}$ a strong tangent functor when $f$ is a surjective submersion. These are the tangent structures on $\Cscr_{/X}$ defined by use of what is called the relative tangent bundle of a map $f\colon X \to Y$ in \cite{GeoffJS}.

Relative tangent bundles were originally discussed and discovered by Rosick{\'y} in \cite[Pages 5, 6]{Rosicky} in his original paper on tangent categories. There the relative tangent bundle functors $T_{(-)/Y}$ are ways of building tangent structures on the categories $\Cscr_{/Y}$ which are not simply induced by applying the ``global'' tangent structure on $\Cscr$ to the ``local-to-$Y$'' category $\Cscr_{/Y}$ by selecting only the tangent vectors which live ``in the direction directly above $Y$.'' Rosick{\'y} did this by considering equalizers of the form:
\[
\begin{tikzcd}
V_{f}(X) \ar[r]{}{\iota_f} & TX \ar[rr, shift left = 1]{}{Tf} \ar[rr, shift right = 1, swap]{}{0_Y \circ p_Y \circ Tf} & & TY
\end{tikzcd}
\]
However, we use the following alternative construction of \cite{GeoffJS} to construct relative tangent bundles.
\begin{dfn}[{\cite{GeoffJS}}]\label{Defn: Relative tangent bundle}
Let $f\colon X \to Y$ be a morphism in a tangent category $\Cscr$. The \emph{relative tangent bundle of $X$ over $Y$}, $V_{f}(X)$, is the pullback
\[
\begin{tikzcd}
V_{f}(X) \ar[r]{}{\iota_f} \ar[d, swap]{}{\iota^f} & Y \ar[d]{}{0_Y} \\
TX \ar[r, swap]{}{Tf} & TY
\end{tikzcd}
\]
if it exists.
\end{dfn}
\begin{rmk}
The existence and basic properties of the relative bundle $V_{f}(X)$ are studied in detail in the preprint \cite{LemayVooysMorphisms} (in which a morphism $f$ is called \emph{$0$-carrable} precisely when $V_{f}(X)$ exists). There Lemay and the second author of this paper showed that whenever $V_{f}(X)$ exists and whenever the pullback
\[
\begin{tikzcd}
f^{\ast}(TY) \ar[r] \ar[d] & TY \ar[d]{}{p_Y} \\
X \ar[r, swap]{}{f} & Y
\end{tikzcd}
\]
exists, then $V_{f}(X)$ is the kernel of the unique map $\theta_f\colon TX \to f^{\ast}(TY)$ in the category $\mathbf{DBun}(X)$ of differential bundles over $X$ equipped with linear maps between them. 
\end{rmk}
\begin{example}
In $\SMan$ if $f\colon X \to Y$ is a morphism of smooth manifolds then the relative tangent bundle $V_{f}(X)$ is given by
\[
V_{f}(X) := \left\lbrace \left(x, \overrightarrow{v}\right) \in TX \; | \; D[f](x)\overrightarrow{v} = \overrightarrow{0}\right\rbrace.
\]
In particular, every morphism $f\colon X \to Y$ admits a relative tangent bundle.
\end{example}
\begin{example}
In $\mathbf{CAlg}_{R}^{\op}$ for a base (commutative) ring $R$, if $A$ is a commutative $R$-algebra then $V_{\varphi}(A) = \Sym{A}(\Kah{R}{A})$; i.e., the relative tangent bundle of $A$ over $R$ is the $R$-algebra with structure map
\[
R \to A \to \Sym_{A}\left(\Kah{A}{R}\right).
\]
\end{example}

When a tangent category $\Cscr$ has the property that every map is $0$-carrable then by \cite[Pages 5, 6]{Rosicky} the functor $T_{(-)/Y}\colon\Cscr_{/Y} \to \Cscr_{/Y}$ is the tangent functor for what we call the relative tangent structure on $\Cscr_{/Y}$. It has bundle projection $p_{X/Y}$
\[
\begin{tikzcd}
V_{f}(X) \ar[r]{}{\pi_0} \ar[dr, swap]{}{p_{X/Y}} & TX \ar[d]{}{p_X} \\
 & X
\end{tikzcd}
\]
and zero section given by the unique map $0_{X/Y}\colon X \to V_{f}(X)$
\[
\begin{tikzcd}
    X \ar[dr, dashed]{}{0_{X/Y}} \ar[drr, bend left = 20]{}{f} \ar[ddr, swap, bend right = 20]{}{0_X} & & \\
     & V_{f}(X) \ar[r]{}{\pi_1} \ar[d, swap]{}{\pi_0} & Y \ar[d]{}{0_Y} \\
     & TX \ar[r, swap]{}{Tf} & TY
\end{tikzcd}
\]
factoring the naturality square of the zero transformation. 

What is remarkable about the relative tangent bundles is that they are \emph{very} well behaved with respect to pullback against tangent display morphisms. Let us recall the definition of tangent display maps (as incarnated in \cite{BenThesis} and \cite{GeoffMarcelloTSubmersionPaper}) and then prove a helpful technical proposition. 
\begin{dfn}[{cf.\! \cite{BenThesis}, \cite{GeoffMarcelloTSubmersionPaper}}]
Let $\Cscr$ be a tangent category. A morphism $f\colon X \to Y$ is said to be a \emph{tangent display morphism} (alterntively \emph{$T$-display}) if for any map $g\colon Z \to Y$ the pullback
    \[
    \begin{tikzcd}
    X \times_Y Z \ar[r] \ar[d] & Z \ar[d]{}{g} \\
    X \ar[r, swap]{}{f} & Y
    \end{tikzcd}
    \]
    exists and is preserved by all powers $T^m$ of the tangent functor for $m \in \N$
\end{dfn}
\begin{example}
In $\mathbf{CAlg}_{R}^{\op}$, for $R$ a commutative rig, every morphism is $T$-display because the tangent functor is continuous (as it is a right adjoint) and because the category$\mathbf{CAlg}_{R}$ is both complete and cocomplete.
\end{example}
\begin{example}
In $\SMan$ a morphism $f\colon X \to Y$ is $T$-display if and only if $f$ is a submersion by \cite[Theorem 2.31]{GeoffMarcelloTSubmersionPaper}.
\end{example}

We now show that $T$-display morphisms behave particularly well with respect to the relative tangent structures on slice categories. We will use this to ultimately deduce that the pullback functors $f^{\ast}\colon\SMan_{/N} \to \SMan_{/M}$ are strong tangent morphisms for submersions $f\colon M \to N$ when each category is equipped with its relative/vertical tangent structure.

\begin{prop}\label{Prop: TDisplay pullback commutes with relative bundles}
Let $\Cscr$ be a tangent category in which every object $S$ has a relative tangent bundle functor $V_{(-)}\colon\Cscr_{/S} \to \Cscr_{/S}$ and let $f\colon X \to Y$ be a $T$-display morphism. Then if $f^{\ast}\colon\Cscr_{/Y} \to \Cscr_{/X}$ is the functor given by $Z \mapsto X \times_Y Z$, there is a natural isomorphism
\[
f^{\ast} \circ \quot{V_{(-)}}{Y} \cong \quot{V_{(-)}}{X} \circ f^{\ast}.
\]
\end{prop}
\begin{proof}
We begin by observing that since $f$ is $T$-display, for any $g:Z \to Y$ with corresponding pullback square
\[
\begin{tikzcd}
f^{\ast}Z \ar[r]{}{\pi_1} \ar[d, swap]{}{\pi_0} & Z \ar[d]{}{g} \\
X \ar[r, swap]{}{f} & Y
\end{tikzcd}
\]
then the diagram
\[
\begin{tikzcd}
T(f^{\ast}Z) \ar[r]{}{T\pi_1} \ar[d, swap]{}{T\pi_0} & TZ \ar[d]{}{Tg} \\
TX \ar[r, swap]{}{Tf} & TY
\end{tikzcd}
\]
is a pullback square as well. Now, on one hand we observe that $T_{f^{\ast}Z/X}$ is computed via the pullback square
\[
\begin{tikzcd}
V_{\pi_0}(f^{\ast}Z) \ar[r]{}{\iota^{\pi_0}} \ar[d, swap]{}{\iota_{\pi_0}}& X \ar[d]{}{0_X} \\
T(f^{\ast}Z) \ar[r, swap]{}{T\pi_0} & TX
\end{tikzcd}
\]
so pasting this  square with the diagram describing that $T$ preserves the pullback $f^{\ast}Z$ and using the Pullback Lemma allows us to deduce that $V_{\pi_0}(f^{\ast}Z)$ also arises from the pullback diagram:
\[
\begin{tikzcd}
V_{\pi_0}(f^{\ast}Z) \ar[r]{}{\iota_{\pi_0}} \ar[d, swap]{}{T\pi_1 \circ \iota^{\pi_0}} & X \ar[d]{}{Tf \circ 0_X} \\
TZ \ar[r, swap]{}{Tg} & TY
\end{tikzcd}
\]
That is, $V_{\pi_0}(f^{\ast}Z)$ arises as the pullback of $Tf \circ 0_X:X \to TY$ against $Tg:TZ \to TY$.

On the other hand, consider that the object $f^{\ast}T_{Z/Y}$ arises by pasting the pullback squares
\[
\begin{tikzcd}
f^{\ast}(V_{g}(Z)) \ar[rr]{}{\id_X \times \iota^{\pi_0}} \ar[d, swap]{}{\id_X \times \iota_{\pi_0}} & & X \times_Y Y \ar[d]{}{\id_X \times_Y 0_Y} \\
f^{\ast}(TZ) \ar[rr]{}{\id_X \times Tg} \ar[d, swap]{}{\pi_1} & & f^{\ast}(TY) \ar[d]{}{\pi_1} \\
TZ \ar[rr, swap]{}{Tg} & & TY
\end{tikzcd}
\]
together. Since $X \times_Y Y \cong X$, we can replace $X \times_Y Y$ with $X$. Furthermore, $f^{\ast}(TY)$ is a differential bundle over $X$ with the property that if we write $\theta_f$ for the unique map rendering
\[
\begin{tikzcd}
TX \ar[drr, bend left = 20]{}{Tf} \ar[ddr, swap, bend right = 20]{}{p_X} \ar[dr, dashed]{}{\exists!\theta_f} \\
 & f^{\ast}(TY) \ar[r]{}{\pi_1} \ar[d, swap]{}{\pi_0} & TY \ar[d]{}{p_Y} \\
 & X \ar[r, swap]{}{f} & Y
\end{tikzcd}
\]
commutative, then the diagram
\[
\begin{tikzcd}
X \ar[r]{}{0_X} \ar[d, swap]{}{\id_X \times 0_Y} & TX \ar[dl]{}{\theta_f} \ar[d]{}{Tf} \\
f^{\ast}(TY) \ar[r, swap]{}{\pi_1} & TY
\end{tikzcd}
\]
commutes as well. But now it follows that $\pi_1 \circ (\id_X \times 0_Y) = \pi_1 \circ \theta_f \circ 0_X = Tf \circ 0_X$ and hence that the diagram
\[
\begin{tikzcd}
X \times_Y V_{g}(Z) \ar[dr, swap]{}{\theta_f} \ar[r]{}{\pi_0} & X \ar[r]{}{0_X} \ar[d]{}{\id \times 0_Y}  & TX \ar[d]{}{Tf} \\
 & f^{\ast}(TY) \ar[r, swap]{}{\pi_1} & TY\end{tikzcd}
\]
commutes. Putting this together we find that $X \times_Y V_{g}(Z) = f^{\ast}(V_{g}(Z))$ is the pullback of $Tf \circ 0_X\colon X \to TY$ against $Tg\colon TZ \to TY$ and hence that $f^{\ast}(V_{g}(Z)) \cong V_{\pi_0}(f^{\ast}Z)$.
\end{proof}
\begin{cor}\label{Cor: The etale corollary we need}
In $\SMan$, if $f\colon X \to Y$ is a submersion then for any map $g\colon Z \to Y$ there is an isomorphism $V_{\pi_0}(f^{\ast}(Z)) \cong f^{\ast}(V_{g}(Z)).$
\end{cor}

With Proposition \ref{Prop: TDisplay pullback commutes with relative bundles} in mind, it is relatively straightforward to show that when every morphism in a tangent category admits a relative tangent bundle then $T$-display morphisms have pullback functors as part of strong tangent morphisms. It is worth noting that this appears implicitly in \cite[Section 3.3]{GeoffMarcelloTSubmersionPaper} when the authors discuss the universal property of slicing tangent categories.
\begin{prop}\label{Prop: Pullback functors as strong tangent morphisms}
Let $\Cscr$ be a tangent category such that every morphism $g\colon W \to Z \in \Cscr_1$ has a pullback square:
\[
\begin{tikzcd}
T_{W/Z} \ar[r]{}{\iota^{g}} \ar[d, swap]{}{\iota_g} & Z \ar[d]{}{0_Z} \\
TW \ar[r, swap]{}{Tg} & TZ
\end{tikzcd}
\]
If $f\colon X \to Y$ is a $T$-display morphism in a tangent category $\Cscr$ and if both $X$ and $Y$ admit relative tangent bundles then the functor $f^{\ast}\colon\Cscr_{/Y} \to \Cscr_{/X}$ is a strong tangent morphism with distributor $\alpha$ the natural isomorphism $T_{f^{\ast}Z/X} \cong f^{\ast}T_{Z/Y}$ for all object $Z$ over $Y$.
\end{prop}

We now apply the various general results above to show that for any Lie group $L$ satisfying Assumption \ref{Assume on Lie groups} and for any left $L$-manifold $M$, there is a ``sliced-over-$\SMan$'' pseudofunctor
\[
\Mcal\colon\FMan(L)^{\op} \to \fCat
\]
which is a tangent indexing functor. Because of this, however, we will have to introduce slice categories with notationally complicated terms which contain many slash and backslash characters. In order to make sure our categories are as readable as possible, we have adopted the comma category notation $\Cscr \downarrow X$ in order to indicate the slice category $\Cscr_{/X}$ of $\Cscr$ over $X$.

Recall that by Part (4) of Proposition \ref{Prop: Obsrevations about free L manifolds}, for each morphism $f\colon F \to E$ in the category $\FMan(L)$, the corresponding morphism $\overline{f}\colon L \backslash (F \times M) \to L \backslash (E \times M)$ is a surjective submersion in $\SMan$. As such, because submersions are $T$-display in $\SMan$, we have a functor
\[
\begin{tikzcd}
\overline{f}^{\ast}\colon \SMan\downarrow{\big(L\backslash(E \times M)\big)} \to \SMan\downarrow{\big(L \backslash(F \times M)\big)}
\end{tikzcd}
\]
for every morphism in $\overline{f}$ (which, in turn, is the technology which allows us to equip $\FMan(L)$ with the pseudofunctor $\Mcal$ described above). Our goal is to establish that this pseudofunctor is a tangent indexing functor when we regard each category $\SMan \downarrow X$ as a tangent category with its relative tangent structure.

To establish this, we once again observe that by Part (4) of Proposition \ref{Prop: Obsrevations about free L manifolds} we know that $\overline{f}$ is a surjective submersion. Thus by Corollary \ref{Cor: The etale corollary we need} there is an isomorphism, for any smooth manifold $\varphi\colon N \to L \backslash (E \times M)$,
\[
\left(\!\quot{\alpha}{f}\right)_{N}\colon V_{\pi_0}\left(N \times_{L \backslash (E \times M)} \big(L \backslash (F \times M)\big)\right) \xrightarrow{\cong} V_{\varphi}(N) \times_{L \backslash (E \times M)} \big(L\backslash (F \times M)\big).
\]
That this ismorphism is natural in the structure map $\varphi\colon N \to L \backslash (E \times M)$ is trivial to verify by virtue of the universal property of pullbacks. We claim that the pairs
\[
(\overline{f}^{\ast},\!\quot{\alpha}{f})\colon \SMan\downarrow{\big(L \backslash (E \times M)\big)} \longrightarrow \SMan\downarrow{\big(L \backslash (F \times M)\big)}
\]
are strong tangent morphisms when each of the corresponding slice categories are given the relative tangent structures.
\begin{prop}\label{Prop: Pullback functor for slice SMAN are smooth tangent morphisms}
For any Lie group $L$ satisfying Assumption \ref{Assume on Lie groups}, for any left $L$-manifold $M$, and for any morphism $f\colon F \to E$ in $\FMan(L)$, the pair
\[
(\overline{f}^{\ast},\!\quot{\alpha}{f})\colon \SMan\downarrow{\big(L \backslash (E \times M)\big)} \longrightarrow \SMan\downarrow{\big(L \backslash (F \times M)\big)}
\]
constitutes a strong tangent morphism.
\end{prop}
\begin{proof}
Observe that because the functor $\overline{f}^{\ast}$ is continuous, it restricts to a functor
\[
\overline{f}^{\ast}\colon \mathbf{CMon}\left(\SMan\downarrow{\big(L \backslash (E \times M)\big)}\right) \longrightarrow \mathbf{CMon}\left(\SMan\downarrow{\big(L \backslash (F \times M)\big)}\right)
\]
by virtue of preserving the relevant diagrams describing commutative monoids; moreover, by the fact that each map $\overline{f}$ is a display map $\overline{f}^{\ast}$ preserves the additivity of the tangent bundles and interacts appropriately with the bundle transformations $p, 0,$ and $\operatorname{add}$ via $\!\quot{\alpha}{f}$ and its pullback power $(\!\quot{\alpha}{f})_2$. Similarly, $\overline{f}^{\ast}$ and $\!\quot{\alpha}{f}$ interact appropriately with the lifts $\lambda$ by virtue of the fact that $\overline{f}$ is a tangent display map. The preservation of the canonical flip $c$ comes down to verifying that the diagram
\[
\begin{tikzcd}
L \backslash (F \times M) \times_{L \backslash (E \times M)} V_{\varphi}^2(N) \ar[d, swap]{}{\overline{f}^{\ast} \ast c} \ar[rrr]{}{(V_{(-)} \ast \!\quot{\alpha}{f}) \circ (\!\quot{\alpha}{f} \ast V_{(-)})} & & & V_{\pi_0}^2\left(L \backslash (F \times M) \times_{L \backslash (E \times M)} N\right) \ar[d]{}{c \ast \overline{f}^{\ast}} \\
L \backslash (F \times M) \times_{L \backslash (E \times M)} V_{\varphi}^2(N) \ar[rrr, swap]{}{(T \ast \!\quot{\alpha}{f}) \circ (\!\quot{\alpha}{f} \ast T)} & & & V_{\pi_0}^2\left(L \backslash (F \times M) \times_{L \backslash (E \times M)} N\right)
\end{tikzcd}
\]
commutes, but this is a straightforward albeit tedious check using that each path around the square describes a cone morphism between the two isomorphic copies of the pullback.
\end{proof}

From here an extremely tedious but straightforward check shows that functor/transformation pairs $(V_{(-)},\!\quot{\alpha}{f}^{-1})$ vary pseudofunctorially in $\FMan(L)^{\op}$; this essentially uses the fact that pullbacks vary pseudofunctorially and that the relative tangent functors themselves vary pseudonaturally over surjective submersions. Applying Theorem \ref{Thm: Pre-Equivariant Tangent Category} shows that the category $\PC(\Mcal) = \Mcal_L(M)$ is a tangent category, and hence (by definition) a tangent category of descent equivariant smooth manifolds over the manifold $M$. We collect these observations into a definition and theorem justifying the name below.
\begin{dfn}\label{Defn: Descent Equivariant Manfiolds}
Let $L$ be a Lie group satisfying Assumption \ref{Assume on Lie groups} and let $M$ be a left $L$-manifold. Then we define the category of {\em descent $L$-equivariant smooth manifolds over $M$} to be the category
\[
\Mcal_{L}(M) =: \left(\SMan_{/M}\right)_L
\]
and we call the pseudolimit tangent structure the {\em descent $L$-equivariant tangent structure over $M$}.
\end{dfn}
\begin{Theorem}\label{Thm: Tangent Indexing Functor}
The descent $L$-equivariant tangent structure on $M$ claimed in Definition \ref{Defn: Descent Equivariant Manfiolds} exists.
\end{Theorem}
\begin{proof}[Sketch]
    In light of Theorem \ref{Thm: Pre-Equivariant Tangent Category} we must argue that the pseudofunctor $\Mcal$ is a tangent-indexing functor. However, this is argued by Proposition \ref{Prop: Pullback functor for slice SMAN are smooth tangent morphisms} and the discussion prior to the statement of Definition \ref{Defn: Descent Equivariant Manfiolds}.
\end{proof}

We now conclude this section by showing that the manifold versions of Propositions \ref{Prop: Scheme gluing equivalence of tan cats} and \ref{Prop: Equivariantifying the gluing equiv for schemes} hold as well. The reason this works is not surprising: the open subspace Grothendieck topology on a manifold is subcanonical.
\begin{prop}\label{Prop: Manifold gluing equivalence of tan cats}
Let $M$ be a smooth manifold and let $\Ucal = \lbrace f_i\colon U_i \to M \; \left. \right| \; i \in I \rbrace$ be a chart in a smooth atlas for $M$. Then there is an equivalence of tangent categories
\[
\left(\SMan_{/M},\Tbb_{\mathbf{Smooth}/M}\right) \simeq \operatorname{pseudolim}\left(\SMan_{/U_i},\Tbb_{\mathbf{Smooth}/U_i}\right).
\]
\end{prop}
\begin{proof}
Using that the open cover Grothendieck topology on $\SMan$ is subcanonical, this result follows mutatis mutandis to Proposition \ref{Prop: Scheme gluing equivalence of tan cats} once we realize that injective local diffeomorphisms $f_i:U_i \to M$ are $T$-display morphisms in $\SMan$.
\end{proof}
\begin{prop}\label{Prop: Equivariantifying the gluing equiv for manifolds}
Let $\Ucal := \lbrace f_i\colon U_i \to M \; \left. \right| \; i \in I \rbrace$ be a chart in a smooth atlas for a manifold $M$ for which each map $f_i$ is $L$-equivariant for a Lie group satisfying Assumption \ref{Assume on Lie groups}. Then there is an equivalence of tangent categories
\[
\left(\left(\SMan_{/M}\right)_L,\Tbb_{M}\right) \simeq \operatorname{pseudolim} \left(\left(\SMan_{/U_i}\right)_L, \Tbb_{U_i}\right).
\]
\end{prop}
\begin{proof}
This follows mutatis mutandis to the proof of Proposition \ref{Prop: Equivariantifying the gluing equiv for schemes}.
\end{proof}

\section{Remarks Toward Future Work}\label{Section: Towards Generalizations}
In the last section of this paper we showed how to take the tangent categorical language of smooth manifolds and have it interact with the equivariant descent recorded by having a Lie group $L$ act on a smooth manifold $M$. We are especially interested in examining this descent-equivariant tangent structure when $M$ is an orbifold regarded as a $L$-space for some Lie group $L$ (each orbifold can be represented this way; cf. \cite{Pardon}). In order to show that this is well-behaved for orbifolds, i.e. independent of the chosen presentation, we need to prove that different Morita equivalent representations of an orbifold give rise to suitably equivalent descent-equivariant (tangent) categories. This is left as a topic for future work, as it will involve delicate and technical arguments involving change of group functors and manipulations of bibundles. An alternative approach towards studying tangent structures on orbifolds would be take Proposition \ref{Prop: Equivariantifying the gluing equiv for manifolds} as a point of departure and generalize to orbifold charts and atlases.

As a final comment, we also close with the following observation.
\begin{prop}
If $F\colon \Cscr^{\op} \to \fCat$ is a pseudofunctor where each category $F(X)$ for $X \in \Cscr_0$ is a category with finite biproducts and each functor $F(f)$ for $f \in \Cscr_1$ preserves finite biproducts, then $F(X)$ is a category finite biproducts as well. In particular, if we regard each category $F(X)$ with its finite biproduct tangent structure, then $F$ is a tangent indexing functor and the induced tangent structure on $\PC(F)$ corresponds with the finite biproduct tangent structure on $\PC(F)$.
\end{prop}
This shows that in particular the pseudolimit in $\fTan$ of a diagram $F\colon \Cscr^{\op} \to \fTan_{\operatorname{strong}}$ for which each object category $F(X)$ is a CDC is itself a CDC. It would be of interest to see how the pseudolimit structure may be used in the study of CDCs.

\bibliography{ETCBib}
\bibliographystyle{amsalpha}

\end{document}